\documentclass{article}
\usepackage[
bookmarks=true,
bookmarksnumbered=true,
colorlinks=true,
linkcolor=blue,
citecolor=blue,
filecolor=blue,
urlcolor=blue
]{hyperref}
\usepackage{amsmath,amssymb,amsthm,array,enumerate,setspace,verbatim}
\usepackage{tabularx}
\usepackage{graphicx}
\usepackage{color}
\usepackage{tikz}
\usetikzlibrary{cd} %tikz用可換図式のパッケージ
\usepackage{authblk}%複数人の著者が同じ所属の場合に対応
\usepackage{makeidx}
\usepackage[top=25truemm,bottom=25truemm,left=27.5truemm,right=27.5truemm,textwidth=155truemm,textheight=247truemm]{geometry}
\makeatletter
	
	\@addtoreset{equation}{section}

	\@addtoreset{figure}{section}

	\@addtoreset{table}{section}
\makeatother

\theoremstyle{definition}
\newtheorem{thm}{Theorem}[section]
\newtheorem{defi}[thm]{Definition}
\newtheorem{example}[thm]{Example}

\newtheorem{prop}[thm]{Proposition}
\newtheorem{lemm}[thm]{Lemma}
\newtheorem{cor}[thm]{Corollary}

\newcounter{ftcount}
\renewcommand{\thefootnote}{(*\arabic{footnote})}%

\DeclareMathOperator{\Z}{\mathbb Z}
\DeclareMathOperator{\Q}{\mathbb Q}

\DeclareMathOperator{\PP}{\mathbb P}
\DeclareMathOperator{\T}{\mathbb T}

\DeclareMathOperator{\SL}{SL}

\DeclareMathOperator{\quid}{\rm quid}
\DeclareMathOperator{\cell}{\rm cell}
\DeclareMathOperator{\per}{\rm per}
\DeclareMathOperator{\Farey}{\rm Farey}
\DeclareMathOperator{\CCF}{\rm CCF}

\DeclareMathOperator{\md}{\rm md}
\DeclareMathOperator{\ad}{\rm ad}
\DeclareMathOperator{\rec}{\rm rec}
\DeclareMathOperator{\tri}{\rm tri}
\DeclareMathOperator{\Diag}{\rm Diag}

\newcommand{\maru}[1]{\text{\textcircled{\scriptsize #1}}}

\usepackage{mathrsfs}

\newenvironment{sumipmatrix}{\def\arraystretch{.8}\biggl(\begin{matrix}}{\end{matrix}\biggr)\def\arraystretch{1}}

\makeatletter
\def\section{\@startsection {section}{1}{\z@}{%
-3.5ex plus -1ex minus -.2ex}{2.3ex plus .2ex}{\large\bf}}

\def\subsection{\@startsection {subsection}{1}{\z@}{%
-3.5ex plus -1ex minus -.2ex}{2.3ex plus .2ex}{\normalsize\bf}}
\makeatother
\renewcommand{\arraystretch}{1}
\title{
Degree shifts between $q$-deformed friezes and $q$-Farey labelings 
\\
for general triangulations
}
\author{
\makebox[\textwidth][c]{
Manabu Inuma$^{1}$
\quad
Ren Motomiya$^{2}$
\quad
Takeyoshi Kogiso$^{1}$
}
}
\date{\today}

\begin{document}
\maketitle
\begingroup
\renewcommand{\thefootnote}{\arabic{footnote}}
\footnotetext[1]{Department of Mathematics, Graduate School of Science, Josai University, 
Keyakidai 1-1, Sakado, Saitama 350-0295, Japan. 
\texttt{inuma@josai.ac.jp, kogiso@josai.ac.jp}}
\footnotetext[2]{
Rakuten Mobile, Inc., Rakuten Crimson House, 1-14-1 Tamagawa, Setagaya-ku, Tokyo 158-0094, Japan. 
\\
\texttt{ren.motomiya@rakuten.com}}
\endgroup
\begin{abstract}
Morier-Genoud and Ovsienko introduced $q$-deformations of
continued fractions, Farey labelings, and Conway--Coxeter friezes,
and established relationships among them 
in restricted settings associated with triangulations having exactly two exterior cells.
In this paper, we extend these correspondences to arbitrary subsequences 
of quiddities arising from general triangulations. 
We show that the numerator and denominator polynomials 
of $q$-deformed continued fractions coincide with entries 
of $q$-deformed Conway--Coxeter friezes, while the corresponding 
polynomials in $q$-Farey labelings agree with them 
up to explicit powers of $q$. 
These powers are described combinatorially 
in terms of the number of diagonals in the triangulation, or equivalently, 
the number of entries equal to $1$ in the associated frieze. 
Furthermore, we determine the minimum and maximum degrees of these polynomials 
in terms of the same combinatorial data.
\end{abstract}
%\newpage
%\tableofcontents
%\newpage
\section{Introduction}
Frieze patterns provide a bridge between combinatorics, geometry, 
and algebra through their connections with polygon triangulations 
and continued fractions. 
Despite these developments, a unified description of the relations among 
$q$-deformed continued fractions, $q$-Farey labelings, and $q$-deformed friezes 
for general triangulations has remained incomplete.
Frieze patterns were introduced by Coxeter \cite{coxeter1971frieze}. 
Conway and Coxeter \cite{conway1973triangulated} showed 
that friezes of finite width are 
characterized by triangulations of polygons: 
a frieze with $n-1$ rows of positive integers corresponds 
to a triangulation of an $n$-gon, and its initial row coincides with the associated quiddity 
(see Theorem \ref{def:quiddity_CCF}). 
In their honor, such friezes are called Conway--Coxeter friezes (CCFs). 
\par
The main goal of this paper is to clarify the correspondence among 
$q$-deformed continued fractions, $q$-Farey labelings, 
and $q$-deformed Conway--Coxeter friezes for arbitrary polygon triangulations. 
Previous results of Morier-Genoud and Ovsienko established such relations 
only for special triangulations with two exterior cells. 
We extend these correspondences to arbitrary subsequences of quiddities 
arising from general triangulations. 
\par
These correspondences are closely related to cluster algebras and Jones polynomials. 
Fomin and Zelevinsky~\cite{fomin2002cluster, Fomin2002ClusterAI} 
introduced cluster algebras, a class of commutative algebras with deep connections 
to diverse areas such as Lie theory, Poisson geometry, Teichm\"{u}ller theory, mathematical physics, representation theory, and combinatorics. 
Assem, Reutenauer, and Smith~\cite{assem2010friezes} showed that 
cluster variables of type $A$ cluster algebras are encoded by CCFs associated with triangulations. 
Since then, numerous works have explored relationships between CCFs and cluster algebras. 
In this paper, we focus in particular on the connection with Jones polynomials. 
Fomin and Zelevinsky~\cite[Theorem 3.7, Corollary 6.3]{fomin2007cluster} 
introduced $F$-polynomials $F[a_1,a_2,\dots,a_m]$ associated with sequences 
of positive integers, and proved that cluster variables can be expressed 
in terms of these polynomials together with monomials determined by $g$-vectors%
\footnote{They showed that each cluster variable admits a separation formula: 
the numerator is a Laurent polynomial in the initial cluster variables and coefficients, 
while the denominator is a monomial in the coefficient semifield, 
expressed using the auxiliary addition $\oplus$.}.
{\c{C}}anak{\c{c}}{\i} and Schiffler \cite[Theorem 6.3, Corollary 6.6]{ccanakcci2018cluster}
, as well as Rabideau \cite[Theorem 3.4]{rabideau2018f}, showed that 
these $F$-polynomials arise as numerators of 
continued fractions of Laurent polynomials. 
Building on this, Lee and Schiffler \cite[Proposition 6.4]{lee2019cluster} proved that  
suitable specializations $F_{a_1,a_2,\dots,a_m}$ 
%of the $F$-polynomials with $y_1=t^{-2}=q^2$ and $y_2=y_3=\cdots=y_n=-t^{-1}=q$ 
coincide with the numerators of $q$-deformed continued fractions. 
Furthermore, they showed \cite[Theorem 7.1, Theorem 7.4]{lee2019cluster} 
that these specializations agree with the Jones polynomials $V_{[a_1,a_2,\dots,a_m]}$ 
of $2$-bridge knots. 
\par
The theory of $q$-deformed continued fractions has been further developed 
by Morier-Genoud and Ovsienko \cite{morier2020q, morier2021quantum}. 
In \cite{morier2020q}, they introduced $q$-deformed Farey labelings (also called weighted Farey labelings) 
and $q$-deformed rationals, and established fundamental properties. 
They also observed \cite[Appendix]{morier2020q} that 
Jones polynomials coincide with the numerators of $q$-rationals. 
In \cite{morier2021quantum}, they introduced $q$-deformed friezes 
(i.e., $q$-deformed CCFs) and studied their relations with $q$-deformed rationals. 
There are also numerous related works on Jones polynomials and $q$-deformed rationals, 
including 
\cite{bapat2023q, kogiso2025arithmetic,kogiso2019bridge, lee2019cluster, nagai2020cluster}. 
%[R22(b)] X. Ren, On $q$-deformed Farey sum and a homological interpretation of $q$-deformed real quadratic irrational numbers, Preprint (2022), arXiv:2210.06056. 
Beyond knot theory, $q$-deformed rationals appear in various contexts 
such as Teichm\"{u}ller theory \cite{fock1999quantum}, 
$2$-Calabi--Yau categories \cite{bapat2023q}, 
Markov--Hurwitz approximation theory \cite{kogiso2020q, labbe2022q, leclere2024radius,ren2022radiuses,ren2023corrigendum}, 
modular and Picard groups \cite{leclere2021q,ovsienko2021towards}, 
and combinatorics on fence posets \cite{mcconville2021rank, ouguz2023rank, ouguz2025oriented}. 
\par
Most closely related to the present work, 
Morier-Genoud and Ovsienko \cite[Proposition 4.3 (i)]{morier2020q} 
showed that certain $q$-deformed continued fractions 
coincide with rational functions arising from $q$-deformed Farey labelings. 
However, their result applies only to special subsequences 
of quiddities coming from triangulations with exactly two exterior cells. 
Similarly, in \cite[Proposition 10]{morier2021quantum}, 
they related numerators and denominators of $q$-deformed rationals 
to $q$-deformed CCFs only in this restricted setting. 
\par
We extend these results to arbitrary subsequences of quiddities 
arising from general triangulations. 
In Theorem \ref{prop:q_CCF_weighted_Farey}, 
we show that the numerator and denominator polynomials of 
$q$-deformed continued fractions coincide with entries of 
$q$-deformed Conway--Coxeter friezes, while the corresponding polynomials  
in $q$-Farey labelings agree with them up to explicit powers of $q$. 
These powers measure the discrepancy between 
$q$-deformed Conway--Coxeter friezes and $q$-Farey labelings, 
and are determined combinatorially by the number of diagonals 
in the associated triangulation, or equivalently, 
the number of occurrences of the entry $1$ in the associated frieze. 
Furthermore, in Proposition \ref{prop:q_CCF_min_deg}, 
we determine the minimum and maximum degrees of these polynomials 
using the same combinatorial quantities. 
\par
These results provide a unified framework connecting 
$q$-deformed continued fractions, $q$-Farey labelings, 
and $q$-deformed Conway--Coxeter friezes for general triangulations. 
\section{Triangulations of polygons, Farey labelings, Conway--Coxeter friezes, and continued fractions}\label{sect:Delta_Farey_CCF}
\subsection{Triangulations of polygons}
Let $n\geq 3$ be an integer. Let ${\mathcal P}$ be a convex $n$-gon, the vertices of ${\mathcal P}$ 
be numbered clockwise as $v_0,v_1,\dots,v_{n-1}$, and $V$ be the set of vertices of ${\mathcal P}$:  
${V}=\{v_0,v_{1},\ldots,v_{n-2},v_{n-1}\}$. 
For any $i\in\Z$, let $k=i\pmod n$ with $0\leq k\leq n-1$ and define $v_i=v_k$. In this way, the numbering of vertices is extended to $\Z$. An edge or a diagonal of ${\mathcal P}$ 
with endpoints $v_\alpha$ and $v_\beta$ in $V$ is denoted by 
$v_\alpha v_\beta$.  Let ${\mathcal E}$ be the set of edges of ${\mathcal P}$: 
${\mathcal E}=\{v_{0}v_{1},v_{1}v_{2},\dots,v_{n-2}v_{n-1},v_{n-1}v_{0}\}$. 
Let $\T$ be a triangulation of ${\mathcal P}$ by $n-3$ non-crossing diagonals. 
Let $\Diag_{\T}$ be the set of diagonals in $\T$. 
We call both edges of ${\mathcal P}$ and diagonals in $\T$ sides of $\T$ and 
denote their set by $E_{\T}$: $E_{\T}={\mathcal E}\cup \Diag_{\T}$. Vertices $v_{\alpha}$ and $v_{\beta}$ are said to be adjacent if $v_{\alpha}v_{\beta}\in E_{\T}$. 
For any subsets $V_1$ and $V_2$ of $V$, define  
$E_{\T}(V_1,V_2)=\{v_\alpha v_\beta\in E\mid 
v_\alpha\in {V}_1,\;v_{\beta}\in  {V}_2\}$ and $\Diag_{\T}(V_1,V_2)=\{v_\alpha v_\beta\in \Diag_{\T}\mid 
v_\alpha\in {V}_1,\;v_{\beta}\in  {V}_2\}$. If ${V}_1={V}_2$, 
we abbreviate $E_{\T}(V_1,V_2)$ to $E_{\T}(V_1)$ and $\Diag_{\T}(V_1,V_2)$ to $\Diag_{\T}(V_1)$. 
Let $\delta(V_1,V_2)$ and  $\delta(V_1)$ denote the cardinalities of $\Diag_{\T}(V_1,V_2)$ 
and $\Diag_{\T}(V_1)$, respectively; that is, 
$\delta(V_1,V_2)=\#\Diag_{\T}(V_1,V_2)$ and  $\delta(V_1)=\#\Diag_{\T}(V_1)$, 
where $\#X$ denotes the cardinality of a set $X$. 
We introduce some terminology concerning triangulations of polygons as in Conley and Ovsienko \cite[Section 6.1]{conley2023quiddities}. 
\begin{defi}[Cell] 
The triangles in $\T$ are called cells. 
As shown in Figure \ref{fig:cell}, we write a cell as ${\mathcal C}=v_{\alpha}v_{\beta}v_{\gamma}$ 
by arranging the vertices clockwise, where cyclic permutations of the vertices represent the same cell. 
\end{defi}
\par
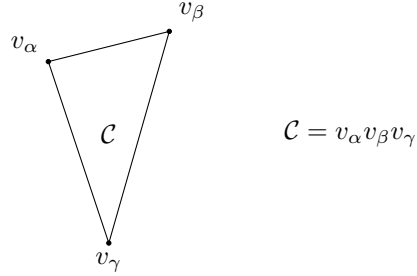
\begin{figure}[htbp]
\begin{center}
\begin{tikzpicture}[scale=.8]
% 点の定義
\coordinate(v1)at(0,0);
\coordinate(v2)at(1,-3);
\coordinate(v3)at(2,0.5);
\filldraw [black] (v1) circle (1pt) node[above left]{$v_{\alpha}$};
\filldraw [black] (v2) circle (1pt) node[below]{$v_{\gamma}$};
\filldraw [black] (v3) circle (1pt) node[above right]{$v_{\beta}$};
\draw (v1)--(v2)--(v3)--(v1);
\node at (1,-1.2) {$\mathcal C$};
\node at (5,-1.2) {${\mathcal C}=v_{\alpha}v_{\beta}v_{\gamma}$};
\end{tikzpicture}
\caption{Cell with vertices $v_{\alpha},v_{\beta},v_{\gamma}$}\label{fig:cell}
\end{center}
\end{figure}
\begin{defi}[Quiddity]\label{def:quiddity} 
For a triangulation $\T$ of an $n$-gon and a vertex $v_\alpha$, 
a cell ${\mathcal C}$ is said to be adjacent to $v_\alpha$ if $v_\alpha$ is a vertex of ${\mathcal C}$. 
Let $\cell(v_\alpha)$ be the set of such cells, and write $c_\alpha=\#\cell(v_{\alpha})$. 
We call a sequence $(c_\alpha)_{\alpha\in\Z}$ a quiddity of $\T$, and define $\quid(\T)=(c_\alpha)_{\alpha\in\Z}$. 
Note that $\quid(\T)$ is a cyclic sequence of period $n$; namely, $c_i=c_{i+n}$ for any $i\in\Z$. 
\end{defi}
\par
\begin{defi}[Exterior cell]\label{def:exterior_cell} 
A cell is called an exterior cell if more than two of its sides are edges of ${\mathcal P}$.  
\end{defi}
\begin{prop}\label{prop:quiddity}\quad
\begin{enumerate}[{(}1{)}]
\item For any triangulation $\T$ with quiddity $\quid(\T)=(c_u)_{u\in\Z}$, 
we have $\sum\limits_{u=0}^{n-1}c_u=3(n-2)$.  
\item If $n\geq 4$, then there exist at least two exterior cells.  
\end{enumerate}
\end{prop}
\begin{proof} (1) When we compute the sum $\sum\limits_{u=0}^{n-1}c_u$, 
each cell ${\mathcal C}=v_{\alpha}v_{\beta}v_{\gamma}$ is counted three times, once at each of $c_{\alpha}, c_{\beta}$, and $c_{\gamma}$. Thus, we have $\sum\limits_{u=0}^{n-1}c_u=3(n-2)$. 
\par
\noindent
(2) Let ${\mathcal C}_1,\;\dots,\;{\mathcal C}_{n-2}$ be the $n-2$ cells in $\T$. 
For each ${\mathcal C}_u$, let $t_u$ be the number of sides of ${\mathcal C}_u$ that are also edges 
of ${\mathcal P}$. Because every edge of ${\mathcal P}$ is a side of exactly one cell, 
we have $\sum\limits_{u=1}^{n-2}t_u=n$. 

First, assume that there exists no exterior cell. Then $t_u\leq 1$ for all $1\leq u\leq n-2$, and therefore we have 
$\sum\limits_{u=1}^{n-2}t_u\leq n-2$. This is a contradiction. 
Next, consider the case where there exists exactly one exterior cell ${\mathcal C}_{u_0}$. 
Assume that $t_{u_0}=2$. Since $t_u\leq 1\;(u\neq u_0)$, we have 
$\sum\limits_{u=1}^{n-2}t_u\leq 2+n-3=n-1$, which is a contradiction. 
entries we have $t_{u_0}=3$, which implies that $n=3$. 
Hence, if $n\geq4$, then there exists at least two exterior cells. 
\end{proof}
\vskip.2cm
\begin{defi}[Base edge, base cell] 
Fix a triangulation $\T$ and an integer $i\in\Z$. 
An edge $v_{i-1}v_i\in{\mathcal E}$ is called a base edge. 
A cell containing the base edge is called a base cell. 
We often denote a triangulation by $(\T,i)$ to indicate the base edge.
\end{defi}
\begin{defi}[Base side] 
For a non-base cell ${\mathcal C}$, there exists a unique side 
$v_{\alpha}v_{\beta}$ with the following properties: 
\begin{enumerate}[{(}1{)}]
\item $v_{\alpha}v_{\beta}\in\Diag_{\T}$,  
\item $v_{\alpha}v_{\beta}$ divides the $n$-gon ${\mathcal P}$ into two sub-polygons, one of which contains the base cell and the other contains 
${\mathcal C}$.  
\end{enumerate}
We call the side $v_{\alpha}v_{\beta}$ the base side of ${\mathcal C}$. 
\end{defi}
In the remainder of this paper, if the base side of ${\mathcal C}$ is $v_{\alpha}v_{\beta}$, then we write 
${\mathcal C}=v_{\alpha}v_{\beta}v_{\gamma}$ so that the vertices are arranged clockwise starting from the base side. 
\begin{defi}[Parent cell, child cell] 
For a non-base cell ${\mathcal C}$, the unique cell ${\mathcal C}'$ that shares the base side of ${\mathcal C}$ is called 
the parent  cell of ${\mathcal C}$; conversely ${\mathcal C}$ is called a child cell of ${\mathcal C}'$. Ancestor cells and descendant cells are defined analogously. 
\end{defi}
\begin{defi}[Level] 
If a cell ${\mathcal C}$ has $L$ ancestors, then it is said to be of level $L$. 
\end{defi}

\begin{defi}[Descendant degree]\label{def:desc_deg} 
Fix a triangulation $(\T,i)$ and a vertex $v_j$. We define the descendant degree 
$\mu_{i,j}$ of $v_j$ as follows. 
The $c_j$ cells adjacent to $v_j$ are numbered counterclockwise around $v_j$ as shown in Figure \ref{fig:occ1_descendant}: 
$\cell(v_j)=\{{\mathcal C}_0,\;{\mathcal C}_1,\;\dots,\;{\mathcal C}_{c_j-1}\}$. 
If ${\mathcal C}_{\mu}$ is the cell of minimum level among these cells, then 
$v_j$ is said to be of descendant degree $\mu$, which is denoted by 
$\mu_{i,j}$. 
\end{defi}
\noindent
For integers $i,j\in\Z$ with $i\leq j\leq i+n-1$, let $V[i,j-2]$ be the subset of $V$ defined by  $V[i,j-2]=\{v_\alpha\in V\mid i\leq \alpha\leq j-2\}$. 
Note that $\mu_{i,j}$ is the number of diagonals $v_{\alpha}v_{j}$ 
with $v_{\alpha}\in V[i,j-2]$; that is 
\begin{align}\label{another_def:descendant}
\mu_{i,j}=\#{\Diag_{\T}}(V[i,j-2],\{v_j\})=\delta(V[i,j-2],\{v_j\})\;, 
\end{align}
where $V[i,j-2]$ is regarded as the empty set if $j=i$ or $i+1$. 
In particular, the descendant degrees of $v_i,v_{i+1}$, and $v_{i+n-1}$ are 
$\mu_{i,i}=0,\mu_{i,i+1}=0$, and $\mu_{i,i+n-1}=c_{i-1}-1$, respectively. 
In Figure \ref{fig:occ1_descendant}, we illustrate a triangulation 
$(\T,i)$ for some $j\neq i,i+n-1$. 
\par
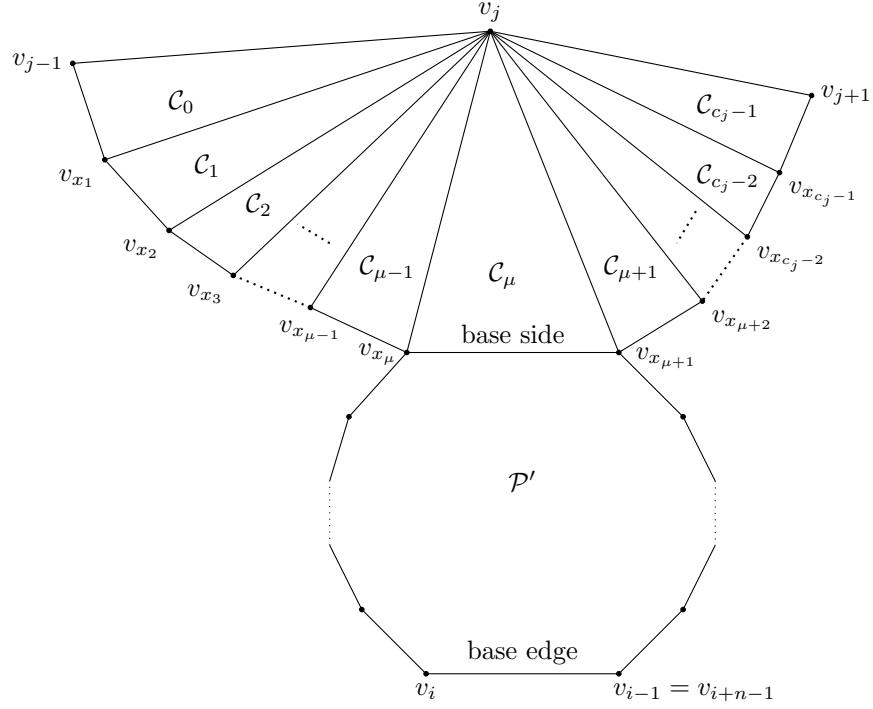
\begin{figure}[htbp]
\begin{center}
\begin{tikzpicture}[scale=.85]
\coordinate(v0)at(0,0);
\coordinate(v1)at(-3,0);
\coordinate(v2)at(-4,1);
\coordinate(v3)at(-4.5,2);
\coordinate(v4)at(-4.5,3);
\coordinate(v5)at(-4.2,4);
\coordinate(v6)at(-3.3,5);
\coordinate(v6_1)at(-4.8,5.7);
\coordinate(v6_2)at(-6,6.2);
\coordinate(v7)at(-7,6.9);
\coordinate(v8)at(-8,8);
\coordinate(v9)at(-8.5,9.5);
\coordinate(v10)at(-2,10);
\coordinate(v11)at(3,9);
\coordinate(v12)at(2.5,7.8);
\coordinate(v12_1) at(2,6.8);
\coordinate(v12_2) at(1.3,5.8);
\coordinate(v13)at(0,5);
\coordinate(v14)at(1,4);
\coordinate(v15)at(1.5,3);
\coordinate(v16)at(1.5,2);
\coordinate(v17)at(1,1);
\filldraw [black] (v0) circle (1pt) node[below right]{\hspace{-.2cm}$v_{i-1}=v_{i+n-1}$};
\filldraw [black] (v1) circle (1pt) node[below]{$v_i$};
\filldraw [black] (v2) circle (1pt);
\filldraw [black] (v5) circle (1pt);
\filldraw [black] (v6) circle (1pt) node[left]{\raisebox{-.2cm}{$v_{x_{\mu}}$}\hspace{.1cm}};
\filldraw [black] (v6_1) circle (1pt) node[below]{\raisebox{-.2cm}{$v_{x_{\mu-1}}$}\hspace{.1cm}};
\filldraw [black] (v6_2) circle (1pt) node[below left]{$v_{x_3}$};
\filldraw [black] (v7) circle (1pt) node[below left]{$v_{x_2}$};
\filldraw [black] (v8) circle (1pt) node[below left]{$v_{x_1}$};
\filldraw [black] (v9) circle (1pt) node[left]{$v_{j-1}$};
\filldraw [black] (v10) circle (1pt) node[above]{$v_{j}$};
\filldraw [black] (v11) circle (1pt) node[right]{$v_{j+1}$};
\filldraw [black] (v12) circle (1pt) node[below right]{$v_{x_{c_j-1}}$};
\filldraw [black] (v12_1) circle (1pt) node[below right]{$v_{x_{c_j-2}}$};
\filldraw [black] (v12_2) circle (1pt) node[below right]{$v_{x_{\mu+2}}$};
\filldraw [black] (v13) circle (1pt) node[right]{\hspace{.1cm}\raisebox{-.3cm}{$v_{x_{\mu+1}}$}};
\filldraw [black] (v14) circle (1pt);
\filldraw [black] (v17) circle (1pt);
\draw (v0)--(v1) node[midway, above]{base edge};
\draw (v1)--(v2);
\draw (v2)--(v3);
\draw[dotted] (v3)--(v4);
\draw (v4)--(v5);
\draw (v5)--(v6);
\draw (v6)--(v6_1);
\draw[thick, dotted] (v6_1)--(v6_2);
\draw[thick, dotted] (-4.5,6.7)--(-5,7);
\draw (v6_2)--(v7);
\draw (v7)--(v8);
\draw (v8)--(v9);
\draw (v6)--(v10);
\draw (v6_1)--(v10);
\draw (v6_2)--(v10);
\draw (v7)--(v10);
\draw (v8)--(v10);
\draw (v9)--(v10);
\draw (v10)--(v11);
\draw (v10)--(v12);
\draw (v10)--(v12_1);
\draw (v10)--(v12_2);
\draw (v10)--(v13);
\draw (v11)--(v12);
\draw (v12)--(v12_1);
\draw[thick, dotted] (v12_1)--(v12_2);
\draw[thick, dotted] (1.2,7.2)--(.9,6.7);
\draw (v12_2)--(v13);
\draw (v13)--(v14);
\draw (v14)--(v15);
\draw[dotted]  (v15)--(v16);
\draw (v16)--(v17);
\draw (v17)--(v0);
\draw (v6)--(v13)node[midway, above]{base side};
\node at(-1.5,3) {${\mathcal P}'$};
\node at(1.7,8.8) {${\mathcal C}_{c_j-1}$};
\node at(1.7,7.7) {${\mathcal C}_{c_j-2}$};
\node at(0.2,6.3) {${\mathcal C}_{\mu+1}$};
\node at(-1.8,6.2) {${\mathcal C}_{\mu}$};
\node at(-3.6,6.3) {${\mathcal C}_{\mu-1}$};
\node at(-5.6,7.3) {${\mathcal C}_2$};
\node at(-6.4,7.9) {${\mathcal C}_1$};
\node at(-6.8,8.9) {${\mathcal C}_0$};
\end{tikzpicture}
\caption{Descendant degree of $v_j$ 
for a triangulation $(\T,i)$ ($j\neq i,\,i+n-1$)}\label{fig:occ1_descendant}
\end{center}
\end{figure}

\newpage
\subsection{Farey labelings}
\begin{defi}[Farey sum]\label{def:Farey_sum} 
Let $\PP^1(\Q)$ be the union of the set of rational numbers $\Q$ and the point at infinity 
$\dfrac{1}{0}$: $\Q\cup\left\{\dfrac{1}{0}\right\}$.  
We call elements of $\PP^1(\Q)$ rational numbers. 
For two rational numbers $\rho_{\alpha}=\dfrac{r_\alpha}{s_\alpha},
\rho_\beta=\dfrac{r_\beta}{s_\beta}\in\PP^1(\Q)$, the Farey sum $\rho_\alpha\oplus \rho_\beta\in\PP^1(\Q)$ is defined by
\begin{align*}
\rho_\alpha\oplus \rho_\beta=\dfrac{r_\alpha+r_\beta}{s_\alpha+s_\beta}\;. 
\end{align*}
\end{defi}
\begin{defi}[Farey labeling {\cite[Definition 2.7]{morier2019farey}}]\label{def:Farey_graph}
Fix a triangulation $(\T,i)$. 
Label each vertex $v_j$, $i\leq j\leq i+n-1$, 
by a rational number $\rho_{i,j}\in\PP^1(\Q)$ as follows. 
First, the labels $\rho_{i,i+n-1}$ and $\rho_{i,i}$ 
of two vertices $v_{i+n-1}$ and $v_i$, the endpoints of the base edge $v_{i+n-1}v_i$, 
are defined by $\rho_{i,i+n-1}=\dfrac{0}{1}$ and $\rho_{i,i}=\dfrac{1}{0}$, respectively. 
Next, the label $\rho_{i,u}$ of the remaining vertex $v_u$ in the base cell 
is defined as the Farey sum of $\rho_{i,i+n-1}$ and $\rho_{i,i}$:  
\begin{align*}
\rho_{i,u}=\rho_{i,i+n-1}\oplus \rho_{i,i}=\dfrac{0+1}{1+0}=\dfrac{1}{1}\;.
\end{align*}
Suppose that the vertices of a cell have been labeled by rational numbers. Then 
the two vertices $v_{\alpha},v_{\beta}$, the endpoints 
of the base side of its child cell, 
have labels $\rho_{i,\alpha}=\dfrac{r_{i,\alpha}}{s_{i,\alpha}},
\rho_{i,\beta}=\dfrac{r_{i,\beta}}{s_{i,\beta}}$, respectively. 
The label $\rho_{i,\gamma}=\dfrac{r_{i,\gamma}}{s_{i,\gamma}}$ 
of the remaining vertex $v_{\gamma}$ of the child cell is defined by 
\begin{align*}
\rho_{i,\gamma}=\dfrac{r_{i,\gamma}}{s_{i,\gamma}}
=\rho_{i,\alpha}\oplus \rho_{i,\beta}=\dfrac{r_{i,\alpha}+r_{i,\beta}}{s_{i,\alpha}+s_{i,\beta}}\quad 
(r_{i,\gamma}=r_{i,\alpha}+r_{i,\beta}\,,\;s_{i,\gamma}=s_{i,\alpha}+s_{i,\beta})\;.
\end{align*}
Note that the rational 
$\rho_{i,\gamma}=\dfrac{r_{i,\gamma}}{s_{i,\gamma}}$ is irreducible. 
Indeed, if $r_{i,\alpha}s_{i,\beta}-s_{i,\alpha}r_{i,\beta}=1$, then 
$r_{i,\alpha}s_{i,\gamma}-s_{i,\alpha}r_{i,\gamma}=1$ and $r_{i,\gamma}s_{i,\beta}-s_{i,\gamma}r_{i,\beta}=1$. 
The sequence $(\rho_{i,j})_{i\leq j\leq i+n-1}$ of rational numbers is called the 
Farey labeling of $(\T,i)$ and is denoted by $\Farey(\T,i)$.  
\par
We extend the Farey labeling $\Farey(\T,i)=(\rho_{i,j})_{i\leq j\leq i+n-1}$ anti-periodically 
as follows. Its anti-periodic extension is an infinite sequence $(\rho_{i,u})_{u\in\Z}$ defined by 
\begin{align}
\rho_{i,u}=\dfrac{r_{i,u}}{s_{i,u}}=\dfrac{(-1)^{t}r_{i,j}}{(-1)^ts_{i,j}}\label{eqn:Farey_graph_infty}
\end{align}
for any $u\in\Z$ with $u=j+tn$ and $i\leq j\leq i+n-1$. 
We denote this sequence by $\Farey(\T,i)^{\per}$. 
\end{defi}
\subsection{Conway--Coxeter friezes}
Let $D\subset \Z\times \Z$ be defined by 
$D=\{(k,\ell)\in\Z\times \Z\mid k\leq \ell\}$.  
In this section, we define the notion of friezes following Coxeter \cite{coxeter1971frieze}. 
Our indexing $(k,\ell)\in D$ for the entries $\sigma_{k,\ell}$ 
differs from that of the entries $c_{a,b}$ 
used by Morier-Genoud and Ovsienko \cite{morier2021quantum}. 
The correspondence is given by $\sigma_{k,\ell}=c_{k+1,\ell-1}$ for any $(k,\ell)\in D$. 
This notation makes the statements concerning the degrees 
of the entries $\sigma_{k,\ell}$ 
in Theorem \ref{prop:q_CCF_weighted_Farey} and Proposition \ref{prop:q_CCF_min_deg} 
easier to understand.  
\begin{defi}[Frieze \cite{coxeter1971frieze}]\label{def:CCF} 
Let $(c_u)_{u\in\Z}$ be a cyclic sequence of positive integers with period $n$. 
For each $(k,\ell)\in D$, we assign an integer $\sigma_{k,\ell}\in\Z$ as follows. 
First, for $(k,\ell)\in D$ with $\ell-k=0,\;1,\;2$, define $\sigma_{k,\ell}$ by 
\begin{align}\label{eqn:CCF0}
\sigma_{k,\ell}=
\left\{
\begin{array}{cl}
0& \text{if $\ell-k=0$}, \\
1& \text{if $\ell-k=1$}, \\
c_{k+1}& \text{if $\ell-k=2$\;.}
\end{array}
\right. 
\end{align}
Moreover, the integers $\sigma_{k,\ell}$ for $(k,\ell)\in D$ with $\ell-k\geq3$ 
are recursively defined so that they satisfy the following unimodular rule:  
\begin{align}\label{eqn:unimodular}
\sigma_{k,\ell-1}\sigma_{k+1,\ell}-\sigma_{k+1,\ell-1}\sigma_{k,\ell}&=1
\end{align}
We call the array $\{\sigma_{k,\ell}\}_{(k,\ell)\in D}$ of these integers 
the Conway--Coxeter frieze, abbreviated as CCF, for $(c_u)_{u\in\Z}$. 
\end{defi}

It is not immediate that $\sigma_{k,\ell}\in\mathbb{Z}$ is well-defined from 
\eqref{eqn:unimodular}. 
As in Theorem \ref{thm:CCF}, Conway and Coxeter \cite[Problem 28, 29]{conway1973triangulated,conway1973triangulated_conti} 
established the well-definedness of CCFs for quiddities of triangulations of polygons. 
Since we work only with such CCFs, we omit the details in the general case. 
\if0
%%%%%% 
It is not immediate that $\sigma_{k,\ell}\in\Z$ is well-defined. 
In $(\ref{eqn:unimodular})$, if $\sigma_{k+1,\ell-1}=0$, then 
$\sigma_{k,\ell}$ is not defined. Furthermore, even if 
$\sigma_{k+1,\ell-1}\neq0$, it is not trivial from the above definition 
whether $\sigma_{k,\ell-1}\sigma_{k+1,\ell}-1$ is divisible by 
$\sigma_{k+1,\ell-1}$. 
As in Theorem \ref{thm:CCF}, 
Conway and Coxeter 
\cite[Problem 28, 29]{conway1973triangulated,conway1973triangulated_conti} 
clarified the well-definedness of CCFs for the quiddity $(c_u)_{u\in\Z}$ 
of a triangulation $\T$ of an $n$-gon. 
Since we work only with such CCFs in this paper, we omit the details of the 
well-definedness for the general case.
\fi
\vskip.5cm

We arrange the entries $\sigma_{k,\ell}$ in the CCF $\{\sigma_{k,\ell}\}_{(k,\ell)\in D}$ 
on a plane so that the entries $\sigma_{k,\ell}$ with $\ell-k=m,0\leq m\leq n+1$ 
are arranged in the $m$-th row from left to right with $k$ increasing.  
Then the $m+1$-th row are placed 
half a column to the right of the $m$-th row as shown in Figure \ref{fig:CCF0}. 
We call the entries $\sigma_{k,\ell}$ the $(k,\ell)$ entries of the CCF. 
For fixed $k\in \Z$, the entries $\sigma_{k,\ell'}$ with $\ell'\in\Z$ and $k\leq \ell'$ 
are called the $k$-th diagonal entries, and for fixed $\ell\in \Z$, 
the entries $\sigma_{k',\ell}$ with $k'\in\Z$ and $k'\leq \ell$ are called the $\ell$-th anti-diagonal 
entries.  
\par
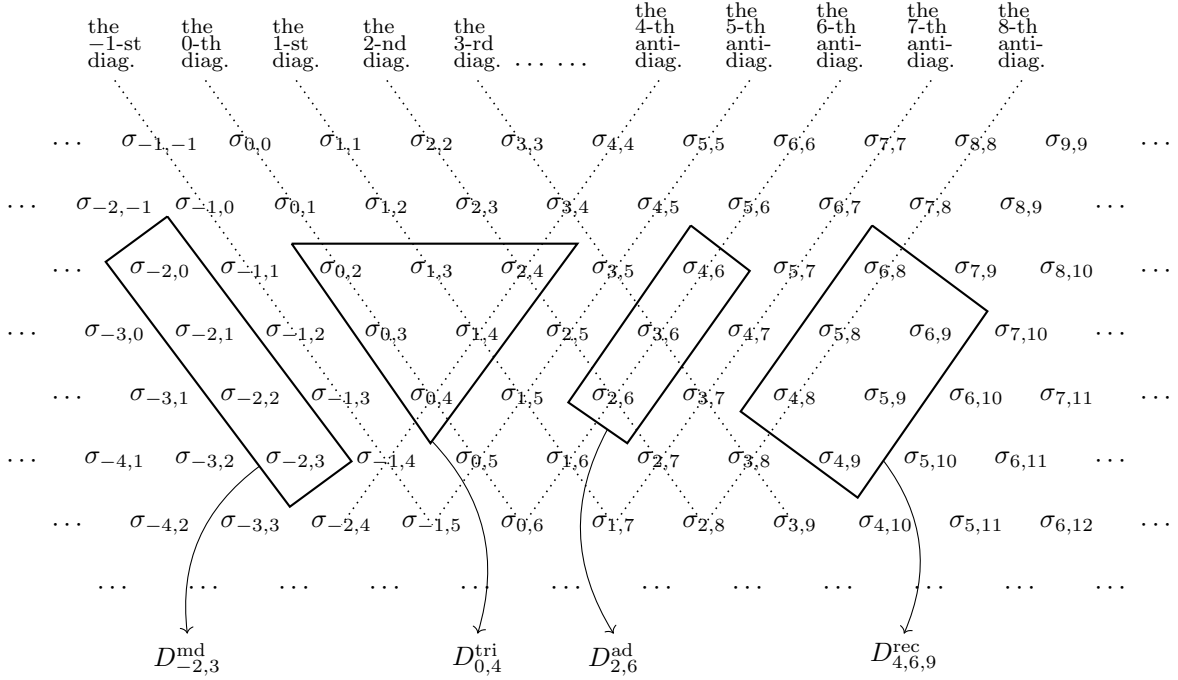
\begin{figure}[htbp]
\begin{center}
\begin{tikzpicture}[scale=1.2]
\foreach \k in{-1,...,3}{
 \pgfmathsetmacro{\l}{int(\k)}
    \draw[dotted, semithick] (\k-.5,2.1) to (\k+3,-2.8);
}
\node at(-1.5,2.1)[above,align=left]{
{\footnotesize the}\\[-.2cm]
{\footnotesize $-1$-st}\\[-.2cm]
{\footnotesize diag.}
};
\node at(-.5,2.1)[above,align=left]{
{\footnotesize the}\\[-.2cm]
{\footnotesize $0$-th}\\[-.2cm]
{\footnotesize diag.}
};
\node at(.5,2.1)[above,align=left]{
{\footnotesize the}\\[-.2cm]
{\footnotesize $1$-st}\\[-.2cm]
{\footnotesize diag.}
};
\node at(1.5,2.1)[above,align=left]{
{\footnotesize the}\\[-.2cm]
{\footnotesize $2$-nd}\\[-.2cm]
{\footnotesize diag.}
};
\node at(2.5,2.1)[above,align=left]{
{\footnotesize the}\\[-.2cm]
{\footnotesize $3$-rd}\\[-.2cm]
{\footnotesize diag.}
};
\node at(2.8,2.1) [above right]{$\cdots\;\cdots$};
\foreach \k in{4,...,8}{
 \pgfmathsetmacro{\l}{int(\k)}
    \draw[dotted,semithick] (\k+.5,2.1) to (\k-3,-2.8);
\node at(\k+.5,2.1)[above,align=left]{
{\footnotesize the}\\[-.2cm]
{\footnotesize $\l$-th}\\[-.2cm]
{\footnotesize anti-}\\[-.2cm]
{\footnotesize diag.}
};
}
\node at(-2,1.4) {$\cdots$};
\foreach \k in{-1,...,9}{   
    \node at(\k,1.4) {$\sigma_{\k,\k}$};
}
\node at(10,1.4) {$\cdots$};
\node at(-2-.5,.7) {$\cdots$};
\foreach \k in{-1,...,9}{  
 \pgfmathsetmacro{\l}{int(\k-1)}
   % 計算結果を座標として使用
\node at(\k-.5,.7) {$\sigma_{\l,\k}$};
}
\node at(10-.5,.7) {$\cdots$};
\node at(-2,0) {$\cdots$};
\foreach \k in{-1,...,9}
{  
 \pgfmathsetmacro{\l}{int(\k-1)}
 \pgfmathsetmacro{\m}{int(\k+1)}
   % 計算結果を座標として使用
   \node at(\k,0) {$\sigma_{\l,\m}$};
}
\node at(10,0) {$\cdots$};
\node at(-2-.5,-.7) {$\cdots$};
\foreach \k in{-1,...,9}{
 \pgfmathsetmacro{\l}{int(\k-2)}
 \pgfmathsetmacro{\m}{int(\k+1)}
 \node at(\k-.5,-.7) {$\sigma_{\l,\m}$};
}
\node at(10-.5,-.7) {$\cdots$};
\node at(-2,-1.4) {$\cdots$};
\foreach \k in{-1,...,9}{   
     \pgfmathsetmacro{\l}{int(\k-2)}
     \pgfmathsetmacro{\m}{int(\k+2)}
     \node at(\k,-1.4) {$\sigma_{\l,\m}$};
}
\node at(10,-1.4) {$\cdots$};
\node at(-2-.5,-2.1) {$\cdots$};
\foreach \k in{-1,...,9}{   
     \pgfmathsetmacro{\l}{int(\k-3)}
     \pgfmathsetmacro{\m}{int(\k+2)}
     \node at(\k-.5,-2.1) {$\sigma_{\l,\m}$};
}
\node at(10-.5,-2.1) {$\cdots$};
\node at(-2,-2.8) {$\cdots$};
\foreach \k in{-1,...,9}{   
     \pgfmathsetmacro{\l}{int(\k-3)}
     \pgfmathsetmacro{\m}{int(\k+3)}
    \node at(\k,-2.8) {$\sigma_{\l,\m}$};
}
\node at(10,-2.8) {$\cdots$};
%\node at(-2-.5,-3.5) {$\cdots$};
\foreach \k in{-1,...,9}{   
   \node at(\k-.5,-3.5) {$\cdots$};
}
\node at(10-.5,-3.5) {$\cdots$};
\draw[thick] (-.92,.6)--(-1.60,.1)--(.43,-2.6)--(1.11,-2.1)--(-.92,.6);
\path [draw, ->, thin, bend right = 30] (0.1,-2.15) to (-.7,-4);
\node at(-.7,-4) [below]{$D^{\md}_{-2,3}$};
\draw[thick] (.45,.3)--(3.6,.3)--(1.98,-1.9)--(.45,.3);
\path [draw, ->, thin, bend left = 30] (2.,-1.88) to (2.5,-4);
\node at(2.5,-4) [below]{$D^{\tri}_{0,4}$};
\draw[thick] (4.85,0.5)--(5.5,0)--(4.15,-1.9)--(3.5,-1.45)--(4.85,0.5);
\path [draw, ->, thin, bend right = 30] (3.93,-1.75) to (4,-4);
\node at(4,-4) [below]{$D^{\ad}_{2,6}$};
\draw[thick] (6.85,0.5)--(8.12,-.445)--(6.69,-2.5)--(5.4,-1.55)--(6.85,0.5);
\path [draw, ->, thin, bend left = 30] (6.98,-2.1) to (7.2,-4);
\node at(7.2,-4) [below]{$D^{\rec}_{4,6,9}$};
\end{tikzpicture}
\caption{Conway-Coxeter frieze $\{\sigma_{k,\ell}\}_{(i,j)\in D}$}\label{fig:CCF0}
\end{center}
\end{figure}
\newpage
For each $(k,\ell)\in D$ and $u\in\Z$, define subsets of $D$, 
$D^{\tri}_{k,\ell},D^{\md}_{k,\ell},D^{\ad}_{k,\ell}$, and $D^{\rec}_{k,u,\ell}$ as follows. 
For $(k,\ell)\in D$ with $\ell-k=0,1$, set 
$D^{\tri}_{k,\ell}=D^{\md}_{k,\ell}=D^{\ad}_{k,\ell}=D^{\rec}_{k,u,\ell}=\phi$. 
For $(k,\ell)\in D$ with $\ell-k\geq2$, define
\begin{align*}
&D^{\tri}_{k,\ell}=\{(a,b)\in D\mid k\leq a, a+2\leq b\leq \ell\}\,,\\
&D^{\md}_{k,\ell}=\{(k,k+2),(k,k+3),\dots,(k,\ell)\}=\{(k,b)\in D\mid k+2 \leq b \leq\ell\}\,,\\
&D^{\ad}_{k,\ell}=\{(k,\ell),(k+1,\ell),\dots,(\ell-2,\ell)\}=\{(a,\ell)\in D\mid k\leq a\leq \ell-2\}\,,\\
&D^{\rec}_{k,u,\ell}=
\left\{
\begin{array}{ll}
\{(a,b)\in D\mid k\leq a\leq u, u+2\leq b\leq\ell\}& \text{if $k\leq u,u+2\leq \ell$}\;,\\[.2cm]
\phi& \text{if $u<k$ or $\ell<u+2$}\;.
\end{array}
\right.
\end{align*}
We call $D^{\tri}_{k,\ell},D^{\md}_{k,\ell}$, and $D^{\ad}_{k,\ell}$ 
the triangular region, the diagonal region, and the anti-diagonal region for $(k,\ell)\in D$, respectively. 
We also call $D^{\rec}_{k,u,\ell}$ the rectangular region for $(k,u,\ell)$. 
In Figure \ref{fig:CCF0}, we illustrate the regions 
$D^{\tri}_{0,4},D^{\md}_{-2,3},D^{\ad}_{2,6}$, and $D^{\rec}_{4,6,9}$ enclosed by frames.  
\begin{defi}[Number of occurrences of $1$]\label{def:occ1}
Let $\{\sigma_{k,\ell}\}_{(k,\ell)\in D}$ be the CCF for $(c_u)_{u\in\Z}$. 
The sets
$U^{\tri}_{k,\ell}, U^{\md}_{k,\ell}, U^{\ad}_{k,\ell}$, and $U^{\rec}_{k,u,\ell}$ 
are defined as the subsets of the corresponding regions 
consisting of those elements $(a,b)$ for which $\sigma_{a,b}=1$. 
We denote their cardinalities by 
$\nu^{\tri}_{k,\ell}, \nu^{\md}_{k,\ell}, \nu^{\ad}_{k,\ell}$, and 
$\nu^{\rec}_{k,u,\ell}$. They are defined precisely as follows: 
\begin{align*}
&U^{\tri}_{k,\ell}=\{(a,b)\in D^{\tri}_{k,\ell}\mid\sigma_{a,b}=1\}, 
\quad \nu^{\tri}_{k,\ell}=\#{U}^{\tri}_{k,\ell}\,,\\
&U^{\md}_{k,\ell}=\{(k,b)\in D^{\md}_{k,\ell}\mid \sigma_{k,b}=1\},\quad 
\nu^{\md}_{k,\ell}=\#{U}^{\md}_{k,\ell}\,,\\
&U^{\ad}_{k,\ell}=\{(a,\ell)\in D^{\ad}_{k,\ell}\mid \sigma_{a,\ell}=1\},\quad
\nu^{\ad}_{k,\ell}=\#{U}^{\ad}_{k,\ell}\,,\\
&U^{\rec}_{k,u,\ell}=\{(a,b)\in D^{\rec}_{k,u,\ell}\mid \sigma_{a,b}=1\},\quad 
\nu^{\rec}_{k,u,\ell}=\#{U}^{\rec}_{k,u,\ell}\;.
\end{align*}
We call these cardinalities the numbers of occurrences of 1.
\end{defi}
\vskip.2cm
If the entries in the first through the $m$-th rows are positive and 
those in the $m$-th row are all $1$, then, by the unimodular rule,  
the entries in the $m+1$-st row are all $0$ and 
those in the $m+2$-nd row are all $-1$. 
We call such a CCF a positive CCF of width $m$ and denote it by 
$\{\sigma_{k,\ell}\}_{(k,\ell)\in D_{m+2}}$, where, for $N>0$, we set 
$D_{N}=\{(k,\ell)\in D\mid 0\leq \ell-k\leq N\}$. 
\par
As will be stated in Theorem \ref{thm:CCF} below, 
Conway and Coxeter \cite[Problem 28,\,29]{conway1973triangulated,conway1973triangulated_conti} 
showed that positive CCFs of finite width are parameterized by triangulations of polygons. 
\begin{defi}[CCFs associated with triangulations]\label{def:quiddity_CCF} 
Let $\T$ be a triangulation of an $n$-gon.  
The CCF for its quiddity $\quid(\T)$ 
is called the CCF associated with $\T$ and is denoted by $\CCF(\T)$.   
\end{defi}

\begin{thm}[{\cite[Problem 28,\,29]{conway1973triangulated,conway1973triangulated_conti}}]\label{thm:CCF}\quad
\begin{enumerate}[{(}1{)}]
\item For any triangulation of an $n$-gon, the associated  CCF $\CCF(\T)$ 
is a positive CCF of width $n-1$.  
\item For any positive CCF of width $m$, there exists 
a triangulation $\T$ of an $m+1$-gon such that the given CCF is $\CCF(\T)$. 
\end{enumerate}
\end{thm}
Theorem \ref{thm:CCF} (1) above is equivalent to Proposition \ref{prop:3desect_cont_frac_CCF} (2), 
stated and proved later. 
%
%\newpage
\begin{example}\label{ex:triangular} 
Figure \ref{fig:triangular} shows the triangulation $\T$ with 
$\quid(\T)=(\dots,1,4,2,1,3,4,1,2,3,\dots)$, 
and Figure \ref{fig:CCF} shows the associated CCF $\CCF(\T)$.   
\par
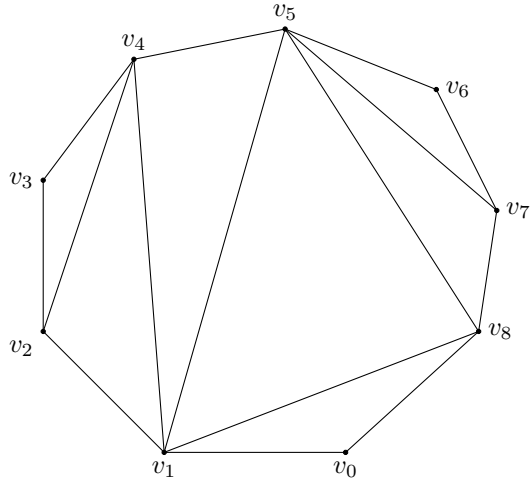
\begin{figure}[htbp]
\begin{center}
\begin{tikzpicture}[scale=.8]
\coordinate(v0)at(0,0);
\coordinate(v1)at(-3,0);
\coordinate(v2)at(-5,2);
\coordinate(v3)at(-5,4.5);
\coordinate(v4)at(-3.5,6.5);
\coordinate(v5)at(-1,7);
\coordinate(v6)at(1.5,6);
\coordinate(v7)at(2.5,4);
\coordinate(v8)at(2.2,2);
\filldraw [black] (v0) circle (1pt) node[below]{$v_0$};
\filldraw [black] (v1) circle (1pt) node[below]{$v_1$};
\filldraw [black] (v2) circle (1pt) node[below left]{$v_2$};
\filldraw [black] (v3) circle (1pt) node[left]{$v_3$};
\filldraw [black] (v4) circle (1pt) node[above]{$v_4$};
\filldraw [black] (v5) circle (1pt) node[above]{$v_5$};
\filldraw [black] (v6) circle (1pt) node[right]{$v_6$};
\filldraw [black] (v7) circle (1pt) node[right]{$v_7$};
\filldraw [black] (v8) circle (1pt) node[right]{$v_8$};
\draw (v0)--(v1);
\draw (v1)--(v2);
\draw (v1)--(v4);
\draw (v1)--(v5);
\draw (v1)--(v8);
\draw (v2)--(v3);
\draw (v2)--(v4);
\draw (v3)--(v4);
\draw (v4)--(v5);
\draw (v5)--(v6);
\draw (v5)--(v7);
\draw (v5)--(v8);
\draw (v6)--(v7);
\draw (v7)--(v8);
\draw (v8)--(v0);
\end{tikzpicture}
\caption{Triangulation $\T$ with $\quid(\T)=(\dots,1,4,2,1,3,4,1,2,3,\dots)$}\label{fig:triangular}
\end{center}
\end{figure}
\par
\newpage
\begin{figure}[htbp]
\begin{center}
\begin{tabularx}{14.5cm}{XXXXXXXXXXXXXXXXXXXX}
&$0$& &$0$& &$0$& &$0$&&$0$&&$0$&&$0$&&$0$&&$0$&&$0$\\
$1$& &$1$& &$1$& &$1$&&$1$&&$1$&&$1$&&$1$&&$1$&&$1$&\\
&$1$& &$4$& &$2$& &$1$&&$3$&&$4$&&$1$&&$2$&&$3$& & $1$\\
$2$&&$3$&&$7$&&$1$&&$2$&&$11$&&$3$&&$1$&&$5$&&$2$\\
&$5$&&$5$&&$3$&&$1$&&$7$&&$8$&&$2$&&$2$&&$3$&&$5$\\
$7$&&$8$&&$2$&&$2$&&$3$&&$5$&&$5$&&$3$&&$1$&&$
7$&\\
&$11$&&$3$&&$1$&&$5$&&$2$&&$3$&&$7$&&$1$&&$2$&&$11$\\
$3$&&$4$&&$1$&&$2$&&$3$&&$1$&&$4$&&$2$&&$1$&&$3$&\\
&$1$&&$1$&&$1$&&$1$&&$1$&&$1$&&$1$&&$1$&&$1$&&$1$\\
$0$& &$0$& &$0$& &$0$&&$0$&&$0$&&$0$&&$0$&&$0$&&$0$&\\
&$-1$& &$-1$& &$-1$& &$-1$&&$-1$&&$-1$&&$-1$&&$-1$&&$-1$&&$-1$
\end{tabularx}
\vskip.5cm
\caption{CCF associated with the triangulation $\T$ with 
$\quid(\T)=(\dots,1,4,2,1,3,4,1,2,3,\dots)$}
\label{fig:CCF}
\end{center}
\end{figure}
\end{example}
\vskip-.5cm
\subsection{Unimodular matrices}
\begin{defi}[Unimodular matrix]\label{def:modular_gp}
Let $\SL(2,\Z)$ denote the group of $2\times2$ unimodular matrices over $\mathbb{Z}$, 
that is, $\SL(2,\Z)=\biggl\{
A=\begin{sumipmatrix}
a&b\\
c&d
\end{sumipmatrix}\;\vrule\;a,b,c,d\in\Z,\det A=1\biggr\}$. Let $(c_u)_{u\in\Z}$ be a sequence of positive integers.  
For any subsequence $(c_i,c_{i+1},\dots,c_j)$ of $(c_u)_{u\in\Z}$, 
we define a matrix $M(c_i,c_{i+1},\dots,c_j)$ in $\SL(2,\Z)$ by 
\begin{align}\label{eqn:prod_mat2}
M(c_i,c_{i+1},\dots,c_j)&=
\begin{pmatrix}
c_i&-1\\
1&0
\end{pmatrix}
\begin{pmatrix}
c_{i+1}&-1\\
1&0
\end{pmatrix}
\cdots
\begin{pmatrix}
c_{j-1}&-1\\
1&0
\end{pmatrix}
\begin{pmatrix}
c_j&-1\\
1&0
\end{pmatrix}\;.
\end{align}
\end{defi}
\subsection{Continued fractions}
\begin{defi}[Negative continued fraction]\label{def:cont_frac_int}
For any subsequence $(c_i,c_{i+1},\dots,c_j)$ of $(c_u)_{u\in\Z}$,  
let $[[c_i,c_{i+1},\dots,c_j]]$ denote a continued fraction defined by 
\begin{align}\label{eqn:cont_frac_int}
[[c_i,c_{i+1},\dots,c_j]]
&=c_i-\dfrac{1}{c_{i+1}
-\dfrac{1}{\ddots
\raisebox{-.2cm}{
$c_{j-2}-\dfrac{1}{c_{j-1}-\dfrac{1}{c_j}}$}
}}\quad.
\end{align}
We call this a negative continued fraction. 
\end{defi}
\noindent
Note that $[[c_i,c_{i+1},\dots,c_j]]$ lies in $\Q$ if and only if $[[c_{i+1},\dots,c_j]]\neq 0$ 
in the formula $(\ref{eqn:cont_frac_int})$. 
We will show in Proposition \ref{prop:3desect_cont_frac_CCF} that 
if the sequence $(c_u)_{u\in\Z}$ is the quiddity of some triangulation $\T$ of an $n$-gon, 
then $[[c_i,c_{i+1},\dots,c_j]]$ lies in $\Q$ for any $i,j$ with 
$i\leq j\leq i+n-2$; moreover, $[[c_{i+1},\dots,c_j]]>0$ for any $i,j$ with $i+1\leq j\leq i+n-2$.  
\subsection{Correspondence between Farey labelings, CCFs, and continued fractions}
The result of Series \cite{series1985modular} is reformulated in Morier-Genoud and 
Ovsienko \cite[Section 2.6]{morier2019farey}, 
implies that $[[c_1,c_2,\dots,c_j]]=\rho_{1,j+1}$ 
for any quiddity $(c_u)_{u\in\Z}$ 
of a triangulation $(\T,1)$ with exactly two exterior cells, where $c_0=c_{k+1}=1$, 
and any integer $j$ with $1\leq j\leq k$.  
Moreover, Morier-Genoud and Ovsienko \cite[Proposition 3.1]{morier2019farey} showed 
that the entries of $M(c_1,c_2,\dots,c_j)$ are equal to the numerators and denominators of 
$[[c_1,c_2,\dots,c_j]]$ and $[[c_1,c_2,\dots,c_{j-1}]]$, 
thereby connecting the matrix product with the continued fractions. 
However, their results are stated for triangulations that contain exactly two exterior cells. 
In Proposition \ref{prop:3desect_cont_frac_CCF} (1) and (3), 
we extend these statements to general triangulations. 
Furthermore, Morier-Genoud, Ovsienko, and Tabachnikov 
\cite[Proposition 2.2.1, Theorem 1]{morier2015sl} provided alternative proofs of 
the results of Coxeter \cite{coxeter1971frieze} and Conway and Coxeter 
\cite[Problem 28,\,29]{conway1973triangulated, conway1973triangulated_conti}, 
which relate CCFs to Farey labelings, 
that $\dfrac{\sigma_{i-1,j+1}}{\sigma_{i,j+1}}=\rho_{i,j+1}$. 
In Proposition \ref{prop:3desect_cont_frac_CCF} (2), we also reformulate these results. 

% Coxeter は単に Farey 列 と frieze の対応（Farey 列の分母分子を並べれることで frieze の対角成分が得られる（1対1になるということは M-G＆O で強調））に注目したが, Conway-Coxeter では, Farey 列が多角形の三角形分割の Farey graph として考えられて, 三角形分割の quiddity が対応する CCF の種数列を与えることを示した。
%
\par
The equation $(\ref{eqn:mat_n-1})$ in the following Lemma \ref{lemm:n-1} 
is the result of Morier-Genoud and Ovsienko \cite[Proposition 3.9]{morier2019farey}, 
in which they compute the unimodular matrix for a particular subsequence $(c_i,c_{i+1},\dots,c_{i+n-1})$ of the quiddity of a triangulation of an $n$-gon. 
We also compute the corresponding unimodular matrices for the subsequences $(c_i,c_{i+1},\dots,c_{i+n-2})$ and $(c_i,c_{i+1},\dots,c_{i+n-3})$.  
\begin{lemm}[cf.\ {\cite[Proposition 3.9]{morier2019farey}}]\label{lemm:n-1}
Let $\quid(\T)=(c_u)_{u\in\Z}$ be the quiddity of a triangulation $\T$ on an $n$-gon. 
For any $i\in\Z$, we have
\begin{align}
&M(c_{i}, c_{i+1},\dots, c_{i+n-1})
=-\begin{pmatrix}
1 & 0 \\
0 & 1 \\
\end{pmatrix}\,,\label{eqn:mat_n-1}\\
&M(c_{i}, c_{i+1},\dots, c_{i+n-2})
=\begin{pmatrix}
0&-1\\
1&-c_{i-1}
\end{pmatrix}\,,\label{eqn:mat_n-2}\\
%\intertext{}
&M(c_{i},c_{i+1},\dots,c_{i+n-3})
=\begin{pmatrix}
1&-c_{i-1}\\
c_{i-2}&-c_{i-2}c_{i-1}+1
\end{pmatrix}\;.\label{eqn:mat_n-3}
\end{align}
\end{lemm}
\noindent
\begin{proof} 
The equation $(\ref{eqn:q-mat_n-1})$ in Lemma \ref{lemm:n-1_q} is 
a $q$-deformation of the equation $(\ref{eqn:mat_n-1})$. 
We prove the equation $(\ref{eqn:q-mat_n-1})$ in Section \ref{sect:proof_lem}. 
Taking the limit as $q\to1$ in this proof, we recover a proof of the equation $(\ref{eqn:mat_n-1})$. 
The equations $(\ref{eqn:mat_n-2})$ and $(\ref{eqn:mat_n-3})$ follow from 
straightforward calculations:  
\begin{align*}
&M(c_{i}, c_{i+1},\dots, c_{i+n-2})\\
&=M(c_{i+1}, c_{i+2},\dots, c_{i+n-2},c_{i+n-1})
\begin{pmatrix}
c_{i-1}&-1\\
1&0
\end{pmatrix}^{-1}=-\begin{pmatrix}
0&1\\
-1&c_{i-1}
\end{pmatrix}=\begin{pmatrix}
0&-1\\
1&-c_{i-1}
\end{pmatrix}\,,\\
&M(c_{i},c_{i+1},\dots,c_{i+n-3})\\
&=M(c_i,c_{i+1},\dots,c_{i+n-3},c_{i+n-2})
\begin{pmatrix}
c_{i-2}&-1\\
1&0
\end{pmatrix}^{-1}=\begin{pmatrix}
0&-1\\
1&-c_{i-1}
\end{pmatrix}
\begin{pmatrix}
0&1\\
-1&c_{i-2}
\end{pmatrix}\\
&\hspace{7.5cm}=
\begin{pmatrix}
1&-c_{i-1}\\
c_{i-2}&-c_{i-2}c_{i-1}+1
\end{pmatrix}\,.
\end{align*}
\end{proof}
\begin{prop}[cf.~{\cite{coxeter1971frieze,morier2019farey,series1985modular}}]
\label{prop:3desect_cont_frac_CCF}
Let $(\T,i)$ be a triangulation of $n$-gon. 
Let $\quid(\T)=(c_u)_{u\in\Z}$, 
$\Farey(\T,i)^{\per}=(\rho_{i,u})_{u\in\Z}$, and $\CCF(\T)=\{\sigma_{k,\ell}\}_{(k,\ell)\in D_{n+1}}$ 
be the corresponding quiddity, anti-periodic extension of the Farey labeling, and CCF, respectively. 
Fix any $j\in\Z$ with $i\leq j\leq i+n-2$. 
If the unimodular matrix $M(c_{i},c_{i+1},\dots,c_{j})$ is written in the form 
$M(c_{i},c_{i+1},\dots,c_{j})=
\begin{sumipmatrix}
{\mathcal R}&-{\mathcal R}'\\
{\mathcal S}&-{\mathcal S}'
\end{sumipmatrix}$, then the following hold:  
\begin{enumerate}[{(}1{)}]
\item We have $\rho_{i,j}=\dfrac{{\mathcal R}'}{{\mathcal S}'}\;,\;\rho_{i,j+1}=\dfrac{\mathcal R}{\mathcal S}$. 
Except in the cases where ${\mathcal S}'=0$ if $j=i$ and ${\mathcal R}=0$ if $j=i+n-2$, 
${\mathcal R},{\mathcal S},{\mathcal R}'$, and ${\mathcal S}'$ are all positive.  
\item The entries of the CCF are related to those of the unimodular matrix as follows: 
\begin{align*}
\begin{sumipmatrix} 
\sigma_{i-1,j+1}& \sigma_{i-1,j}\\
\sigma_{i,j+1}& \sigma_{i,j}\\
\end{sumipmatrix}
=\begin{sumipmatrix}
{\mathcal R}&{\mathcal R}'\\
{\mathcal S}&{\mathcal S}'
\end{sumipmatrix}\,.
\end{align*}
\item We have $[[c_i,c_{i+1},\dots,c_{j-1}]]=\dfrac{{\mathcal R}'}{{\mathcal S}'}\;,\;[[c_i,c_{i+1},\dots,c_{j-1},c_j]]=\dfrac{{\mathcal R}}{{\mathcal S}}$, 
where, if $j=i$, then 
\\
$[[c_i,c_{i+1},\dots,c_{j-1}]]=[[\quad]]=\dfrac{1}{0}$. 
\end{enumerate}
\end{prop}
\begin{proof} We use the abbreviation $\rho_{i,j}=\rho_j$. 
We will prove (1), (2), and (3) by induction on $m=j-i$. 
\par
For $m=0$ and $m=1$, the statements follow by direct computation
from the definitions of the Farey labeling, CCF, and continued fractions.
Indeed,
\[
M(c_i)
=
\begin{pmatrix}
c_i & -1\\
1 & 0
\end{pmatrix},
\qquad
M(c_i,c_{i+1})
=
\begin{pmatrix}
c_ic_{i+1}-1 & -c_i\\
c_{i+1} & -1
\end{pmatrix}\,.
\]
The corresponding Farey labels are
\[
\rho_i=\frac10,\quad
\rho_{i+1}=\frac{c_i}{1},\quad
\rho_{i+2}=\frac{c_ic_{i+1}-1}{c_i}.
\]
The positivity statements are immediate except when
$c_ic_{i+1}-1=0$, equivalently $n=3$. This proves (1).  
\par
From the definition of the CCF, we have 
\[
\sigma_{i,i}=0,\quad \sigma_{i-1,i}=1,\quad \sigma_{i,i+1}=1,\quad \sigma_{i-1,i+1}=c_i,\quad\sigma_{i,i+2}=c_{i+1}\,. 
\]
The unimodular rule~\eqref{eqn:unimodular} gives 
$\sigma_{i-1,i+2}=c_ic_{i+1}-1$. This proves (2). 
\par
Finally, the corresponding continued fractions are 
\[
[[\;\;]]=\dfrac{1}{0},\quad [[c_i]]=\dfrac{c_i}{1},\quad [[c_i,c_{i+1}]]=\dfrac{c_ic_{i+1}-1}{c_{i+1}}\,.
\]
This proves (3). 
\par
Next, assume that the results (1), (2), and (3) hold for all $j\leq i+m$ with $m\geq 1$. 
We show that they also hold for $j=i+m+1$. 
We prove the result (1) separately in the two cases $j\leq i+n-3$ and $j=i+n-2$. 
\par
\noindent
{\bf Case 1}. Suppose $j\leq i+n-3$. By the induction hypothesis, if  
\begin{align}\label{eq:prod_mat_1}
&M(c_i,\dots,c_{j-1})
=\begin{pmatrix}
c_i&-1\\
1&0
\end{pmatrix}
\cdots
\begin{pmatrix}
c_{j-1}&-1\\
1&0
\end{pmatrix}
=
\begin{pmatrix}
R&-R'\\
S&-S'
\end{pmatrix}\,,
\end{align}
then $\rho_{j-1}=\dfrac{R'}{S'}$, $\rho_j=\dfrac{R}{S}$, and $R,S,R',S'$ are positive. 
Therefore, we have 
\begin{align*}
&M(c_i,\dots,c_j)=
M(c_i,\dots,c_{j-1})M(c_j)
=\begin{pmatrix}
R&-R'\\
S&-S'
\end{pmatrix}
\begin{pmatrix}
c_j&-1\\
1&0
\end{pmatrix}
=
\begin{pmatrix}
c_jR-R'&-R\\
c_jS-S'&-S
\end{pmatrix}\,.
\end{align*}
Since $\rho_j=\dfrac{R}{S}$, the first equation of (1) follows. Moreover, 
both $R$ and $S$ are positive.  
To prove the second equation of (1), it suffices to show that 
\begin{align}\label{eqn:CCF_Farey_k+1}
\rho_{j+1}=\dfrac{c_jR-R'}{c_jS-S'}\,.
\end{align}
Here, since $j=i+m+1\le i+n-3$ and $1\le m$, we have $i+2\le j\le i+n-3$ and $n\ge5$. 
Hence the vertex $v_{j+1}$ is distinct from the endpoints
$v_{i-1},v_i$ of the base edge. By the Farey labeling, both the numerator and denominator of
$\rho_{j+1}$ are positive. 
\par
Let $\cell(v_j)=\{{\mathcal C}''_0,\dots,{\mathcal C}''_{c_j-1}\}$ 
be the counterclockwise ordering of the cells adjacent to $v_j$. 
Write the cell of lowest level as ${\mathcal C}''_\mu=v_{z_{\mu+1}}v_{z_\mu}v_j$, 
where $\mu=\mu_{i,j}$ is the descendant degree of $v_j$. 
%Then, $v_{z_\mu}v_j$, $v_{z_{\mu+1}}v_{z_\mu}$, and $v_jv_{z_{\mu+1}}$ are the base side of 
%${\mathcal C}''_{\mu-1}$, ${\mathcal C}''_{\mu}$, and ${\mathcal C}''_{\mu+1}$, respectively, 
%where, if $\mu=0$ or $\mu=c_j-1$, then 
%${\mathcal C}''_{\mu-1}$ or ${\mathcal C}''_{\mu+1}$ do not exist, respectively. 
%By the Farey labeling, 
Let the endpoints of the base side
$v_{z_{\mu+1}}v_{z_\mu}$ be labeled by
$\rho_{z_\mu}=\dfrac{a}{b}$ and $\rho_{z_{\mu+1}}=\dfrac{a'}{b'}$. 
Since 
\[
\dfrac{R}{S}=\rho_j=\rho_{z_{\mu+1}}\oplus \rho_{z_\mu}
=\dfrac{a'+a}{b'+b}\,,
\]
we obtain $\rho_{z_\mu}=\dfrac{R-a'}{S-b'}$. 
Successive applications of the Farey sum along the cells adjacent to $v_j$ yield 
\[
\rho_{j-1}=\frac{(\mu+1)R-a'}{(\mu+1)S-b'}\,.
\]
Since $\rho_{j-1}=\dfrac{R'}{S'}$, it follows that $\rho_{z_{\mu+1}}=
\dfrac{(\mu+1)R-R'}{(\mu+1)S-S'}$. 
Applying the same argument in the opposite direction, we obtain
\[
\rho_{j+1}
=\dfrac{c_jR-R'}{c_jS-S'}.
\]
Therefore, the equation $(\ref{eqn:CCF_Farey_k+1})$ follows. 
\par
\noindent
{\bf Case 2}. Suppose $j=i+n-2$. 
Lemma \ref{lemm:n-1} gives
$M(c_i,c_{i+1},\dots,c_{i+n-3})
=\begin{sumipmatrix}
1&-c_{i-1}\\
c_{i-2}&-(c_{i-2}c_{i-1}-1)
\end{sumipmatrix}$ and $M(c_i,c_{i+1},\dots,c_{i+n-2})=
\begin{sumipmatrix}
0&-1\\
1&-c_{i-1}
\end{sumipmatrix}$. 
By the induction hypothesis, $\rho_{i+n-2}=\dfrac{1}{c_{i-1}}$, 
while the initial values of the Farey labeling give $\rho_{i+n-1}=\dfrac{0}{1}$. 
Therefore, the result (1) follows. 
\par
Next, we prove the result (2) and (3). 
Let $M(c_{i+1},\dots,c_j)
=
\begin{sumipmatrix}
r&-r'\\
s&-s'
\end{sumipmatrix}$ and 
$M(c_i,\dots,c_{j-1})=\begin{sumipmatrix}
R&-R'\\
S&-S'
\end{sumipmatrix}$, where all entries are positive by the result (1). 
By the induction hypothesis for (2) and (3), we have
\begin{align}
\begin{pmatrix}
\sigma_{i,j+1}&\sigma_{i,j}\\
\sigma_{i+1,j+1}&\sigma_{i+1,j}
\end{pmatrix}
=\begin{pmatrix}
r&r'\\
s&s'
\end{pmatrix}\,,\;
\begin{pmatrix}
\sigma_{i-1,j}&\sigma_{i-1,j-1}\\
\sigma_{i,j}&\sigma_{i,j-1}
\end{pmatrix}
=\begin{pmatrix}
R&R'\\
S&S'
\end{pmatrix}
\label{eqn:shucho2}
\end{align}
and 
\begin{align}
&[[c_{i+1},\dots,c_{j-1}]]=\dfrac{r'}{s'}\,,\; 
[[c_{i+1},\dots,c_j]]=\dfrac{r}{s}\,,\;
[[c_i,\dots,c_{j-2}]]=\dfrac{R'}{S'}\,,
\;[[c_i,\dots,c_{j-1}]]=\dfrac{R}{S}\,.\label{eqn:shucho3}
\end{align}
Then, we have 
\begin{align*}
&M(c_i,\dots,c_j)=M(c_i)M(c_{i+1},\dots,c_j)
=
\begin{pmatrix}
c_ir-s&-(c_ir'-s')\\
r&-r'
\end{pmatrix}
\end{align*}
and 
\begin{align*}
&M(c_i,\dots,c_j)
=M(c_i,\dots,c_{j-1})M(c_j)
=\begin{pmatrix}
c_jR-R'&-R\\
c_jS-S'&-S
\end{pmatrix}\,.
\end{align*}
Let $M(c_{i},\dots,c_{j})=
\begin{sumipmatrix}
{\mathcal R}&-{\mathcal R}'\\
{\mathcal S}&-{\mathcal S}'
\end{sumipmatrix}$. 
From $(\ref{eqn:shucho2})$, we obtain 
\begin{align*}
{\mathcal R}'=R=\sigma_{i-1,j}\,,\;{\mathcal S}'
=S=r'=\sigma_{i,j}\,,\;{\mathcal S}=r=\sigma_{i,j+1}\,.
\end{align*}
Hence $M(c_i,\dots,c_j)=
\begin{sumipmatrix}
{\mathcal R}&-\sigma_{i-1,j}\\
\sigma_{i,j+1}&-\sigma_{i,j}
\end{sumipmatrix}$. 
Since each matrix $\begin{sumipmatrix}
c_u&-1\\
1&0
\end{sumipmatrix}$ in the product defining $M(c_i,\dots,c_j)$ has determinant $1$, we have 
\[
\det M(c_i,\dots,c_j)
=\sigma_{i-1,j}\sigma_{i,j+1}
-{\mathcal R}\sigma_{i,j}=1\,.
\]
Combining this with the unimodular rule of CCFs, 
$\sigma_{i-1,j}\sigma_{i,j+1}
-\sigma_{i-1,j+1}\sigma_{i,j}
=1$, we obtain
\[
(\sigma_{i-1,j+1}-{\mathcal R})\sigma_{i,j}=0\,. 
\]
Since $\sigma_{i,j}\neq0$, it follows that
${\mathcal R}=\sigma_{i-1,j+1}$. Therefore, the result~(2) follows.
\par
By (\ref{eqn:shucho3}), we have 
\begin{align*}
[[c_i,\dots,c_{j-1}]]
&=\dfrac{R}{S}\,,\\
[[c_i,\dots,c_j]]&=
c_i-\dfrac{1}{[[c_{i+1},\dots,c_j]]}=
c_i-\dfrac{s}{r}=
\dfrac{c_ir-s}{r}\,.
\end{align*}
Thus, the result (3) follows. 
\end{proof}

The following Proposition \ref{prop:CCF_n-1} extends 
Proposition \ref{prop:3desect_cont_frac_CCF} (1), (2), and (3) to the case $j=i+n-1$. 
\begin{prop}\label{prop:CCF_n-1} 
The equations in Proposition \ref{prop:3desect_cont_frac_CCF} (1), (2), and (3) remain valid for 
$j=i+n-1$. 
\end{prop}
\begin{proof} From the equation $(\ref{eqn:mat_n-1})$ in Lemma \ref{lemm:n-1}, we have 
$M(c_i,c_{i+1},\dots,c_{i+n-1})
=\begin{sumipmatrix}
-1&0\\
0&-1
\end{sumipmatrix}$. 
By the definition $(\ref{eqn:Farey_graph_infty})$ 
of the anti-periodic extension $\Farey(\T,i)^{\per}=(\rho_u)_{u\in\Z}$, 
we have $\rho_{i,i+n-1}=\dfrac{0}{1},\;\rho_{i,i+n}=\dfrac{-1}{-0}=\dfrac{-1}{0}$. 
Hence, the equations in Proposition \ref{prop:3desect_cont_frac_CCF} (1) also hold 
for $j=i+n-1$. 
Furthermore, from the equation $(\ref{eqn:mat_n-2})$ in Lemma \ref{lemm:n-1}, we have 
\begin{align*}
&M(c_i,c_{i+1},\dots,c_{i+n-2})
=\begin{pmatrix}
0&-1\\
1&-c_{i-1}
\end{pmatrix}\,,\; 
M(c_{i+1},c_{i+2},\dots,c_{i+n-1})
=\begin{pmatrix}
0&-1\\
1&-c_i
\end{pmatrix}\,.
\end{align*}
From Proposition \ref{prop:3desect_cont_frac_CCF} (2) in the cases 
$j=i+n-2$ for the triangulation $(\T,i)$ or $j=i+n-1$ for the triangulation $(\T,i+1)$, it follows that
\begin{align*}
\sigma_{i-1,i+n-1}=0\,,\quad 
\sigma_{i,i+n-1}=1\,,\quad 
\sigma_{i,i+n}=0\,.
%\label{eqn:mat_n-2_element}
\end{align*}
By the unimodular rule of CCFs and these equations, 
% $(\ref{eqn:mat_n-2_element})$
%\begin{align*}
%\sigma_{i-1,i+n-1}\sigma_{i,i+n}-\sigma_{i,i+n-1}\sigma_{i-1,i+n}&=1
%\end{align*}
$\sigma_{i-1,i+n}=-1$ holds. Hence, 
the equations in Proposition \ref{prop:3desect_cont_frac_CCF} (2) also hold 
for $j=i+n-1$. 
\par
From Proposition \ref{prop:3desect_cont_frac_CCF} in the case 
$j=i+n-1$ for the triangulation $(\T,i+1)$, we have 
\\
$[[c_{i+1},\dots,c_{i+n-1}]]=\dfrac{\sigma_{i,i+n}}{\sigma_{i+1,i+n}}=\dfrac{0}{1}$. 
Here, we compute the continued fraction $[[c_i,c_{i+1},\dots,c_{i+n-1}]]$ as follows: 
\begin{align*}
[[c_i,c_{i+1},\dots,c_{i+n-1}]]=&c_i-\dfrac{1}{[[c_{i+1},\dots,c_{i+n-1}]]}
=c_i-\dfrac{1}{\quad\dfrac{0}{1}\quad}
=c_i-\dfrac{1}{0}
=\dfrac{c_i\cdot0-1}{0}
=\dfrac{-1}{0}\,.
\end{align*}
Since $\sigma_{i,i+n}=0$ and $\sigma_{i-1,i+n}=-1$, we have 
$[[c_i,c_{i+1},\dots,c_{i+n-1}]]=\dfrac{\sigma_{i-1,i+n}}{\sigma_{i,i+n}}$. 
Therefore, the equations in Proposition \ref{prop:3desect_cont_frac_CCF} (3) hold for $j=i+n-1$. 
\end{proof}
\subsection{Diagonal counts in a triangulation and occurrences of \texorpdfstring{$1$}{\it 1} 
in the associated CCF}
\begin{lemm}\label{lemm:count1_trianglation} 
Let $\CCF(\T)=\{\sigma_{(i,j)}\}_{(i,j)\in D_{n+1}}$ be 
the CCF of a triangulation $\T$ of an $n$-gon. 
For any $(i,j)\in D_{n+1}$, we have 
\begin{align*}
&\nu^{\tri}_{i,j}=
\left\{
\begin{array}{ll}
\delta(V[i,j])&\text{if $i\leq j\leq i+n-2$},\\
n-2&\text{if $j=i+n-1$},\\
n-2+c_i&\text{if $j=i+n$},\\
n-2+c_i+c_{i+1}&\text{if $j=i+n+1$},
\end{array}
\right.\\
&\nu^{\md}_{i,j}=
\left\{
\begin{array}{ll}
\delta(\{v_i\},V[i+2,j])&\text{if $i\leq j\leq i+n-2$},\\
c_i&\text{if $i+n-1\leq j\leq i+n+1$},
\end{array}
\right.\\
&\nu^{\ad}_{i,j}=
\left\{
\begin{array}{ll}
\delta(V[i,j-2],\{v_j\})&\text{if $i\leq j\leq i+n-2$},\\
c_j&\text{if $i+n-1\leq j\leq i+n+1$}.
\end{array}
\right.
\intertext{For any $(i,j)\in D_{n+1}$ with $i\leq u,u+2\leq j$, we have}
&\nu^{\rec}_{i,u,j}-\delta(V[i,u],V[u+2,j])
=\left\{
\begin{array}{ll}
0&\text{if $i\leq j\leq i+n-2$},\\
1&\text{if $j=i+n-1$},\\
2&\text{if $j=i+n,i+1\leq u\leq i+n-3$},\\
1&\text{if $j=i+n,u=i,i+n-2$},\\
3&\text{if $j=i+n+1,i+2\leq u \leq i+n-3$},\\
2&\text{if $j=i+n+1,u=i+1,\; i+n-2$},\\
1&\text{if $j=i+n+1,u=i,\; i+n-1$}. 
\end{array}
\right.
\end{align*}
\end{lemm}
\noindent
\begin{proof} If $j=i$ or $j=i+1$, then 
$\nu^{\tri}_{i,j}=\nu^{\md}_{i,j}=\nu^{\ad}_{i,j}=\nu^{\rec}_{i,u,j}=0$, and 
$\delta(V[i,j])=\delta(\{v_i\},V[i+2,j])=\delta(V[i,j-2],\{v_j\})=\delta(V[i,u],V[u+2,j])
=0$. Therefore, the results follow. 
\par
Suppose $i+2\leq j\leq i+n+1$. Let $(a,b)\in U^{\tri}_{i,j}$. 
We consider the anti-periodic extension of the Farey labeling 
$\Farey(\T,a+1)^{\per}=(\rho_{a+1,u})_{u\in\Z}$ of the triangulation $(\T,a+1)$ with the base edge $v_{a}v_{a+1}$. 
Since $i\leq a$ and $a+2\leq b\leq j$, $a+1\leq b-1\leq (a+1)+n-1$ holds. 
Therefore, by Proposition \ref{prop:3desect_cont_frac_CCF} (1) and Proposition \ref{prop:CCF_n-1}, we have 
\begin{align}\label{eqn:Farey2}
\rho_{a+1,b}=\dfrac{\sigma_{a,b}}{\sigma_{a+1,b}}\;.
\end{align}
\par
\noindent
{\bf Case 1}. Suppose $i+2\leq j\leq i+n$. Then, $a+2\leq b\leq j \leq a+n$ holds. 
Thus, by the definition of the Farey labeling, 
the numerator of $\rho_{a+1,b}$ is equal to $1$ if and only if 
$v_av_b\in E_{\T}$. Therefore,  
\begin{align}\label{eqn:Farey3}
U^{\tri}_{i,j}
&=\{(a,b)\mid i\leq a,\;a+2\leq b\leq j,\;\sigma_{a,b}=1\}
=\{(a,b)\mid i\leq a,\;a+2\leq b\leq j,\;v_av_b\in E_{\T}\}\,.
\end{align}
We consider $(a,b)\in U^{\tri}_{i,j}$ separately in the following cases. 
\par
\noindent
{\bf Case 1.1}. Suppose $i+2\leq j\leq i+n-2$. Since $a+2\leq b\leq j\leq a+n-2$, 
$v_av_b\in E_{\T}$ if and only if $v_av_b\in\Diag_{\T}$. Hence, 
$U^{\tri}_{i,j}=\{(a,b)\mid i\leq a,\;a+2\leq b\leq j,\;v_av_b\in\Diag_{\T}\}$.  
The map $(a,b)\mapsto v_av_b$ 
defines a bijection from $U^{\tri}_{i,j}$ to $\Diag_{\T}(V[i,j])$. Therefore, 
$\nu^{\tri}_{i,j}=\#U^{\tri}_{i,j}=\#\Diag_{\T}(V[i,j])=\delta(V[i,j])$.  
\par
\noindent
{\bf Case 1.2}. Suppose $j=i+n-1$. Then $a+2\leq b\leq i+n-1$ holds. 
If $a=i$ and $b=i+n-1$, then $v_av_b=v_iv_{i+n-1}\in{\mathcal E}$; otherwise, 
$v_av_b\in\Diag_{\T}$. Hence,  
\begin{align*}
U^{\tri}_{i,i+n-1}
&=\{(a,b)\mid i\leq a,\;a+2\leq b\leq i+n-1,\;v_av_b\in\Diag_{\T}\}\cup\{(i,i+n-1)\}\,.
\end{align*}
Since $V[i,i+n-1]=V$ and $\#\Diag_{\T}(V)=n-3$, we obtain 
$\nu^{\tri}_{i,i+n-1}=n-2$.   
\par
\noindent
{\bf Case 1.3}. Suppose that $j=i+n$. Since 
\begin{align*}
U^{\tri}_{i,i+n}=U^{\tri}_{i,i+n-1}\cup U^{\ad}_{i,i+n},\quad U^{\tri}_{i,i+n-1}\cap U^{\ad}_{i,i+n}=\phi\,,
\end{align*}
we have $\nu^{\tri}_{i,i+n}=\nu^{\tri}_{i,i+n-1}+\nu^{\ad}_{i,i+n}$. 
By Case 1.2, $\nu^{\tri}_{i,i+n-1}=n-2$. 
Let $(a,i+n)\in U^{\ad}_{i,i+n}=\{(a,i+n)\mid i\leq a\leq i+n-2,v_av_i\in E_{\T}\}$. 
If $a=i$, then $v_iv_i\not\in E_{\T}$. 
If $a=i+1$, then $v_{i+1}v_i\in {\mathcal E}$.   
Furthermore, if $i+2\leq a\leq i+n-2$, 
then $v_av_i\in E_{\T}$ if and only if 
$v_av_i\in \Diag_{\T}(V[i+2,i+n-2],\{v_{i}\})$. 
Therefore
\begin{align*}
U^{\ad}_{i,i+n}
&=\{(i+1,i+n)\}\cup
\{(a,i+n)\mid i+2\leq a\leq i+n-2,\;v_av_{i}\in\Diag_{\T}(V[i+2,i+n-2],\{v_i\})\}\,.
\end{align*}
Since $\delta(V[i+2,i+n-2],\{v_i\})=c_i-1$, we obtain $\nu^{\ad}_{i,i+n}=c_i$. 
Therefore, $\nu^{\tri}_{i,i+n}=n-2+c_i$. 
\par
\noindent
{\bf Case 2}. Suppose $j=i+n+1$. Since 
\begin{align*}
U^{\tri}_{i,i+n+1}=U^{\tri}_{i,i+n}\cup U^{\ad}_{i,i+n+1},\quad U^{\tri}_{i,i+n}\cap U^{\ad}_{i,i+n+1}=\phi\,,
\end{align*}
we have $\nu^{\tri}_{i,i+n+1}=\nu^{\tri}_{i,i+n}+\nu^{\ad}_{i,i+n+1}$.  
By Case 1.3, $\nu^{\tri}_{i,i+n+1}=n-2+c_i$. 
Let $(a,i+n+1)\in U^{\ad}_{i,i+n+1}=\{(a,i+n+1)\mid i\leq a\leq i+n-1,\sigma_{a,i+n+1}=1\}$. 
If $a=i$, then $\sigma_{i,i+n+1}=-1$. If $a=i+1$, then 
$\sigma_{i+1,i+n+1}=0$. Thus, $(i,i+n+1),\;(i+1,i+n+1)\not\in U^{\ad}_{i,i+n+1}$. 
If $i+2\leq a\leq i+n-1$, then $\sigma_{a,i+n+1}=1$ if and only if 
$v_av_{i+n+1}=v_av_{i+1}\in E_{\T}(V[i+2,i+n-1],\{v_{i+1}\})$. 
Hence, $\nu^{\ad}_{i,i+n+1}=\#E_{\T}(V[i+2,i+n-1],\{v_{i+1}\})=c_{i+1}$, 
and therefore, $\nu^{\tri}_{i,i+n+1}=n-2+c_i+c_{i+1}$. 
\par
The results for $\nu^{\md}_{i,j}$, $\nu^{\ad}_{i,j}$, and $\nu^{\rec}_{i,u,j}$ follow 
from analogous counting arguments, so we omit the details to avoid repetition. 
\end{proof}
\begin{lemm}\label{lem:occ1} 
Let $\quid(\T)=(c_u)_{u\in\Z}$ be the quiddity of a triangulation $\T$ of an $n$-gon. 
For any $i,j\in\Z$ with $i-1\leq j\leq i+n-1$, we have
\begin{align*}
\sum_{u=i}^j(c_u-1)
&=\begin{cases}
\nu^{\tri}_{i-1,j}+\nu^{\tri}_{i,j+1}+\nu^{\rec}_{i,j,i+n-2}
&\text{if $i-1\leq j\leq i+n-3$}, \\[2mm]
\nu^{\tri}_{i-1,i+n-2}+\nu^{\tri}_{i,i+n-1}-c_{i-1}-1
&\text{if $j=i+n-2$}, \\[2mm]
\nu^{\tri}_{i-1,i+n-1}+\nu^{\tri}_{i,i+n}-c_{i-1}-c_i-2
&\text{if $j=i+n-1$}.
\end{cases}
\end{align*}
\end{lemm}
\begin{proof} 
Throughout the proof, we repeatedly use 
Lemma \ref{lemm:count1_trianglation} without further mention. 
\par
We first consider the boundary case $j=i-1$. 
Since  
\begin{align*}
&\Diag_{\T}(V[i-1,i-1])=U^{\tri}_{i-1,i-1}=\phi\;,\quad\Diag_{\T}(V[i,i])=U^{\tri}_{i,j+1}=\phi\;,\\
&\Diag_{\T}(V[i,i-1],V[i+1,i+n-2])=U^{\rec}_{i,i-1,i+n-2}=\phi\,, 
\end{align*}
and since we interpret $\sum\limits_{u=i}^{i-1}(c_u-1)$ as $0$, the result follows. 
\par
Next, let $j=i+n-2$. Then 
$\sum\limits_{u=i}^{i+n-2}(c_u-1)
=\sum\limits_{u=i}^{i+n-1}(c_u-1)-(c_{i+n-1}-1)
=2(n-3)-c_{i-1}+1$. 
Moreover,  
$\nu^{\tri}_{i-1,i+n-2}=\nu^{\tri}_{i,i+n-1}=n-2$. 
Hence, 
$\sum\limits_{u=i}^{i+n-2}(c_u-1)=\nu^{\tri}_{i-1,i+n-2}+\nu^{\tri}_{i,i+n-1}-c_{i-1}-1$. 
\par
Next, let $j=i+n-1$. Then $\sum\limits_{u=i}^{i+n-1}(c_u-1)=2(n-3)$. 
Moreover, 
$\nu^{\tri}_{i-1,i+n-1}=n-2+c_{i-1}$, $\nu^{\tri}_{i,i+n}=n-2+c_i$. 
Hence, $\sum\limits_{u=i}^{i+n-1}(c_u-1)=\nu^{\tri}_{i-1,i+n-1}
+\nu^{\tri}_{i,i+n}-c_{i-1}-c_i-2$.  
\par
Finally, assume that $i\leq j\leq i+n-3$. Then  
\begin{align*}
&\delta(V[i-1,j])=\nu^{\tri}_{i-1,j}\;,\quad\delta(V[i,j+1])=\nu^{\tri}_{i,j+1}\;,\quad
\delta(V[i,j],V[j+2,i+n-2])=\nu^{\rec}_{i,j,i+n-2}\,.
\end{align*}
Thus, it suffices to prove 
\begin{align}\label{eqn:sum_quiddity}
\sum_{u=i}^j(c_u-1)=\delta(V[i-1,j])+\delta(V[i,j+1])+\delta(V[i,j],V[j+2,i+n-2])\,.
\end{align}
\par
\noindent
{\bf Case 1}. Suppose $\sigma_{i-1,j+1}=1$. 
As observed in the proof of Lemma \ref{lemm:count1_trianglation}, 
$\sigma_{i-1,j+1}=1$ implies $v_{i-1}v_{j+1}\in E_{\T}$. 
Set $V'=V[i-1,j+1]=\{v_{i-1},\;v_i,\;\dots,\;v_{j+1}\}$, and let $\T'$ be the subtriangulation of $\T$ with vertex set $V'$ and edge set $E_{\T}(V')$. 
Since $\T'$ is a triangulation of a closed subpolygon, 
$\Diag_{\T}(V[i,j])=\Diag_{\T'}(V'[i,j])$. 
Therefore, in the quantity $\sum\limits_{u=i}^j(c_u-1)$, 
each diagonal in $\Diag_{\T}(V[i,j])$ contributes twice, and 
each diagonal in $\Diag_{\T}(V[i,j],\{v_{i-1},v_{j+1}\})$ contributes once. 
Hence
\begin{align*}
\sum\limits_{u=i}^j(c_u-1)=2\delta(V[i,j])+\delta(V[i,j],\{v_{i-1},v_{j+1}\})\;.
\end{align*}
Counting the diagonals in 
$\Diag_{\T}(V[i-1,j])$ and $\Diag_{\T}(V[i,j+1])$ with multiplicity gives 
\begin{align*}
\delta(V[i-1,j])+\delta(V[i,j+1])=2\delta(V[i,j])+\delta(V[i,j],\{v_{i-1},v_{j+1}\})\;.
\end{align*} 
Therefore, $\sum\limits_{u=i}^j(c_u-1)=\delta(V[i-1,j])+\delta(V[i,j+1])$. 
Since $\T'$ triangulates the closed subpolygon, 
$\delta(V[i,j],V[j+2,i+n-2])=0$, and 
the equation $(\ref{eqn:sum_quiddity})$ follows. 
\par
\noindent
{\bf Case 2}. Suppose $\sigma_{i-1,j+1}\geq2$. 
Again by the proof of Lemma \ref{lemm:count1_trianglation}, 
$\sigma_{i-1,j+1}\geq2$ implies $v_{i-1}v_{j+1}\not\in E_{\T}$. 
Let $v_{x}$ and $v_{y}$ be the vertices immediately before and after 
$v_{j+1}$ that are adjacent to $v_{i-1}$, that is, 
$x=\max\{i\leq u<j+1\mid v_{i-1}v_u\in E_{\T}\}$ and 
$y=\max\{j+1< u\leq i+n-2\mid v_{i-1}v_u\in E_{\T}\}$. 
Set $V'=V[i-1,y]=\{v_{i-1},\;v_i,\;\dots,\;v_y\}$ and 
let $\T'$ be the subtriangulation of $\T$ with vertex set $V'$ and edge set 
$E_{\T}(V')$. 
In the sum $\sum\limits_{u=i}^j(c_u-1)$,
each diagonal in $\Diag_{\T}(V[i,j])$ contributes twice, and 
each diagonal in $\Diag_{\T}(V[i,j],\{v_{i-1}, v_{j+1}\}\cup V[j+2,y])$ contributes once. 
Hence, 
\begin{align*}
\sum\limits_{u=i}^j(c_u-1)
&=2\delta(V[i,j])+\delta(V[i,j],\{v_{i-1}, v_{j+1}\}\cup V[j+2,y])\\
&=2\delta(V[i,j])+\delta(V[i,j],\{v_{i-1}, v_{j+1}\})+\delta(V[i,j],V[j+2,y])\,.
\end{align*}
Since $\T'$ triangulates a closed subpolygon, 
$\Diag_{\T'}(V[i,j],V[y+1,i+n-2])=\phi$. 
Therefore, $\Diag_{\T}(V[i,j],V[j+2,y])=\Diag_{\T}(V[i,j],V[j+2,i+n-2])$. 
Hence, 
\begin{align*}
\sum\limits_{u=i}^j(c_u-1)
&=2\delta(V[i,j])+\delta(V[i,j],\{v_{i-1}, v_{j+1}\})+\delta(V[i,j],V[j+2,i+n-2])\,.
\end{align*}
By the same counting argument as in Case 1, 
$\delta(V[i-1,j])+\delta(V[i,j+1])=2\delta(V[i,j])+\delta(V[i,j],\{v_{i-1},v_{j+1}\})$. 
Thus, the equation $(\ref{eqn:sum_quiddity})$ follows. 
\par
Note that if $\sigma_{i-1,j+1}\geq2$, then 
$v_{x}v_{y}\in\Diag_{\T}(V[i,j],V[j+2,i+n-2])$, and hence 
\begin{align}\label{eqn:sum_quiddity2}
\sum\limits_{u=i}^j(c_u-1)>\delta(V[i-1,j])+\delta(V[i,j+1]). 
\end{align} 
\end{proof}
\section{Correspondence between 
\texorpdfstring{$q$}{\it q}-deformations of CCFs, unimodular matrices, continued fractions, and Farey labelings}
In Sections~\ref{sect:q_CCF_unim_conti_frac}
and~\ref{sect:q_Farey},
we review $q$-deformations of Farey labelings,
Conway--Coxeter friezes, unimodular matrices,
and continued fractions, following
\cite{morier2020q,morier2021quantum}. 
In \cite[Theorem~3]{morier2020q} 
and \cite[Proposition~8, Proposition~10]{morier2021quantum}, 
Morier-Genoud and Ovsienko established correspondences 
among these $q$-deformations. 
However, these results were obtained for triangulations 
with exactly two exterior cells. 
In Proposition~\ref{prop:q_CCF_q_rational}, 
Proposition~\ref{prop:q-CCF_n-1}, 
and Theorem~\ref{prop:q_CCF_weighted_Farey}, 
we extend these correspondences to arbitrary triangulations. 
\subsection{\texorpdfstring{$q$}{\it q}-deformations of 
CCFs, unimodular matrices, and continued fractions}
\label{sect:q_CCF_unim_conti_frac}
In this section, we introduce $q$-deformations of Conway-Coxeter friezes, 
unimodular matrices, and continued fractions.  
\par
For any non-negative integer $a$, 
define a polynomial $[a]_q\in\Z[q]$ with variable $q$ over $\Z$ by 
\begin{align*}
&[0]_q=0\;,\quad [a]_q=1+q+q^2+\cdots+q^{a-1}=\dfrac{1-q^a}{1-q}\,.
\end{align*}
 
\begin{defi}[$q$-deformed CCF {\cite[Definition 2]{morier2021quantum}}]\label{def:q_CCF}
Let $(c_u)_{u\in\Z}$ be a cyclic sequence of positive integers with period $n$. 
We define a $q$-deformation $\{\zeta_{k,\ell}\}_{(k,\ell)\in D}$ of the CCF for $(c_u)_{u\in\Z}$. 
For each $(k,\ell)\in D$, we assign a polynomial $\zeta_{k,\ell}\in\Z[q]$ as follows. 
\par
First, for $(k,\ell)\in D$ with $\ell-k=0,1,2$, 
define $\zeta_{k,\ell}$ by
\begin{align}\label{eqn:q_unimodular0}
\zeta_{k,\ell}=
\begin{cases}
0&\text{if $\ell-k=0$}, \\
1&\text{if $\ell-k=1$}, \\
\left[c_i\right]_q&\text{if $\ell-k=2$}.
\end{cases}
\end{align}
Moreover, 
for $(k,\ell)\in D$ with $\ell-k\geq3$, 
the polynomials $\zeta_{k,\ell}$ are recursively defined 
so that they satisfy the following $q$-unimodular rule: 
\begin{align}\label{eqn:q_unimodular}
\zeta_{k,\ell-1}\zeta_{k+1,\ell}-\zeta_{k+1,\ell-1}\zeta_{k,\ell}&=q^{\sum\limits_{u=k+1}^{\ell-1}(c_u-1)}\,.
\end{align}
We call $\{\zeta_{k,\ell}\}_{(k,\ell)\in D}$ $q$-deformed Conway-Coxeter frieze, 
abbreviated as $q$-CCF, for $(c_u)_{u\in\Z}$. 
\par
Let $\T$ be a triangulation of an $n$-gon. z
The $q$-CCF for its quiddity $\quid(\T)=(c_u)_{u\in\Z}$ is called
the $q$-CCF associated with $\T$ and is denoted by $\CCF_q(\T)$.  
\end{defi}
As for CCFs, 
it is not immediate that $\zeta_{k,\ell}\in\Z[q]$ is well-defined for general cyclic sequences 
$(c_u)_{u\in\Z}$ from \eqref{eqn:q_unimodular}. 
In Proposition \ref{prop:q_CCF_q_rational}, we show that 
$q$-CCFs for quiddities of triangulations of polygons are well-defined. 
Since we work only with such $q$-CCFs, we omit the details in the general case. 
\par
It is clear that $\lim\limits_{q\to1}[c_u]_q=c_u$. 
Moreover, by taking the limit of the $q$-unimodular rule $(\ref{eqn:q_unimodular})$ as $q\to1$,  
we obtain the ordinary unimodular rule  $(\ref{eqn:unimodular})$. 
Let $\sigma_{(k,\ell)}=\lim\limits_{q\to1}\zeta_{k,\ell}$. 
Then $\{\sigma_{(k,\ell)}\}_{(k,\ell)\in D}$ is the CCF for $(c_u)_{u\in\Z}$. 
If $\{\sigma_{(k,\ell)}\}_{(k,\ell)\in D}$ is a positive CCF of width $m$, then we call 
the original $q$-CCF a positive $q$-CCF of width $m$ and denote it by 
$\{\zeta_{k,\ell}\}_{(k,\ell)\in D_{m+2}}$. 
Theorem \ref{thm:q_CCF_2} below and Corollary \ref{cor:q_CCF_1} later state that the positive $q$-CCFs of width $m$ are characterized by triangulations of an $m+1$-gon, which is a $q$-analogue of Conway and Coxeter's theorem \cite[Problem 29]{conway1973triangulated,conway1973triangulated_conti} stated in Theorem \ref{thm:CCF} (2). 
\begin{thm}[cf.~{\cite[Problem 29]{conway1973triangulated,conway1973triangulated_conti}}]
\label{thm:q_CCF_2} 
For any positive $q$-CCF of width $m$, there exists a triangulation $\T$ of an $m+1$-gon 
such that the given $q$-CCF is $\CCF_q(\T)$. 
\end{thm}
\begin{proof} Let $\{\zeta_{k,\ell}\}_{(k,\ell)\in D}$ be a positive $q$-CCF of width $m$ 
for a cyclic sequence $(c_u)_{u\in\Z}$. 
Let $\sigma_{(k,\ell)}=\lim\limits_{q\to1}\zeta_{k,\ell}$. 
Then $\{\sigma_{(k,\ell)}\}_{(k,\ell)\in D}$ is a positive CCF of width $m$. 
Therefore, by Theorem \ref{thm:CCF} (2), there exists a triangulation $\T$ of an $m+1$-gon 
such that $\{\sigma_{(k,\ell)}\}_{(k,\ell)\in D}=\CCF(\T)$, that is, $\quid(\T)=(c_u)_{u\in\Z}$. 
Hence, 
$\{\zeta_{k,\ell}\}_{(k,\ell)\in D}=\CCF_q(\T)$. The result follows. 
\end{proof}
%さらに, 正値帯状 $q$-CCF の成分 $\zeta_{k,\ell}$ は, 正整数係数多項式となる（命題 \ref{prop:coprime}）.  
%
\begin{defi}[$q$-deformed unimodular matrix (cf.\ {\cite[Definition 4.2]{morier2020q}})]
Let 
\begin{align*}
\SL(2,\Z[q])=\left\{
A=\begin{pmatrix}
a&b\\
c&d
\end{pmatrix}\;\vrule\;a,b,c,d\in\Z[q],\;\det A\in\{q^m\mid m\in\Z\}\;\right\}\,.
\end{align*}
Let $(c_u)_{u\in \Z}$ be a sequence of positive integers. 
For any subsequence $(c_i,c_{i+1},\dots,c_j)$ of $(c_u)_{u\in \Z}$, 
we define a matrix $M_q(c_i,c_{i+1},\dots,c_j)$ in $\SL(2,\Z[q])$ by
\begin{align}\label{eqn:q-prod_mat2}
M_q(c_i,c_{i+1},\dots,c_j)
&=
\begin{pmatrix}
[c_i]_q&-q^{c_i-1}\\
1&0
\end{pmatrix}
\begin{pmatrix}
[c_{i+1}]_q&-q^{c_{i+1}-1}\\
1&0
\end{pmatrix}
\cdots
\begin{pmatrix}
[c_j]_q&-q^{c_j-1}\\
1&0
\end{pmatrix}\;.
\end{align}
\end{defi}

The following lemma is a $q$-analogue of Lemma \ref{lemm:n-1}. 	
The equation \eqref{eqn:q-mat_n-1} is the extension of Morier-Genoud and 
Ovsienko \cite[Corollary 4.8 (1)]{morier2020q} to general triangulations. 
\begin{lemm}[{\cite[Corollary 4.8 (1)]{morier2020q}}]\label{lemm:n-1_q}
Let $\quid(\T)=(c_u)_{u\in\Z}$ be the quiddity of a triangulation $\T$ of an $n$-gon. 
Then, for any $i\in\Z$, we have 
\begin{align}
&M_q(c_i, c_{i+1}, \dots, c_{i+n-1})
=-q^{\,n-3}I\,,\label{eqn:q-mat_n-1}\\
&M_q(c_i, c_{i+1},\dots, c_{i+n-2})
=q^{n-2-c_{i-1}}\begin{pmatrix}
0&-q^{c_{i-1}-1}\\
1&-[c_{i-1}]_q
\end{pmatrix}\,,\label{eqn:q-mat_n-2}\\
&M_q(c_i,c_{i+1},\dots,c_{i+n-3})
=q^{n-1-c_{i-2}-c_{i-1}}
\begin{pmatrix}
q^{c_{i-1}-1}&-q^{c_{i-1}-1}[c_{i-2}]_q\\[.2cm]
[c_{i-1}]_q&-[c_{i-2}]_q[c_{i-1}]_q+q^{c_{i-2}-1}
\end{pmatrix}\,,\label{eqn:q-mat_n-3}
\end{align}
where $I$ denotes the $2\times2$ identity matrix. 
\end{lemm}
\begin{proof} We prove $(\ref{eqn:q-mat_n-1})$ 
in Section \ref{sect:proof_lem}. 
The equation \eqref{eqn:q-mat_n-2} and \eqref{eqn:q-mat_n-3} follow  
by straightforward calculations from \eqref{eqn:q-mat_n-1} as follows.  
\begin{align*}
M_q(c_i, c_{i+1},\dots, c_{i+n-2})
&=M_q(c_i, c_{i+1},\dots, c_{i+n-2},c_{i+n-1})
\begin{pmatrix}
[c_{i-1}]_q&-q^{c_{i-1}-1}\\
1&0
\end{pmatrix}^{-1}\\
&=-q^{n-3}\cdot q^{1-c_{i-1}}
\begin{pmatrix}
0&q^{c_{i-1}-1}\\
-1&[c_{i-1}]_q
\end{pmatrix}\\
&=q^{n-2-c_{i-1}}
\begin{pmatrix}
0&-q^{c_{i-1}-1}\\
1&-[c_{i-1}]_q
\end{pmatrix}
\end{align*}
\begin{align*}
M_q(c_i,c_{i+1},\dots,c_{i+n-3})&= M_q(c_i, c_{i+1},\dots,c_{i+n-3},c_{i+n-2})
{\begin{pmatrix} 
[c_{i-2}]_q& -q^{c_{i-2}-1} \\
1 & 0 \end{pmatrix}
}^{-1}\\
&=q^{n-2-c_{i-1}}\cdot q^{1-c_{i-2}}
\begin{pmatrix}
0&-q^{c_{i-1}-1}\\
1&-[c_{i-1}]_q
\end{pmatrix}
\begin{pmatrix}
0 & q^{c_{i-2}-1} \\
-1&[c_{i-2}]_q 
\end{pmatrix}\\
&=q^{n-1-c_{i-2}-c_{i-1}}
\begin{pmatrix}
q^{c_{i-1}-1}& -q^{c_{i-1}-1}[c_{i-2}]_q\\
[c_{i-1}]_q&-[c_{i-2}]_q[c_{i-1}]_q+q^{c_{i-2}-1}
\end{pmatrix}
\end{align*}
\end{proof}
\begin{defi}[$q$-deformed negative continued fraction
 (cf.\ {\cite[Definition 1.1]{morier2020q}})]
For any subsequence $(c_i,c_{i+1},\dots,c_j)$ of $(c_u)_{u\in \Z}$, 
$[[c_i,c_{i+1},\dots,c_j]]_q$ denote a continued fraction defined by
\begin{align}\label{eqn:q_cont_frac}
[[c_i,c_{i+1},\dots,c_j]]_q
=[c_i]_q-\dfrac{q^{c_i-1}}{[c_{i+1}]_q-\dfrac{q^{c_{i+1}-1}}{[c_{i+2}]_q-\dfrac{q^{c_{i+2}-1}}{
\ddots\raisebox{-.2cm}{$[c_{j-2}]_q-\dfrac{q^{c_{j-2}-1}}{
 [c_{j-1}]_q-\dfrac{q^{c_{j-1}-1}}{[c_j]_q}}$}
 }}}
\end{align}
We call this a $q$-deformed negative continued fraction, which is abbreviated as 
a $q$-continued fraction.   
\end{defi}
As in Definition \ref{def:cont_frac_int}, 
we remark that $[[c_i,c_{i+1},\dots,c_j]]_q$ lies in $\Q(q)$ 
if and only if $[[c_{i+1},\dots,c_j]]_q\neq 0$ 
in the formula \eqref{eqn:q_cont_frac}. 
We will show in Proposition \ref{prop:q_CCF_q_rational} (2) 
that if the sequence $(c_u)_{u\in\Z}$ is the quiddity of some triangulation $\T$ of an $n$-gon, 
then $[[c_i,c_{i+1},\dots,c_j]]_q$ lies in $\Q(q)$ for any $i,j$ with 
$i\leq j\leq i+n-2$. 
Proposition \ref{prop:q_CCF_q_rational} is a $q$-analogue of Proposition 
\ref{prop:3desect_cont_frac_CCF} (2), (3). This is also the extension 
of Morier-Genoud and Ovsienko \cite[Proposition 8, Proposition 10]{morier2021quantum} 
to general triangulations. 
\begin{prop}[cf.\ {\cite[Proposition 8, Proposition 10]{morier2021quantum}}]
\label{prop:q_CCF_q_rational}
Let $\quid(\T)=(c_u)_{u\in\Z}$ be the quiddity of a triangulation $\T$ of an $n$-gon. 
Let $\CCF_q(\T)=\{\zeta_{k,\ell}\}_{(k,\ell)\in D}$ be the $q$-CCF associated with $\T$. 
Fix any $j\in\Z$ with $i\leq j\leq i+n-2$. 
If the $q$-unimodular matrix $M_q(c_i,c_{i+1},\dots,c_j)$ is written in the form 
$M_q(c_i,c_{i+1},\dots,c_j)=
\begin{sumipmatrix}
{\mathcal R}&-q^{c_j-1}{\mathcal R}'\\
{\mathcal S}&-q^{c_j-1}{\mathcal S}'
\end{sumipmatrix}$, then the following hold: 
\begin{enumerate}[{(}1{)}]
\item The entries of the $q$-CCF are related to those of the $q$-unimodular matrix as follows: 
\begin{align*}
\begin{sumipmatrix}
\zeta_{i-1,j+1}&\zeta_{i-1,j}\\
\zeta_{i,j+1}&\zeta_{i,j}
\end{sumipmatrix}
=
\begin{sumipmatrix}
{\mathcal R}&{\mathcal R}'\\
{\mathcal S}&{\mathcal S}'
\end{sumipmatrix}\,,
\end{align*}
where $\zeta_{i,j}=\zeta_{i,i}={\mathcal S}'=0$ if $j=i$, 
$\zeta_{i-1,j+1}=\zeta_{i-1,i+n-1}={\mathcal R}=0$ if $j=i+n-2$, and otherwise 
${\mathcal R}',\;{\mathcal S}',\;{\mathcal R}$, and ${\mathcal S}$ are nonzero polynomials. 
\item We have 
$[[c_i,c_{i+1},\dots,c_{j-1}]]_q=\dfrac{{\mathcal R}'}{{\mathcal S}'}\,,\;
[[c_i,c_{i+1},\dots,c_{j-1},c_j]]_q=\dfrac{{\mathcal R}}{{\mathcal S}}$, 
where $[[c_i,c_{i+1},\dots,c_{j-1}]]_q=[[\quad]]_q=\dfrac{1}{0}$ if $j=i$. 
\end{enumerate}
\end{prop}
In the remainder of this paper, we often abbreviate $[a]_q$ as $[a]$. 
\begin{proof}
As in Proposition \ref{prop:3desect_cont_frac_CCF} (2), (3), 
we will prove (1) and (2) by induction on $m=j-i$. 
\par
First, we consider the case $m=0$, that is, $j=i$. 
$M_q(c_i)=\begin{sumipmatrix}
[c_i]&-q^{c_i-1}\\
1&0
\end{sumipmatrix}$ holds. 
By the definition of $q$-CCFs, we have
\begin{align*}
\begin{pmatrix}
\zeta_{i-1,j+1}&\zeta_{i-1,j}\\
\zeta_{i,j+1}&\zeta_{i,j}
\end{pmatrix}
=
\begin{pmatrix}
\zeta_{i-1,i+1}&\zeta_{i-1,i}\\
\zeta_{i,i+1}&\zeta_{i,i}
\end{pmatrix}
=
\begin{pmatrix}
[c_i]&1\\
1&0
\end{pmatrix}\,.
\end{align*}
Except for $\zeta_{i,i}=0$, entries are nonzero, and therefore, the result (1) follows. 
\par
Furthermore, 
$[[c_i,\cdots,c_{i-1}]]_q=[[\;\;]]_q=\dfrac{1}{0}\,,\;[[c_i,\dots,c_i]]_q
=[[c_i]]_q=[c_i]=\dfrac{[c_i]}{1}$. 
Hence, the result (2) follows.  
\par
Next, we consider the case $m=1$, that is, $j=i+1$. 
By straightforward calculation, we have
\[
M_q(c_i,c_{i+1})=
\begin{pmatrix}
[c_i]&-q^{c_i-1}\\
1&0
\end{pmatrix}
\begin{pmatrix}
[c_{i+1}]&-q^{c_{i+1}-1}\\
1&0
\end{pmatrix}
=
\begin{pmatrix}
[c_i][c_{i+1}]-q^{c_i-1}&-q^{c_{i+1}-1}[c_i]\\
[c_{i+1}]&-q^{c_{i+1}-1}
\end{pmatrix}\,.
\]
By the definition of $q$-CCF, it follows that
\begin{align}\label{eqn:q_CCF_i+1}
&\zeta_{i-1,i+1}=[c_i]\,,\quad \zeta_{i,i+1}=1\,,\quad 
\zeta_{i,i+2}=[c_{i+1}], 
\end{align}
and these are nonzero polynomials. 
The remaining equation 
\begin{align*}
\zeta_{i-1,j+1}=\zeta_{i-1,i+2}=[c_i][c_{i+1}]-q^{c_i-1}
\end{align*}
in (1) holds from the $q$-unimodular rule and the equations in \eqref{eqn:q_CCF_i+1}. 
We will prove separately, in the two cases $j=i+1\leq i+ n-3$ and $j=i+1=i+n-2$, 
that the entries are nonzero.  
\par
\noindent
{\bf Case 1}. Suppose $j=i+1\leq i+ n-3$. 
Since $n\geq 4$, we have $(c_i,c_{i+1})\neq(1,1)$, that is, $c_i\geq 2$ or $c_{i+1}\geq 2$. 
Since 
\begin{align*}
[c_i][c_{i+1}]-q^{c_i-1}
&=[c_i-1]+q[c_i-1][c_{i+1}-1]+q^{c_i}[c_{i+1}-1]
\end{align*}
and $[c_i-1]$ or $q^{c_i}[c_{i+1}-1]$ are nonzero, 
$[c_i][c_{i+1}]-q^{c_i-1}$ is nonzero. 
\par
\noindent
{\bf Case 2}. Suppose $j=i+1=i+n-2$. Then $n=3$. 
Since the possible quiddity is $(1,1,1)$,   
\begin{align*}
[c_i][c_{i+1}]-q^{c_i-1}
=[1][1]-1=0
\end{align*}
holds. Therefore, the result (1) follows. 
\par
The result (2) follows from straightforward calculations: 
\begin{align*}
&[[c_i,\dots,c_{j-1}]]_q=[[c_i]]_q=[c_i]=\dfrac{[c_i]}{1}\,,\\
&[[c_i,\dots,c_j]]_q=[[c_i,c_{i+1}]]_q=[c_i]-\dfrac{q^{c_i-1}}{[c_{i+1}]}=\dfrac{[c_i][c_{i+1}]-q^{c_i-1}}{[c_{i+1}]}\,.
\end{align*}
\par
Next, assume that the results (1) and (2) hold for all $j\leq i+m$ with $m\geq1$. 
We show that they also hold for $j=i+m+1$. 
The argument is parallel to that in the proof of 
Proposition \ref{prop:3desect_cont_frac_CCF} (2) and (3). 
\par
Let 
$M_q(c_{i+1},\dots,c_j)=
\begin{sumipmatrix}
r&-q^{c_j-1}r'\\
s&-q^{c_j-1}s'
\end{sumipmatrix}$ and 
$M_q(c_i,\dots,c_{j-1})=
\begin{pmatrix}
R&-q^{c_{j-1}-1}R'\\
S&-q^{c_{j-1}-1}S'
\end{pmatrix}$. 
Then, by the induction hypothesis,  
\begin{align}
&\begin{sumipmatrix}
\zeta_{i,j+1}&\zeta_{i,j}\\
\zeta_{i+1,j+1}&\zeta_{i+1,j}
\end{sumipmatrix}
=\begin{sumipmatrix}
r&r'\\
s&s'
\end{sumipmatrix}\,,\;
\begin{sumipmatrix}
\zeta_{i-1,j}&\zeta_{i-1,j-1}\\
\zeta_{i,j}&\zeta_{i,j-1}
\end{sumipmatrix}
=\begin{sumipmatrix}
R&R'\\
S&S'
\end{sumipmatrix}\label{eqn:q_shucho2}\,,
\end{align}
and the entries are nonzero polynomials. 
Furthermore, we have 
\begin{align}
&[[c_{i+1},\dots,c_{j-1}]]_q=\dfrac{r'}{s'}\,,\;
[[c_{i+1},\dots,c_j]]_q=\dfrac{r}{s}\,,\;
[[c_i,\dots,c_{j-2}]]_q=\dfrac{R'}{S'}\,,\;
[[c_i,\dots,c_{j-1}]]_q=\dfrac{R}{S}\,.\label{eqn:q_shucho3}
\end{align}
Then, we have
\begin{align*}
&M_q(c_i,\dots,c_j)
=M_q(c_i)M_q(c_{i+1},\dots,c_j)
=\begin{pmatrix}
[c_i]r-q^{c_i-1}s&-q^{c_j-1}([c_i]r'-q^{c_i-1}s')\\
r&-q^{c_j-1}r'
\end{pmatrix}\,,
\end{align*}
and 
\begin{align}\label{eqn:q-CCF_calc}
&M_q(c_i,\dots,c_j)
=M_q(c_i,\dots,c_{j-1})M_q(c_j)
=\begin{pmatrix}
[c_j]R-q^{c_{j-1}-1}R'&-q^{c_j-1}R\\
[c_j]S-q^{c_{j-1}-1}S'&-q^{c_j-1}S
\end{pmatrix}\,.
\end{align}
Let $M_q(c_i,c_{i+1},\dots,c_j)=
\begin{sumipmatrix}
{\mathcal R}&-q^{c_j-1}{\mathcal R}'\\
{\mathcal S}&-q^{c_j-1}{\mathcal S}'
\end{sumipmatrix}$. 
From \eqref{eqn:q_shucho2}, it follows that  
\begin{align*}
{\mathcal R}'=R=\zeta_{i-1,j},\;{\mathcal S}'=S=r'=\zeta_{i,j},\;
{\mathcal S}=r=\zeta_{i,j+1}\,,
\end{align*}
and these are nonzero polynomials. 
Thus, we have $M_q(c_i,\dots,c_j)=
\begin{sumipmatrix}
{\mathcal R}&-q^{c_j-1}\zeta_{i-1,j}\\
\zeta_{i,j+1}&-q^{c_j-1}\zeta_{i,j}
\end{sumipmatrix}$. 
Since  each matrix 
$\begin{sumipmatrix}
[c_u]&-q^{c_u-1}\\
1&0
\end{sumipmatrix}$ in the product defining $M_q(c_i,\dots,c_j)$ has 
determinant $q^{c_u-1}$, it follows that 
\begin{align*}
\det M_q(c_i,\dots,c_j)
=q^{c_j-1}(\zeta_{i-1,j}\zeta_{i,j+1}-{\mathcal R}\zeta_{i,j})=q^{\sum\limits_{u=i}^{j}(c_u-1)}\,. 
\end{align*}
Thus $\zeta_{i-1,j}\zeta_{i,j+1}-{\mathcal R}\zeta_{i,j}=q^{\sum\limits_{u=i}^{j-1}(c_u-1)}$. 
By the $q$-unimodular rule, 
$\zeta_{i-1,j}\zeta_{i,j+1}-\zeta_{i-1,j+1}\zeta_{i,j}=q^{\sum\limits_{u=i}^{j-1}(c_u-1)}$, 
we obtain ${\mathcal R}=\zeta_{i-1,j+1}$ since $\zeta_{i,j}=S=r'$ is a nonzero polynomial.
Moreover, $\lim\limits_{q\to1}\zeta_{i-1,j+1}=\sigma_{i-1,j+1}>0$ 
for $j\leq i+n-3$ by Proposition \ref{prop:3desect_cont_frac_CCF}, 
so $\zeta_{i-1,j+1}$ is nonzero in this case, 
whereas $\zeta_{i-1,j+1}=0$ if $j=i+n-2$ by 
\eqref{eqn:q-mat_n-2} of Lemma~\ref{lemm:n-1_q}. 
Therefore, the result (1) holds.
\par
By \eqref{eqn:q_shucho3}, 
we have
\begin{align*}
&[[c_i,\dots,c_{j-1}]]_q=\dfrac{R}{S}\,,\\
&[[c_i,\dots,c_j]]_q
=[c_i]-\dfrac{q^{c_i-1}}{[[c_{i+1},\dots,c_j]]_q}
%=[c_i]-\dfrac{\;q^{c_i-1}\;}{\dfrac{r}{s}}
=[c_i]-\frac{\;q^{c_i-1}s\;}{r}
=\dfrac{\;[c_i]r-q^{c_i-1}s\;}{r}\,. 
\end{align*}
Thus, the result (2) follows. 
\end{proof}
The following Proposition \ref{prop:q-CCF_n-1} extends 
Proposition \ref{prop:q_CCF_q_rational} (1) and (2) to the case $j=i+n-1$. 
\begin{prop}\label{prop:q-CCF_n-1} 
The equation in Proposition \ref{prop:q_CCF_q_rational} (1) remains valid for 
$j=i+n-1$. 
\end{prop}
\begin{proof} 
Let $M_q(c_i,c_{i+1},\dots,c_{i+n-1})=
\begin{sumipmatrix}
{\mathcal R}&-q^{c_{i-1}-1}{\mathcal R}'\\
{\mathcal S}&-q^{c_{i-1}-1}{\mathcal S}'
\end{sumipmatrix}$. 
From the equation \eqref{eqn:q-mat_n-1} in Proposition \ref{lemm:n-1_q}, we have
$M_q(c_i,c_{i+1},\dots,c_{i+n-1})
=-q^{n-3}
\begin{sumipmatrix}
1&0\\
0&1
\end{sumipmatrix}$. Therefore, we obtain 
\begin{align*}
\begin{pmatrix}
{\mathcal R}&{\mathcal R}'\\
{\mathcal S}&{\mathcal S}'
\end{pmatrix}
=
\begin{pmatrix}
-q^{n-3}&0\\
0&q^{n-2-c_{i-1}}
\end{pmatrix}\,.
\end{align*}
From the equation \eqref{eqn:q-mat_n-2} in Proposition \ref{lemm:n-1_q}, we have 
\begin{align*}
&M_q(c_i,c_{i+1},\dots,c_{i+n-2})
=q^{n-2-c_{i-1}}
\begin{sumipmatrix}
0&-q^{c_{i-1}-1}\\
1&-[c_{i-1}]
\end{sumipmatrix}\,,\\
&M_q(c_{i+1},c_{i+2},\dots,c_{i+n-1})
=q^{n-2-c_i}
\begin{sumipmatrix}
0&-q^{c_i-1}\\
1&-[c_i]
\end{sumipmatrix}\,.
\end{align*}
From Proposition \ref{prop:q_CCF_q_rational} (1) in the case $j=i+n-2$ for the triangulation 
$(\T,i)$ or $j=i+n-1$ for the triangulation $(\T,i+1)$, it follows that
\begin{align*}
\zeta_{i-1,i+n-1}=0,\quad\zeta_{i,i+n-1}=q^{n-2-c_{i-1}},\quad\zeta_{i,i+n}=0\,.
\end{align*}
By these equations and the $q$-unimodular rule, 
$\zeta_{i-1,i+n-1}\zeta_{i,i+n}-\zeta_{i-1,i+n}\zeta_{i,i+n-1}
=q^{\sum\limits_{u=i}^{i+n-2}(c_u-1)}$, we have 
$\zeta_{i-1,i+n}=-q^{\sum\limits_{u=i}^{i+n-1}(c_u-1)-(n-3)}$. 
Since $\sum\limits_{u=i}^{i+n-1}(c_u-1)=2(n-3)$, 
$\zeta_{i,i+n-1}=-q^{n-3}$ holds. 
Hence, the equation in Proposition \ref{prop:q_CCF_q_rational} (1) holds for $j=i+n-1$.  
\par
From Proposition \ref{prop:q_CCF_q_rational} (2) in the case $j=i+n-1$ for the triangulation 
$(\T,i+1)$, we have 
$[[c_{i+1},\dots,c_{i+n-1}]]_q=\dfrac{\zeta_{i,i+n}}{\zeta_{i+1,i+n}}=
\dfrac{0}{q^{n-2-c_{i+n}}}=\dfrac{0}{q^{n-2-c_i}}$. 
Here, we compute $[[c_i,c_{i+1},\dots,c_{i+n-1}]]_q$ as follows: 
\begin{align*}
[[c_i,c_{i+1},\dots,c_{i+n-1}]]_q
=[c_i]_q-\dfrac{q^{c_i-1}}{\quad\dfrac{0}{q^{n-2-c_i}}\quad}
&=[c_i]_q-\dfrac{q^{c_i-1}\cdot q^{n-2-c_i}}{0}
=\dfrac{[c_i]_q\cdot0-q^{n-3}}{0}
=\dfrac{-q^{n-3}}{0}\,.
\end{align*}
Since $\zeta_{i,i+n}=0$ and $\zeta_{i-1,i+n}=-q^{n-3}$, 
we have $[[c_i,c_{i+1},\dots,c_{i+n-1}]]_q=\dfrac{\zeta_{i,i+n-1}}{\zeta_{i,i+n}}$. 
Therefore, the equations in Proposition \ref{prop:q_CCF_q_rational} (2) hold for 
$j=i+n-1$. 
\end{proof}
From Proposition \ref{prop:q_CCF_q_rational} (1) and Proposition \ref{prop:q-CCF_n-1}, 
for the $q$-CCF $\CCF_q(\T)=\{\zeta_{k,\ell}\}_{D_{n+1}}$ associated with a triangulation $\T$ 
of an $n$-gon and $i,j\in\Z$ with $i\leq j\leq i+n-1$, 
we have
\begin{align*}
&M_q(c_i,c_{i+1},\dots,c_j)=
\begin{pmatrix}
\zeta_{i-1,j+1}&-q^{c_j-1}\zeta_{i-1,j}\\
\zeta_{i,j+1}&-q^{c_j-1}\zeta_{i,j}
\end{pmatrix}\,,\\
&M_q(c_i,c_{i+1},\dots,c_j)=M_q(c_i,c_{i+1},\dots,c_{j-1})M_q(c_j)
=\begin{pmatrix}
\zeta_{i-1,j}&-q^{c_{j-1}-1}\zeta_{i-1,j-1}\\
\zeta_{i,j}&-q^{c_{j-1}-1}\zeta_{i,j-1}
\end{pmatrix}
\begin{pmatrix}
[c_j]&-q^{c_j-1}\\
1&0
\end{pmatrix}\,.
\end{align*}
Therefore, we obtain the recurrence relation of $\zeta_{k,\ell}$ with $(k,\ell)\in D_{n+1}$ as follows:
\begin{align*}
\zeta_{i-1,j+1}=[c_j]\zeta_{i-1,j}-q^{c_{j-1}-1}\zeta_{i-1,j-1}\,.
\end{align*}
From the same propositions, we further obtain the following corollary. 
This is a $q$-analogue of Conway and Coxeter's theorem \cite[Problem 28]{conway1973triangulated,conway1973triangulated_conti} formulated 
in Theorem \ref{thm:CCF} (1). 
\begin{cor}\label{cor:q_CCF_1} 
For any triangulation of an $n$-gon, the associated  $q$-CCF $\CCF_q(\T)$ 
is a positive $q$-CCF of width $n-1$. 
\end{cor}
\vskip.2cm
\begin{example}\label{ex:q_CCF}
Figure \ref{fig:q_CCF} shows the $q$-CCF $\CCF_q(\T)$ associated with the triangulation $\T$ considered in Example \ref{ex:triangular}, where $\langle5\rangle,\;\{5\},\;
\langle7\rangle,\;\{7\},\;\langle8\rangle,\;\langle11\rangle\in\Z[q]$ are defined by 
\begin{align}
&\langle5\rangle=1+q+2q^2+q^3\,,&&\{5\}=1+2q+q^2+q^3\,,\nonumber\\
&\langle7\rangle=1+q+2q^2+2q^3+q^4\,,&&\{7\}=1+2q+2q^2+q^3+q^4\,,\label{eqn:modified_q-int}\\
&\langle8\rangle=1+2q+2q^2+2q^3+q^4\,,&&\langle11\rangle=1+2q+2q^2+3q^3+2q^4+q^5\,.
\nonumber
\end{align}
%\newpage
\begin{figure}[htbp]
\begin{minipage}{\textwidth}
\begin{center}
\def\arraystretch{1.3}
\begin{tabularx}{15cm}{XXXXXXXXXXXXXXXXXXXX}
&$0$& &$0$& &$0$& &$0$&&$0$&&$0$&&$0$&&$0$&&$0$&&$0$\\
$1$& &$1$& &$1$& &$1$&&$1$&&$1$&&$1$&&$1$&&$1$&&$1$&\\
&$1$& &$[4]$& &$[2]$& &$1$&&$[3]$&&$[4]$&&$1$&&$[2]$&&$[3]$& & $1$\\
$[2]$&&$q[3]$&&$\{7\}$&&$1$&&$q[2]$&&$\langle11\rangle$&&$[3]$&&$q$&&$\langle5\rangle$&&$[2]$\\
&$q\langle5\rangle$&&$q\{5\}$&&$[3]$&&$q^2$&&$q\langle7\rangle$&&$\langle8\rangle$&&$q[2]$&&$q^2[2]$&&$[3]$&&$q\langle5\rangle$\\
$q\langle7\rangle$&&$q\langle8\rangle$&&$q[2]$&&$q^2[2]$&&$q^3[3]$&&$q\langle5\rangle$&&$q\{5\}$&&$q^2[3]$&&$q^2$&&$q\langle7\rangle$&\\
&$q\langle11\rangle$&&$q[3]$&&$q^3$&&$q^3\langle5\rangle$&&$q^3[2]$&&$q^2[3]$&&$q^2\{7\}$&&$q^2$&&$q^4[2]$&&$q\langle11\rangle$\\
$q^4[3]$&&$q[4]$&&$q^3$&&$q^5[2]$&&$q^3[3]$&&$q^4$&&$q^3[4]$&&$q^2[2]$&&$q^5$&&$q^4[3]$&\\
&$q^4$&&$q^3$&&$q^6$&&$q^5$&&$q^4$&&$q^6$&&$q^3$&&$q^5$&&$q^6$&&$q^4$\\
$0$& &$0$& &$0$& &$0$&&$0$&&$0$&&$0$&&$0$&&$0$&&$0$&\\
&$-q^6$& &$-q^6$& &$-q^6$& &$-q^6$&&$-q^6$&&$-q^6$&&$-q^6$&&$-q^6$&&$-q^6$&&$-q^6$
\end{tabularx}
\def\arraystretch{1}
\vskip.5cm
\caption{$q$-CCF associated with the triangulation $\T$ with 
$\quid(\T)=(\dots,1,4,2,1,3,4,1,2,3,\dots)$}\label{fig:q_CCF}
\end{center}
\end{minipage}
\end{figure}
\par
%\newpage
For example, $q$-continued fraction $[[c_3,c_4,c_5,c_6,c_7]]_q=[[1,3,4,1,2]]_q$ is computed as follows: 
\begin{align*}
[[1,3,4,1,2]]_q=1-\dfrac{1}
{[3]-\dfrac{q^2}{
[4]-\dfrac{q^3}{
1-\dfrac{1}
{[2]}}
}}
=1-\dfrac{1}
{[3]-\dfrac{q^2}{
[4]-\dfrac{q^3}{\quad
\dfrac{q}{[2]}\quad
}
}}
=1-\dfrac{1}{[3]-\dfrac{q^2}{\quad\dfrac{q[2]}{q}\quad}}
%&=1-\dfrac{1}{\quad\dfrac{q\{5\}}{q[2]}}\\
%&
=\cdots=\dfrac{q^2[3]}{q\{5\}}\,.
\end{align*}
By Figure \ref{fig:q_CCF}, 
$\zeta_{2,8}=q^2[3]$, $\zeta_{3,8}=q\{5\}$, and therefore, 
$[[c_3,c_4,c_5,c_6,c_7]]_q=[[1,3,4,1,2]]_q=\dfrac{\zeta_{2,8}}{\zeta_{3,8}}$. 
\end{example}
\subsection{\texorpdfstring{$q$}{\it q}-Farey labelings}\label{sect:q_Farey}
In this section, we define a $q$-analogue of the Farey labeling defined 
in Definition \ref{def:Farey_graph}, following Morier-Genoud and Ovsienko \cite[Definition 2.7, Definition 2.9]{morier2020q}. 
\begin{defi}[$q$-Farey sum {\cite[Definition 2.7, Definition 2.9]{morier2020q}}]\label{def:weighted_Farey_sum}
Let $\PP^1(\Q(q))=\Q(q)\cup\left\{\dfrac{1}{0}\right\}$. 
For two rational functions $\theta_{\alpha}=\dfrac{r_{\alpha}}{s_{\alpha}}$, 
$\theta_{\beta}=\dfrac{r_{\beta}}{s_{\beta}}\in\PP^1(\Q(q))$, 
suppose that the weight $w(\theta_{\alpha},\theta_{\beta})$ of the pair $(\theta_{\alpha},\theta_{\beta})$ 
is given by $w(\theta_{\alpha},\theta_{\beta})=q^{\ell-1}$ for some $\ell\in\Z$. 
The $q$-Farey sum of $\theta_{\alpha}$ and $\theta_{\beta}$ is defined by
\begin{align}\label{eqn:weighted_Farey_sum}
\theta_{\alpha}\oplus_q \theta_{\beta}=\dfrac{r_{\alpha}+q^\ell r_{\beta}}{s_\alpha+q^\ell s_\beta}\,.
\end{align}
For $\theta_{\gamma}=\theta_{\alpha}\oplus_q \theta_{\beta}$, 
the weights $w(\theta_{\gamma},\theta_{\alpha})$, 
$w(\theta_{\beta},\theta_{\gamma})$ 
of $(\theta_{\gamma},\theta_{\alpha})$, 
$(\theta_{\beta},\theta_{\gamma})$ are defined by
\begin{align*}
w(\theta_{\gamma},\theta_{\alpha})=1\,,\quad 
w(\theta_{\beta},\theta_{\gamma})=q^\ell\,,
\end{align*}
respectively. The triple 
$(w(\theta_{\alpha},\theta_{\beta}),\;w(\theta_{\gamma},\theta_{\alpha}),\;w(\theta_{\beta},\theta_{\gamma}))=(q^{\ell-1},1,q^{\ell})$ is called the weight triple of 
$(\theta_{\alpha},\theta_{\beta},\theta_{\gamma})$. 
\end{defi}
\begin{defi}[$q$-Farey labelings {\cite[Definition 2.7, Definition 2.9]{morier2020q}}]\label{def:weighted_Farey_graph}
Let $(\T,i)$ be a triangulation of an $n$-gon. 
As in Definition \ref{def:Farey_graph}, for each vertex $v_j$ with $i\leq j\leq i+n-1$, 
we define a rational function $\theta_{i,j}\in\PP^1(\Q(q))$ 
by replacing the Farey sum with the $q$-Farey sum. 
For a cell ${\mathcal C}=v_{\alpha}v_{\beta}v_{\gamma}$ with the base side 
$v_{\alpha}v_{\beta}$, 
we identify the weight triple of 
$(\theta_{i,\alpha},\theta_{i,\beta},\theta_{i,\gamma})$ with 
the side-weight triple 
$(w(v_{\alpha}v_{\beta}),\,w(v_{\gamma}v_{\alpha}),\,
w(v_{\beta}v_{\gamma}))$ and denote it by $w({\mathcal C})$.  
\par
The labels of the endpoints of the base edge $v_{i+n-1}v_i$, are defined by 
$\theta_{i,i+n-1}=\dfrac{0}{1}$ and $\theta_{i,i}=\dfrac{1}{0}$ and 
we set $w(v_{i+n-1}v_i)=q^{-1}$. 
If ${\mathcal C}=v_{i+n-1}v_iv_u$ is the base cell, then 
\begin{align*}
\theta_{i,u}=\theta_{i,i+n-1}\oplus_q \theta_{i,i}=\dfrac{0+1}{1+0}=\dfrac{1}{1}\;,\quad 
w({\mathcal C})=(q^{-1},1,1)\,.
\end{align*}
Suppose that the vertices of a cell have been labeled by rational functions. 
Let $v_{\alpha}v_{\beta}$ be the base side of its child cell 
${\mathcal C}'=v_{\alpha}v_{\beta}v_{\gamma}$ 
with weight $w(v_{\alpha}v_{\beta})=q^{\ell-1}$. 
If $\theta_{i,\alpha}=\dfrac{r_{i,\alpha}}{s_{i,\alpha}}$, 
$\theta_{i,\beta}=\dfrac{r_{i,\beta}}{s_{i,\beta}}$, 
then the label $\theta_{i,\gamma}$ 
of the remaining vertex $v_{\gamma}$ is defined by 
\begin{align*}
\theta_{i,\gamma}=\theta_{i,\alpha}\oplus_q \theta_{i,\beta}
=\dfrac{r_{i,\alpha}+q^\ell r_{i,\beta}}{s_{i,\alpha}+q^\ell s_{i,\beta}}\,,
\end{align*}
where fractions are not reduced, and 
we set $w({\mathcal C}')=(q^{\ell-1},1,q^{\ell})$ 
(see Figure \ref{fig:wedge_weight}). 
\par
The sequence $(\theta_{i,j})_{i\leq j\leq i+n-1}$ is called 
the $q$-Farey labeling of $(\T,i)$ and is denoted by $\Farey_q(\T,i)$. 
Its anti-periodic extension 
$\Farey_q(\T,i)^{\per}=(\theta_{i,u})_{u\in\Z}$ is 
defined by 
\begin{align}
\theta_{i,u}=\dfrac{r_{i,u}}{s_{i,u}}=\dfrac{(-1)^{t}r_{i,j}}{(-1)^ts_{i,j}}
\label{eqn:q_Farey_graph_infty}
\end{align}
for any $u\in\Z$ with $u=j+tn$ and $i\leq j\leq i+n-1$. 
\end{defi}
\par
\begin{figure}[htbp]
\begin{center}
\begin{tikzpicture}[scale=.55]
% 点の定義
\coordinate(v1)at(0,0);
\coordinate(v2)at(2,-5);
\coordinate(v3)at(4,0);
\filldraw [black] (v1) circle (1pt) node[above]{$v_{\alpha}\,(\theta_{i,\alpha})$};
\filldraw [black] (v2) circle (1pt) node[below]{$\quad v_{\gamma}\,(\theta_{i,\gamma})$};
\filldraw [black] (v3) circle (1pt) node[above]{$\quad v_{\beta}\,(\theta_{i,\beta})$};
\draw (v1)--(v2)--(v3)--(v1);
%\draw (v3)--(v4);
%\draw[dashed] (v4)--(v5);
%\draw (v1)--(v7);
%\draw[dashed] (v7)--(v8);
\node at(2,-2) {${\mathcal C}$};
\node at(2.2,.6) {$q^{\ell-1}$};
\node at(.5,-2.5) {$1$};
\node at(3.5,-2.5) {$q^{\ell}$};
\end{tikzpicture}
\caption{Side-weight triple of a cell ${\mathcal C}$}\label{fig:wedge_weight}
\end{center}
\end{figure}
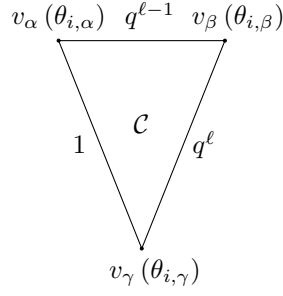
\par
\newpage
\begin{example}\label{ex:q_CCF_Farey_graph}
Consider the triangulation in Example \ref{ex:triangular},  
we take the edge $v_2v_3=v_{11}v_3$ as the base edge. 
The $q$-Farey labeling 
$\Farey_q(\T,3)=(\theta_{3,j})_{3\leq j\leq 11}$ of $(\T,3)$ is given 
in Figure \ref{fig:q_CCF_Farey_graph0}. 
The polynomials $\langle5\rangle,\;\{5\},\;
\langle7\rangle,\;\{7\},\;\langle8\rangle,\;\langle11\rangle\in\Z[q]$ are 
defined in $(\ref{eqn:modified_q-int})$. 
Weights are indicated beside each side. 
\par
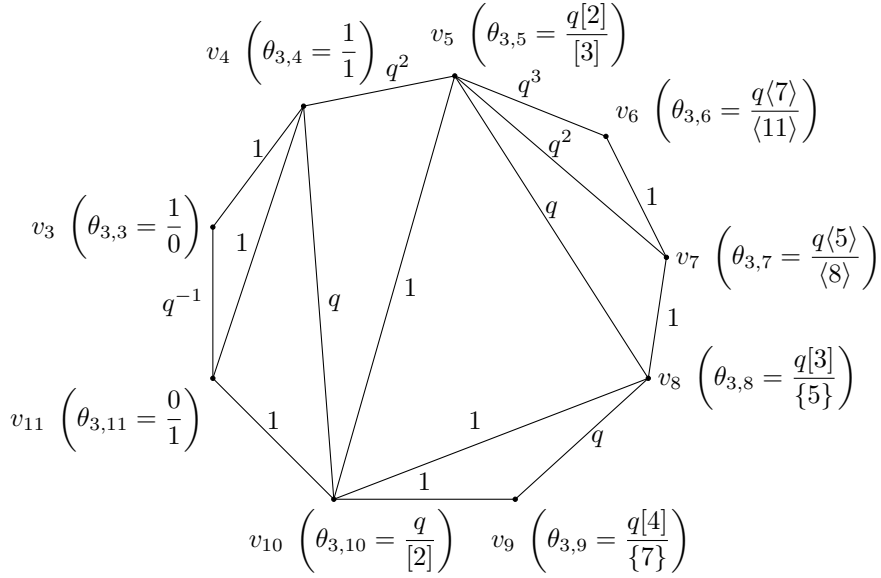
\begin{figure}[htbp]
\begin{center}
\begin{tikzpicture}[scale=.8]
\coordinate(v0)at(0,0);
\coordinate(v1)at(-3,0);
\coordinate(v2)at(-5,2);
\coordinate(v3)at(-5,4.5);
\coordinate(v4)at(-3.5,6.5);
\coordinate(v5)at(-1,7);
\coordinate(v6)at(1.5,6);
\coordinate(v7)at(2.5,4);
\coordinate(v8)at(2.2,2);
\filldraw [black] (v0) circle (1pt) node[below]{
\hspace{2cm}$v_9\;\left(\theta_{3,9}=\dfrac{q[4]}{\{7\}}\right)$};
\filldraw [black] (v1) circle (1pt) node[below]{
\hspace{.5cm}$v_{10}\;\left(\theta_{3,10}=\dfrac{q}{[2]}\right)$};
\filldraw [black] (v2) circle (1pt) node[below left]{
\hspace{1cm}$v_{11}\;\left(\theta_{3,11}=\dfrac{0}{1}\right)$};
\filldraw [black] (v3) circle (1pt) node[left]{
\hspace{1cm}$v_3\;\left(\theta_{3,3}=\dfrac{1}{0}\right)$};
\filldraw [black] (v4) circle (1pt) node[above]{
\raisebox{.5cm}{\hspace{-.3cm}$v_4\;\left(\theta_{3,4}=\dfrac{1}{1}\right)$}};
\filldraw [black] (v5) circle (1pt) node[above]{
\hspace{2cm}$v_5\;\left(\theta_{3,5}=\dfrac{q[2]}{[3]}\right)$};
\filldraw [black] (v6) circle (1pt) node[right]{
\raisebox{1cm}{$v_6\;\left(\theta_{3,6}=\dfrac{q\langle7\rangle}{\langle11\rangle}\right)$}};
\filldraw [black] (v7) circle (1pt) node[right]{
$v_7\;\left(\theta_{3,7}=\dfrac{q\langle5\rangle}{\langle8\rangle}\right)$};
\filldraw [black] (v8) circle (1pt) node[right]{
$v_8\;\left(\theta_{3,8}=\dfrac{q[3]}{\{5\}}\right)$};
\draw (v0)--(v1) node[midway, above]{$1$};
\draw (v1)--(v2) node[midway, above]{$1$};
\draw (v1)--(v4) node[midway, right]{$q$};
\draw (v1)--(v5)node[midway, right]{$1$};
\draw (v1)--(v8)node[midway, above left]{$1$};
\draw (v2)--(v3)node[midway, left]{$q^{-1}$};
\draw (v2)--(v4)node[midway, left]{$1$};
\draw (v3)--(v4)node[midway, above]{$1$};
\draw (v4)--(v5)node[midway, above]{$\hspace{.5cm}q^2$};
\draw (v5)--(v6)node[midway, above]{$q^3$};
\draw (v5)--(v7)node[midway, above]{$q^2$};
\draw (v5)--(v8)node[midway, above]{$q$};
\draw (v6)--(v7)node[midway, right]{$1$};
\draw (v7)--(v8)node[midway, right]{$1$};
\draw (v8)--(v0)node[midway, right]{$q$};
\end{tikzpicture}
\caption{$q$-Farey labeling of the triangulation $(\T,3)$ 
with quiddity
$\quid(\T)=(\dots,1,4,2,1,3,4,1,2,3,
\protect\\
\dots)$
}\label{fig:q_CCF_Farey_graph0}
\end{center}
\end{figure}
\end{example}
The following Proposition \ref{prop:weighted_Farey_graph_degree} 
establishes positivity of the coefficients, as well as unit trailing and leading coefficients 
for the numerators and denominators of $\theta_{i,j}$ in the 
$q$-Farey labeling. 
Later, in Theorem \ref{prop:q_CCF_weighted_Farey}, we show that 
the entries of the $q$-CCF coincide with the numerators and denominators of $\theta_{i,j}$ up to explicit powers of $q$. 
As shown in Propositions \ref{prop:q_CCF_q_rational} 
and \ref{prop:q-CCF_n-1}, the entries of the $q$-CCF 
are precisely the numerators and denominators of the associated $q$-continued fractions. 
Consequently, these numerators and denominators share the same degree and coefficient properties 
as those of $\theta_{i,j}$ (see Proposition \ref{prop:coprime} (2)). 
Morier-Genoud and Ovsienko 
\cite[Proposition 1.6 (i)--(iii), Corollary 1.7 (ii), (iii)]{morier2020q} 
proved analogous results for continued fractions $\dfrac{r}{s}=[[c_1,c_2,\dots,c_k]]$ 
with $c_u\geq 2$, corresponding to triangulations with exactly two exterior cells. 
Our results extend these correspondences 
to arbitrary subsequences of quiddities arising from general triangulations.
\par
Moreover, Morier-Genoud and Ovsienko \cite[Proposition 1.6 (iv)]{morier2020q} showed that 
the numerator and denominator of $\theta_{i,j}$ in the $q$-Farey labeling are coprime. 
The same coprimality property holds for arbitrary triangulations.  
Furthermore, the greatest common divisor 
of the corresponding pairs of entries of the $q$-CCF is a power of $q$. 
These results are collected in Proposition \ref{prop:coprime} (1). 
\par
Let $\deg_{\min}(r)$ and $\deg(r)$ denote the trailing and leading degrees of a polynomial 
$r\in\Z[q]$, respectively.  
\begin{prop}\label{prop:weighted_Farey_graph_degree} 
Let $\theta_{i,j}=\dfrac{r_{i,j}}{s_{i,j}}$ be a rational function 
in the $q$-Farey labeling of a triangulation $(\T,i)$ of an $n$-gon. 
For any $j\in\Z$ with $i+1\leq j\leq i+n-2$, the following hold: 
\begin{enumerate}[{(}1{)}]
\item The polynomials $r_{i,j}$ and $s_{i,j}$ have trailing and leading coefficients equal to $1$, 
and all intermediate coefficients are positive.  
\item $\deg_{\min}(s_{i,j})=0$.  
\end{enumerate}
\end{prop}
\noindent
Note that, in the above proposition,  we exclude the cases $j=i$ and $j=i+n-1$, 
since these cases imply trivial results; that is 
$\theta_{i,i}=\dfrac{1}{0}$ ($r_{i,i}=1,s_{i,i}=0$) if $j=i$, 
and $\theta_{i,i+n-1}=\dfrac{0}{1}$ ($r_{i,i+n-1}=0,s_{i,i+n-1}=1$) if $j=i+n-1$. 
\begin{proof} We abbreviate $\theta_{i,j}$, $r_{i,j}$, and $s_{i,j}$ as $\theta_j$, 
$r_j$, and $s_j$, respectively. 
\par
First, suppose that $j=i+n-2$. Following Definition \ref{def:desc_deg}, 
let $\cell(v_{i+n-1})=\{\widetilde{\mathcal C}_0,
\widetilde{\mathcal C}_1,\dots,\widetilde{\mathcal C}_{c_{i-1}-1}\}$. 
The cell $\widetilde{\mathcal C}_{c_{i-1}-1}$ is the base cell 
with side-weight triple $(q^{-1},1,1)$, while the other cells 
have side-weight triple $(1,1,q)$. 
The labeling procedure for the $q$-Farey labeling 
is identical to that for the classical Farey labeling, except for the weights. 
Thus, by induction along these cells, the vertex 
$v_{i+n-2}$ is labeled by $\theta_{i+n-2}=\dfrac{q^{c_{i-1}-1}}{[c_{i-1}]}$. 
Hence, the results (1) and (2) hold. 
\par
Next, suppose $j=i+1$. 
Let $\cell(v_i)=\{{\mathcal C}_0,{\mathcal C}_1,\dots,{\mathcal C}_{c_i-1}\}$. 
The cell ${\mathcal C}_0$ is the base cell, and 
the other cell ${\mathcal C}_b$ with $1\leq b\leq c_i-1$ has the side-weight triple 
$(q^{b-1},1,q^b)$. Similarly, by induction along these cells, 
the vertex $v_{i+1}$ is labeled by $\theta_{i+1}=\dfrac{[c_i]}{1}$, 
so the results (1) and (2) hold. 
\par
The remaining vertices $v_j$ lie in descendant cells whose ancestor cell is 
$\widetilde{\mathcal C}_{a}$ with $0\leq a\leq c_{i-1}-2$ or 
${\mathcal C}_{b}$ with $1\leq b\leq c_i-1$.  
In each case, the labeling of the new vertex is obtained 
from two previously labeled vertices via a cell with side-wight triple 
$(q^{\ell-1},1,q^{\ell})$ for some $\ell\geq1$. 
Therefore, by Lemma \ref{lem:wFg_deg_posi}, 
the results (1) and (2) hold for all descendant cells. 
This completes the proof of Proposition \ref{prop:weighted_Farey_graph_degree}. 
\par
For each rational function $\theta_{j}=\dfrac{r_j}{s_j}$, we denote 
$m_j=\deg_{\min}(r_j)$, $d_j=\deg(r_j)$, $m'_j=\deg_{\min}(s_j)$, and $d'_j=\deg(s_j)$.  
\begin{lemm}\label{lem:wFg_deg_posi} 
If the numerators and denominators of $\theta_{\alpha}$ and $\theta_{\beta}$ for the vertices 
$v_{\alpha}$ and $v_{\beta}$ in the $q$-Farey labeling satisfy conditions (1) and (2) 
of Proposition \ref{prop:weighted_Farey_graph_degree}, 
and if the cell ${\mathcal C}=v_{\alpha}v_{\beta}v_{\gamma}$ 
with base side $v_{\alpha}v_{\beta}$ has 
the side-weight triple $w({\mathcal C})=(q^{\ell-1},1,q^{\ell})$ $(\ell\geq 1)$ such that
\begin{align}\label{eqn:wFg_deg_posi1}
m_{\alpha}\leq m_{\beta}+\ell-1\leq d_{\alpha}\leq d_{\beta}+\ell-1\,,\quad 
\ell-1\leq d'_{\alpha}\leq d'_{\beta}+\ell-1, 
\end{align}
then $\theta_{\gamma}$ also satisfies conditions (1) and (2), and  
\begin{align}\label{eqn:wFg_deg_posi2}
&m_{\alpha}\leq m_{\gamma}\leq d_{\alpha}\leq d_{\gamma}\,,\quad 
0\leq d'_{\alpha}\leq d'_{\gamma}\,,\\
&m_{\gamma}\leq m_{\beta}+\ell\leq d_{\gamma}\leq d_{\beta}+\ell\,,\quad 
\ell\leq d'_{\gamma}\leq d'_{\beta}+\ell\,. \nonumber
\end{align}
Note that, since the child cells with the base edges $v_{\alpha}v_{\gamma}$ 
(resp.~$v_{\gamma}v_{\beta}$) have side-weight triples $(1,1,q)$ 
(resp.~$(q^\ell,1,q^{\ell+1})$), the inequalities $(\ref{eqn:wFg_deg_posi2})$ imply 
that these child cells satisfy the same inequalities as in $(\ref{eqn:wFg_deg_posi1})$.
\end{lemm}
\par
\noindent
\begin{proof}[Proof of Lemma {\rm\ref{lem:wFg_deg_posi}}]
Since the cell ${\mathcal C}=v_{\alpha}v_{\beta}v_{\gamma}$ has the side-weight triple 
$(q^{\ell-1},1,q^{\ell})$, $\theta_{\gamma}$ is given by 
\begin{align*}
\theta_{\gamma}=\theta_{\alpha}\oplus_q \theta_{\beta}=\dfrac{r_{\alpha}+q^{\ell}r_{\beta}}{s_{\alpha}+q^{\ell}s_{\beta}}\,.
\end{align*}
The inequality $(\ref{eqn:wFg_deg_posi1})$ implies 
$m_{\alpha}+1\leq m_{\beta}+\ell$ and $d_{\alpha}+1\leq d_{\beta}+\ell$. Thus we have
\begin{align*}
&m_{\gamma}=\deg_{\min}r_{\gamma}=\deg_{\min}(r_{\alpha}+q^{\ell}r_{\beta})=
\min\{m_{\alpha},m_{\beta}+\ell\}=m_{\alpha}\,,\\
&d_{\gamma}=\deg r_{\gamma}=\deg(r_{\alpha}+q^{\ell}r_{\beta})=
\max\{d_{\alpha},d_{\beta}+\ell\}=d_{\beta}+\ell\,.
\end{align*}
Therefore, the trailing and leading coefficients of $r_{\gamma}$ 
are equal to the trailing coefficient of $r_{\alpha}$ and 
the leading coefficient of $q^{\ell}r_{\beta}$, respectively. 
From condition (1), these coefficients are equal to $1$. 
By the inequality $m_{\beta}+\ell \leq d_{\alpha}+1$ in $(\ref{eqn:wFg_deg_posi1})$, 
all intermediate coefficients of $r_{\gamma}$ are positive. 
Hence, the numerator $r_{\gamma}$ satisfies condition (1). 		
\par
Using \eqref{eqn:wFg_deg_posi1} together with 
the above equations $m_{\gamma}=m_{\alpha}$, $d_{\gamma}=d_{\beta}+\ell$, 
it follows that  
\begin{align*}
&m_{\gamma}=m_{\alpha}\leq d_{\alpha}\leq d_{\beta}+\ell-1\leq d_{\beta}+\ell=d_{\gamma}\,,\\
&m_{\gamma}=m_{\alpha}\leq m_{\beta}+\ell-1\leq m_{\beta}+\ell 
\leq d_{\beta}+\ell=d_{\gamma}\,.
\end{align*}
Therefore, the inequalities 
for $d_{\alpha},d_{\beta},d_{\gamma},m_{\alpha},m_{\beta}$, and $m_{\gamma}$ 
in $(\ref{eqn:wFg_deg_posi2})$ follow. 
\par
The argument for the denominator $s_{\gamma}$ is analogous to that 
for the numerator $r_{\gamma}$. 
Indeed, using condition (2), $m'_{\alpha}=m'_{\beta}=0$, 
and \eqref{eqn:wFg_deg_posi1}, we obtain
\begin{align*}
m'_{\gamma}=m'_{\alpha}=0\,,\quad d'_{\gamma}=d'_{\beta}+\ell,
\end{align*}
and the same reasoning as above shows that $s_{\gamma}$ satisfies conditions (1) and (2).
Moreover, the inequalities for $d'_{\alpha},d'_{\beta},d'_{\gamma}$ in
\eqref{eqn:wFg_deg_posi2} follow in the same way.
\if0
The inequality $(\ref{eqn:wFg_deg_posi1})$ implies 
$d'_{\alpha}+1\leq d'_{\beta}+\ell$ and, by condition (2), $m_\alpha=m_\beta=0$. 
Thus we have
\begin{align*}
&m'_{\gamma}=\deg_{\min}s_{\gamma}=\deg_{\min}(s_{\alpha}+q^{\ell}s_{\beta})=
\min\{m_{\alpha},m_{\beta}+\ell\}=m_{\alpha}=0\,,\\
&d'_{\gamma}=\deg s_{\gamma}=\deg(s_{\alpha}+q^{\ell}s_{\beta})=
\max\{d'_{\alpha},d'_{\beta}+\ell\}=d'_{\beta}+\ell\,.
\end{align*}
Therefore, the trailing and leading coefficients of $s_{\gamma}$ 
are equal to the trailing coefficient of $s_{\alpha}$ 
and the leading coefficient of $q^{\ell}s_{\beta}$. 
From condition (1), these coefficients are equal to $1$. 
By the inequality $\ell \leq d'_{\alpha}+1$ in \eqref{eqn:wFg_deg_posi1}, 
all intermediate coefficients of $s_{\gamma}$ are positive. Hence, the denominator 
$s_{\gamma}$ satisfies the condition (1). 
From the above equation $m'_{\gamma}=0$, $s_{\gamma}$ also satisfies the condition (2). 
\par
Moreover, using \eqref{eqn:wFg_deg_posi1} together with 
the above equations, $m'_{\gamma}=m_{\alpha}=0$, $d'_{\gamma}=d'_{\beta}+\ell$, 
it follows that
\begin{align*}
&0\leq \ell-1\leq d'_{\alpha}\leq d'_{\beta}+\ell-1\leq d'_{\beta}+\ell=d'_{\gamma}\,\\
&\ell\leq d'_{\alpha}+1\leq d'_{\beta}+\ell=d'_{\gamma}\,.
\end{align*}
Hence, the inequalities of $d'_{\alpha},d'_{\beta},d'_{\gamma}$ in 
\eqref{eqn:wFg_deg_posi2} hold. 
\fi
\end{proof}
\par
This completes the proof of Proposition \ref{prop:weighted_Farey_graph_degree}. 
\end{proof}
\subsection{Correspondence between 
the \texorpdfstring{$q$}{\it q}-CCF and the \texorpdfstring{$q$}{\it q}-Farey labeling}
Morier-Genoud and Ovsienko \cite[Theorem 3]{morier2020q} 
showed that, for triangulations with exactly two exterior cells, 
the numerators and denominators of 
$q$-continued fractions coincide with those appearing 
in the associated $q$-Farey labeling. 
For general triangulations, however, this exact correspondence no longer holds. 
As shown in Propositions \ref{prop:q_CCF_q_rational} 
and \ref{prop:q-CCF_n-1}, the entries of the $q$-CCF 
recover the numerators and denominators of the associated $q$-continued fractions. 
The following theorem shows that these entries differ from the corresponding polynomials 
in the $q$-Farey labeling only by explicit powers of $q$, determined by descendant degrees.
\begin{thm}\label{prop:q_CCF_weighted_Farey} 
Let $(\T,i)$ be a triangulation of an $n$-gon. 
Let $\CCF_q(\T)=\{\zeta_{k,\ell}\}_{(k,\ell)\in D_{n+1}}$ and $\Farey_q(\T,i)^{\per}=\Bigl(\theta_{i,u}=\dfrac{r_{i,u}}{s_{i,u}}\Bigr)_{u\in\Z}$ be the corresponding $q$-CCF and 
anti-periodic extension of the $q$-Farey labeling, respectively. 
For any $j\in\Z$ with $i-1\leq j\leq i+n-1$, we have 
\begin{align}\label{eqn:quiddity_Farey_graph}
&\zeta_{i-1,j+1}=q^{\ell_{i,j}}r_{i,j+1}\,,\quad \zeta_{i,j+1}=q^{\ell_{i,j}}s_{i,j+1}\,,
\end{align}
where the sequence $\{\ell_{i,j}\}_{i-1\leq j\leq i+n-1}$ is defined by  
\begin{align}\label{eqn:descendant_sum}
&\ell_{i,i-1}=0\,\quad 
\ell_{i,j}-\ell_{i,j-1}=\mu_{i,j}\quad (i\leq j\leq i+n-1), 
\end{align}
and $\mu_{i,j}$ denotes the descendant degree of $v_j$ 
(see Definition \ref{def:desc_deg}). 
In particular, $\ell_{i,i}=\ell_{i,i+1}=0$ since $\mu_{i,i}=\mu_{i,i+1}=0$. 
\end{thm}
\begin{proof} We abbreviate $\theta_{i,j+1}$, $r_{i,j+1}$, $s_{i,j+1}$, $\ell_{i,j}$, 
and $\mu_{i,j}$ as $\theta_{j+1}$, $r_{j+1}$, $s_{j+1}$, $\ell_j$, and $\mu_j$. 
\par
\noindent
{\bf Case 1}. Suppose $i-1\leq j \leq i+n-3$. 
We prove the claim by induction on $m=j-i$. 
\par
First, consider the case $m=-1$, that is, $j=i-1$. 
Since $\ell_{i-1}=0$, $r_i=1$, and $s_i=0$, we have 
\begin{align*}
&\zeta_{i-1,i}=1=q^{\ell_{i-1}}r_i\,,\quad \zeta_{i,i}=0=q^{\ell_{i-1}}s_i\,.
\end{align*}
Therefore, the equations \eqref{eqn:quiddity_Farey_graph} follow. 
\par
Next, suppose $m=0$, that is, $j=i$. By the calculation in the proof of Proposition \ref{prop:weighted_Farey_graph_degree}, we showed $\theta_{i+1}=\dfrac{[c_i]}{1}$. 
Thus $r_{i+1}=[c_i]$ and $s_{i+1}=1$ hold. 
From the definition of the $q$-Farey labeling, $\zeta_{i-1,i+1}=[c_i]$ and 
$\zeta_{i,i+1}=1$. Furthermore, since $\ell_{i}=0$, we have 
\begin{align*}
\zeta_{i-1,i+1}=[c_i]=r_{i+1}=q^{\ell_{i}}r_{i+1}\,,\quad 
\zeta_{i,i+1}=1=s_{i+1}=q^{\ell_{i}}s_{i+1}\,.
\end{align*}
Therefore, the equations \eqref{eqn:quiddity_Farey_graph} follow.  
\par
Moreover, suppose $m=1$, that is, $j=i+1$. Note that $n\geq 4$ since $j\leq i+n-3$. 
By \eqref{eqn:q_unimodular0} and \eqref{eqn:q_unimodular} 
in the definition of the $q$-Farey labeling, 
\begin{align}
&\zeta_{i-1,i+2}=\zeta_{i-1,i+1}\zeta_{i,i+2}-q^{c_i-1}=[c_i][c_{i+1}]-q^{c_i-1}\;,\quad 
\zeta_{i,i+2}=[c_{i+1}]\,.\label{eqn:element_q_CCF1}
\end{align}
Let $\cell(v_{i+1})=\{{\mathcal C}'_0,{\mathcal C}'_1,\dots,{\mathcal C}'_{c_{i+1}-1}\}$ 
be the counterclockwise numbering of the cells adjacent to $v_{i+1}$. 
We write the cell ${\mathcal C}'_0$ of the minimum level 
among them as ${\mathcal C}'_0=v_iv_yv_{i+1}$. 
This cell coincides with the cell ${\mathcal C}_{c_i-1}$ adjacent to $v_i$ discussed 
in the proof of Proposition \ref{prop:weighted_Farey_graph_degree}. 
By inductive computation in that proof, the vertices $v_{y}$ and $v_{i+1}$ 
are labeled by $\theta_{y}=\dfrac{[c_i-1]}{1}$ and $\theta_{i+1}=\dfrac{[c_i]}{1}$, respectively. 
The cell ${\mathcal C}'_0={\mathcal C}_{c_i-1}$ has 
the side-weight triple $(q^{c_i-2},1,q^{c_i-1})$, and 
the other cell ${\mathcal C}'_u$ with $1\leq u\leq c_{i+1}-1$ has 
the side-weight triple $(q^{u-1},1,q^u)$. 
By inductive computation along these cells, the vertex $v_{i+2}$ is labeled by 
$\theta_{i+2}=\dfrac{[c_{i+1}][c_i]-q^{c_i-1}}{[c_{i+1}]}$. 
Therefore, $r_{i+2}=[c_i][c_{i+1}]-q^{c_i-1}$ and $s_{i+2}=[c_{i+1}]$ hold. 
By \eqref{eqn:element_q_CCF1}, since $\ell_{i+1}=0$, it follows that
\begin{align*}
&\zeta_{i-1,i+2}=[c_i][c_{i+1}]-q^{c_i-1}=r_{i+2}=q^{\ell_{i+1}}r_{i+2}\\
&\zeta_{i,i+2}=[c_{i+1}]=s_{i+2}=q^{\ell_{i+1}}s_{i+2}\,.
\end{align*}
Hence, the equations \eqref{eqn:quiddity_Farey_graph} follow. 
\par
Assume that the equations \eqref{eqn:quiddity_Farey_graph} hold 
for all $j\leq i+m$ with $m \geq1$. 
We show that they also hold for $j=i+m+1$. 
\par
By induction hypothesis, we have 
\begin{align}
&\begin{pmatrix}
\zeta_{i-1,j}&\zeta_{i-1,j-1}\\
\zeta_{i,j}&\zeta_{i,j-1}
\end{pmatrix}
=\begin{pmatrix}
q^{\ell_{j-1}}r_{j}&q^{\ell_{j-2}}r_{j-1}\\
q^{\ell_{j-1}}s_{j}&q^{\ell_{j-2}}s_{j-1}
\end{pmatrix}\label{eqn:q-CCF_calc1}\\
&\quad \ell_{i-1}=0,\;\ell_{u}-\ell_{u-1}=\mu_{u}\quad(i\leq u\leq j-1)\label{eqn:q-CCF_calc2}
\end{align}
As in the proof of Proposition \ref{prop:3desect_cont_frac_CCF}, 
since $j=i+m+1\leq i+n-3$ and $m\geq1$, we note that $i+2\leq j\leq i+n-3$ 
and $n\geq 5$. Therefore, the edges $v_{j-1}v_j$ and $v_jv_{j+1}$ are different from 
the base edge $v_{i-1}v_i$. 
%In particular, $v_{j+1}\neq v_{i-1}$ and $v_{j+1}\neq v_i$. 
%Hence, by the definition of the $q$-Farey labeling, both the numerator and denominator of $\theta_{j+1}$ are non-zero polynomials. 証明には用いていないので省略する. 
\par
Let 
\begin{align*}  
\cell(v_{j-1})=\{\widetilde{\mathcal C}_0,\widetilde{\mathcal C}_1,
\dots,\widetilde{\mathcal C}_{c_{j-1}-1}\}\,\text{ and }\quad
\cell(v_{j})=\{{\mathcal C}''_0,{\mathcal C}''_1,\dots,{\mathcal C}''_{c_j-1}\}
\end{align*}
be the counterclockwise numberings of the cells adjacent to $v_{j-1}$ and $v_j$, respectively. 
Note that $\widetilde{\mathcal C}_{c_{j-1}-1}={\mathcal C}''_0$. 
We abbreviate the descendant degree $\mu_{j}$ as $\mu$. 
Then ${\mathcal C}''_{\mu}$ is the cell of the minimum level in $\cell(v_{j})$. 
We prove the claim separately in the two cases, $\mu>0$ and $\mu=0$. 
\par
\noindent
{\bf Subcase 1.1}. Suppose $\mu>0$. 
Then $\widetilde{\mathcal C}_{c_{j-1}-1}$ is the cell of the minimum level in $\cell(v_{j-1})$, and 
so we have
\begin{align}\label{eqn:desc_deg1}
\mu_{j-1}=c_{j-1}-1\,.
\end{align}
Moreover, the two vertices $v_{z_\mu}$ and $v_{z_{\mu+1}}$, 
the endpoints of the base side of ${\mathcal C}''_{\mu}$ have been labeled 
by $\theta_{z_{\mu}}=\dfrac{a}{b}$ and $\theta_{z_{\mu+1}}=\dfrac{a'}{b'}$, respectively, 
and ${\mathcal C}''_{\mu}$ has the side-weight triple $(q^{\ell-1},1,q^{\ell})$. 
Thus the label of the vertex $v_j$ is given by 
\begin{align*}
\theta_{j}=\dfrac{r_{j}}{s_{j}}
=\theta_{z_{\mu+1}}\oplus_q \theta_{z_{\mu}}
=\dfrac{a'}{b'}\oplus_q\dfrac{a}{b}
=\dfrac{a'+q^{\ell}a}{b'+q^{\ell}b}\,.
\end{align*}
Thus we obtain $\theta_{z_{\mu}}=\dfrac{a}{b}=\dfrac{q^{-\ell}(r_{j}-a')}{q^{-\ell}(s_{j}-b')}$.  
Since $\mu>0$, the cell ${\mathcal C}''_{\mu}$ has a child cell 
${\mathcal C}''_{\mu-1}$ 
with base side $v_jv_{z_\mu}$ and side-weight triple $(q^{\ell},1,q^{\ell+1})$. 
Hence the third vertex $v_{z_{\mu-1}}$ (other than the endpoints $v_j$ and $v_{z_{\mu}}$ 
of the base side) of ${\mathcal C}''_{\mu-1}$ is labeled by  
\begin{align*}
\theta_{z_{\mu-1}}=\theta_{j}\oplus_q \theta_{z_{\mu}}
=\dfrac{r_{j}}{s_{j}}\oplus_q \dfrac{q^{-\ell}(r_{j}-a')}{q^{-\ell}(s_{j}-b')}
=\dfrac{r_{j}+q^{\ell+1}\cdot q^{-\ell}(r_{j}-a')}{s_{j}+q^{\ell+1}\cdot q^{-\ell}(s_{j}-b')}
=\dfrac{[2]r_{j}-qa'}{[2]s_{j}-qb'}\,.
\end{align*}
Iterating the same argument along the cells
${\mathcal C}''_{\mu-2},\;{\mathcal C}''_{\mu-3},\;\dots,\;{\mathcal C}''_0$, 
which have the same side-weight triple $(1,1,q)$, 
we obtain 
\begin{align*}
\theta_{j-1}=\dfrac{[\mu+1]r_{j}-q^{\mu}a'}{[\mu+1]s_{j}-q^{\mu}b'}\,.
\end{align*} 
Since $\theta_{j-1}=\dfrac{r_{j-1}}{s_{j-1}}$, 
$\theta_{z_{\mu+1}}=\dfrac{a'}{b'}=
\dfrac{q^{-\mu}([\mu+1]r_{j}-r_{j-1})}{q^{-\mu}([\mu+1]s_{j}-s_{j-1})}$ holds. 
Furthermore, applying the same argument along the cells 
${\mathcal C}''_{x}$ with side-weight triple $(q^{x-\mu-1},1,q^{x-\mu})$ for 
$\mu+1\leq x\leq c_j-1$, we obtain 
\begin{align*}
\theta_{j+1}=\dfrac{q^{-\mu}([c_j]r_{j}-r_{j-1})}{q^{-\mu}([c_j]s_{j}-s_{j-1})}\,.
\end{align*}
Substituting $(\ref{eqn:q-CCF_calc1})$, $(\ref{eqn:q-CCF_calc2})$, 
and $(\ref{eqn:desc_deg1})$, we get 
\begin{align*}
\begin{sumipmatrix}
r_{j+1}\\
s_{j+1}
\end{sumipmatrix}
&=q^{-\mu}
\left([c_{j}]
\begin{sumipmatrix}
r_{j}\\
s_{j}
\end{sumipmatrix}
-\begin{sumipmatrix}
r_{j-1}\\
s_{j-1}
\end{sumipmatrix}
\right)
=q^{-\mu-\ell_{j-1}}
\left(
[c_{j}]
\begin{sumipmatrix}
\zeta_{i-1,j}\\
\zeta_{i,j}
\end{sumipmatrix}
-q^{c_{j-1}-1}
\begin{sumipmatrix}
\zeta_{i-1,j-1}\\
\zeta_{i,j-1}
\end{sumipmatrix}
\right)\,.
\end{align*}
Moreover, by the equation \eqref{eqn:q-CCF_calc} in the proof of Proposition \ref{prop:q_CCF_q_rational}, we get
\begin{align}
\begin{sumipmatrix}
\zeta_{i-1,j+1}\\
\zeta_{i,j+1}
\end{sumipmatrix}
&=[c_j]
\begin{sumipmatrix}
R\\
S
\end{sumipmatrix}
-q^{c_{j-1}-1}
\begin{sumipmatrix}
R'\\
S'
\end{sumipmatrix}
=
[c_j]
\begin{sumipmatrix}
\zeta_{i-1,j}\\
\zeta_{i,j}
\end{sumipmatrix}
-q^{c_{j-1}-1}
\begin{sumipmatrix}
\zeta_{i-1,j-1}\\
\zeta_{i,j-1}
\end{sumipmatrix}\,.\label{eqn:q-CCF_calc_again}
\end{align}
Therefore, we obtain
\begin{align*}
r_{j+1}=q^{-\mu-\ell_{j-1}}\zeta_{i-1,j+1}\,,\quad s_{j+1}=q^{-\mu-\ell_{j-1}}\zeta_{i,j+1}
\end{align*} 
Set $\ell_{j}=\ell_{j-1}+\mu$. Then we have
\begin{align*}
&\zeta_{i-1,j+1}=q^{\ell_{j}}r_{j+1}\,,\quad \zeta_{i,j+1}=q^{\ell_{j}}s_{j+1}\,,\\
&\ell_{j}-\ell_{j-1}=\mu=\mu_{j}\,.
\end{align*}
\par
\noindent
{\bf Subcase 1.2}. Suppose $\mu=0$. 
Then 
${\mathcal C}''_0=\widetilde{\mathcal C}_{c_{j-1}-1}$  
is the cell of minimum level in $\cell(v_j)$. We denote this cell by ${\mathcal C}$. 
Let $v_{z_1}$ be the vertex other than two vertices $v_{j-1}$ and $v_j$ in ${\mathcal C}$. 
Let $\widetilde{\mu}=\mu_{j-1}$ be the descendant degree of $v_{j-1}$. 
Then $\widetilde{\mu}<c_{j-1}-1$ holds. 
Because, assumption that $\widetilde{\mu}=c_{j-1}-1$ implies the contradiction as follows. 
From the assumption, the side $v_{z_1}v_j$ is the base side of ${\mathcal C}$. 
If $c_j\geq 2$, then ${\mathcal C}$ has a parent cell in $\cell(v_j)$, which contradicts $\mu=0$. 
Thus $c_j=1$. However, it implies that $v_{z_1}=v_{j+1}$, that is, the edge $v_{j}v_{j+1}$ 
is the base edge, which is a contradiction. 
Since $\widetilde{\mu}<c_{j-1}-1$, the cell 
$\widetilde{\mathcal C}_{\widetilde{\mu}}$ has at least one child 
$\widetilde{\mathcal C}_{\widetilde{\mu}+1}$. 
The child cells $\widetilde{\mathcal C}_{y}$ with $\widetilde{\mu}+1\leq y\leq c_{j-1}-1$ 
have the side-weight triples $(q^{y-\widetilde{\mu}-1},1,q^{y-\widetilde{\mu}})$. 
In particular, the cell ${\mathcal C}$ with the base side $v_{j-1}v_{z_1}$ has 
the side-weight triple $(q^{c_{j-1}-\widetilde{\mu}-2},1,q^{c_{j-1}-\widetilde{\mu}-1})$. 
Therefore, 
\begin{align*}
\theta_j=\dfrac{r_j}{s_j}=\theta_{z_1}\oplus_q \theta_{j-1}=
\dfrac{a+q^{c_{j-1}-\widetilde{\mu}-1}r_{j-1}}{b+q^{c_{j-1}-\widetilde{\mu}-1}s_{j-1}}\,.
\end{align*}
Hence we obtain $\theta_{z_1}=\dfrac{a}{b}
=\dfrac{r_j-q^{c_{j-1}-\widetilde{\mu}-1}r_{j-1}}{s_j-q^{c_{j-1}-\widetilde{\mu}-1}s_{j-1}}$. 
Furthermore, since the child cells ${\mathcal C}''_{x}$ of ${\mathcal C}$ with $1\leq x\leq c_j-1$ 
has the side-weight triple $(q^{x-1},1,q^{x})$, 
by inductive computation along these cells, the vertex $v_{j+1}$ is labeled by 
\begin{align*}
\theta_{j+1}
=\dfrac{[c_j]r_j-q^{c_{j-1}-\widetilde{\mu}-1}r_{j-1}}{[c_j]s_j-q^{c_{j-1}-\widetilde{\mu}-1}s_{j-1}}\,.
\end{align*}
Substituting \eqref{eqn:q-CCF_calc1}, \eqref{eqn:q-CCF_calc2}, and 
$\ell_{j-1}-\ell_{j-2}=\mu_{j-1}=\widetilde{\mu}$, we get 
\begin{align*}
\begin{sumipmatrix}
r_{j+1}\\
s_{j+1}
\end{sumipmatrix}
&=[c_{j}]
\begin{sumipmatrix}
r_{j}\\
s_{j}
\end{sumipmatrix}
-\begin{sumipmatrix}
r_{j-1}\\
s_{j-1}
\end{sumipmatrix}
=q^{-\ell_{j-1}}
\left(
[c_{j}]
\begin{sumipmatrix}
\zeta_{i-1,j}\\
\zeta_{i,j}
\end{sumipmatrix}
-q^{c_{j-1}-1}
\begin{sumipmatrix}
\zeta_{i-1,j-1}\\
\zeta_{i,j-1}
\end{sumipmatrix}
\right)\,.
\end{align*}
Moreover, again by \eqref{eqn:q-CCF_calc_again}, we get
\begin{align*}
r_{j+1}=q^{-\ell_{j-1}}\zeta_{i-1,j+1}\,,\quad s_{j+1}=q^{-\ell_{j-1}}\zeta_{i,j+1}\,.
\end{align*}
Set $\ell_{j}=\ell_{j-1}$. Then we obtain
\begin{align*}
\zeta_{i-1,j+1}=q^{\ell_{j}}r_{j+1}\,,\quad \zeta_{i,j+1}=q^{\ell_{j}}s_{j+1}\,.
\end{align*}
Since $\mu_{j}=0$, we have 
\begin{align*}
\ell_{j}-\ell_{j-1}=0=\mu_{j}\,.
\end{align*}
Therefore, the result follows for $j=i+m+1$. 
\par
\noindent
{\bf Case 2}. Suppose $j=i+n-2$. Then, since $\theta_{j+1}=\theta_{i+n-1}=\dfrac{0}{1}$ by the definition of $q$-Farey labeling, we have 
\begin{align}
&r_{i+n-1}=0,\;s_{i+n-1}=1\label{eqn:q_i-1}\,.
\end{align}
Furthermore, by Proposition \ref{prop:q_CCF_q_rational} (1) and Lemma \ref{lemm:n-1_q}, we have 
\begin{align}			
&\zeta_{i-1,i+n-1}=0,\;\zeta_{i,i+n-1}=q^{n-2-c_{i-1}}\,.\label{eqn:zeta_i-1}
\end{align}
Here, set 
\begin{align*}
\zeta_{i-1,i+n-1}=q^{\ell_{i+n-2}}r_{i+n-1}\,,\quad \zeta_{i,i+n-1}=q^{\ell_{i+n-2}}s_{i+n-1}\,. 
\end{align*}
Substituting \eqref{eqn:q_i-1} and \eqref{eqn:zeta_i-1}, we obtain 
\begin{align}\label{eqn:ell_i+n-2}
\ell_{i+n-2}=n-2-c_{i-1}\,.
\end{align}
On the other hand, let 
\begin{align*}
\cell(v_{i+n-2})=\bigl\{{\widetilde{\widetilde{\mathcal C}}_0},{\widetilde{\widetilde{\mathcal C}}_1},\dots,
{\widetilde{\widetilde{\mathcal C}}_{c_{i-2}-1}}\bigr\}
\,,\quad 
\cell(v_{i+n-1})=\{\widetilde{\mathcal C}_0,\widetilde{\mathcal C}_1,\dots,
\widetilde{\mathcal C}_{c_{i-1}-1}\}\,,
\end{align*}
where ${\widetilde{\widetilde{\mathcal C}}_{c_{i-2}-1}}=\widetilde{\mathcal C}_0$.  
Since the cell $\widetilde{\mathcal C}_{c_{i-1}-1}$ is the base cell with the base edge $v_{i+n-1}v_i$, 
the cell of minimum level in $\cell(v_{i+n-2})$ is ${\widetilde{\widetilde{\mathcal C}}_{c_{i-2}-1}}=\widetilde{\mathcal C}_0$. Therefore, we obtain
\begin{align}\label{eqn;desc_deg_i-2}
\mu_{i+n-2}=c_{i-2}-1\,.
\end{align}
In the proof of Proposition \ref{prop:weighted_Farey_graph_degree}, we showed that 
$\theta_{i+n-2}=\dfrac{q^{c_{i-1}-1}}{[c_{i-1}]}$, that is, 
$r_{i+n-2}=q^{c_{i-1}-1}$, $s_{i+n-2}=[c_{i-1}]$. 
In Case 1, we show that 
$\zeta_{i-1,i+n-2}=q^{\ell_{i+n-3}}r_{i+n-2}$ and $\zeta_{i,i+n-2}=q^{\ell_{i+n-3}}s_{i+n-2}$. 
Thus  
\begin{align*}
\zeta_{i-1,i+n-2}=q^{\ell_{i+n-3}+c_{i-1}-1}\,,\quad 
\zeta_{i,i+n-2}=q^{\ell_{i+n-3}}[c_{i-1}]\,.
\end{align*}
holds. Furthermore, by Proposition \ref{prop:q_CCF_q_rational} (1) and Lemma \ref{lemm:n-1_q}, 
we have
\begin{align*}
\zeta_{i-1,i+n-2}=q^{n-2-c_{i-2}}\,,\quad \zeta_{i,i+n-2}=q^{n-1-c_{i-2}-c_{i-1}}[c_{i-1}]\,. 
\end{align*}
Hence $\ell_{i+n-3}=n-1-c_{i-2}-c_{i-1}$ holds. 
This equation together with \eqref{eqn:ell_i+n-2} and \eqref{eqn;desc_deg_i-2} implies that 
\begin{align*}
\ell_{i+n-2}-\ell_{i+n-3}=n-2-c_{i-1}-(n-1-c_{i-2}-c_{i-1})=c_{i-2}-1=\mu_{i+n-2}\,. 
\end{align*}
Therefore, the result follows for $j=i+n-2$.  
\par
\noindent
{\bf Case 3}. Suppose $j=i+n-1$. 
In Case 2, we show that  
\begin{align*}
\zeta_{i-1,i+n-1}=q^{\ell_{i+n-2}}r_{i+n-1}\,,\quad 
\zeta_{i,i+n-1}=q^{\ell_{i+n-2}}s_{i+n-1}\;.
\end{align*}
By \eqref{eqn:q-mat_n-1} in Lemma \ref{lemm:n-1_q} and Proposition \ref{prop:q-CCF_n-1}, we have
\begin{align}
\zeta_{i-1,i+n}=-q^{n-3}\,,\quad \zeta_{i,i+n}=0 \label{eqn:q_CCF_n-1_ele}
\end{align}
By the definition of the anti-periodic repetition
$\Farey_q(\T,i)^{\per}=(\theta_{i,u})_{u\in\Z}$, 
\begin{align}
r_{i+n}=-r_i=-1,\;s_{i+n}=-s_i=-0=0\label{eqn:q_Farey_n}
\end{align}
holds. Here, set  
\begin{align*}
\zeta_{i-1,i+n}=q^{\ell_{i+n-1}}r_{i+n}\,,\quad \zeta_{i,i+n}=q^{\ell_{i+n-1}}s_{i+n}\,.
\end{align*}
Substituting \eqref{eqn:q_CCF_n-1_ele} and \eqref{eqn:q_Farey_n}, we obtain
\begin{align}\label{eqn:ell_i+n-1}
\ell_{i+n-1}=n-3
\end{align}
holds. This equation together with \eqref{eqn:ell_i+n-2} implies that 
\begin{align}\label{eqn:descendant_i+n-1}
\ell_{i+n-1}-\ell_{i+n-2}=(n-3)-(n-2-c_{i-1})=c_{i-1}-1\,.
\end{align}
Since the descendant degree $\mu_{i+n-1}$ of $v_{i+n-1}$ is given as $\mu_{i+n-1}=c_{i-1}-1$, 
\begin{align*}
\ell_{i+n-1}-\ell_{i+n-2}=\mu_{i+n-1}
\end{align*}
holds. Therefore, the result follows for $j=i+n-1$. 
\end{proof}
%
%\newpage
\begin{example}\label{ex:q_CCF_Farey_graph1}
From Theorem \ref{prop:q_CCF_weighted_Farey}, 
by computing the $q$-Farey labeling of a triangulation and multiplying 
by the accumulated descendant degrees, 
we obtain the entries of the associated $q$-CCF. 
Figure \ref{fig:q_CCF_Farey_graph1} shows the $q$-Farey labeling 
$\Farey_q(\T,3)=(\theta_{3,j})_{3\leq j\leq 11}$ 
for the triangulation $(\T,3)$ in Example \ref{ex:q_CCF_Farey_graph}. 
For each cell of level $L$, the label \maru{$L$} is placed inside the cell. 
For each vertex $v_{j}$ with $3\leq j\leq 11$, 
the descendant degree $\mu_{i,j}$ and $\ell_{3,j}$ are indicated beside $v_j$.  
\par
\begin{figure}[htbp]
\begin{center}
\begin{tikzpicture}[scale=.8]
\coordinate(v0)at(0,0);
\coordinate(v1)at(-3,0);
\coordinate(v2)at(-5,2);
\coordinate(v3)at(-5.5,4.5);
\coordinate(v4)at(-3.5,6.5);
\coordinate(v5)at(-1,7);
\coordinate(v6)at(1.6,6.2);
\coordinate(v7)at(2.5,4);
\coordinate(v8)at(2.2,2);
\filldraw [black] (v0) circle (1pt) node[below]{
\hspace{3.5cm}$v_9\left(
\theta_{3,9}=\dfrac{q[4]}{\{7\}}\right), 
\begin{array}{l}
\mu_{3,9}=0,\\
\ell_{3,9}=2
\end{array}
$};
\filldraw [black] (v1) circle (1pt) node[below]{
\hspace{-2cm}$v_{10}
\left(\theta_{3,10}=\dfrac{q}{[2]}\right),
\begin{array}{l}
\mu_{3,10}=3,\\
\ell_{3,10}=5
\end{array}
$};
\filldraw [black] (v2) circle (1pt) node[left]{
\raisebox{-2cm}{
$v_{11}
\left(\theta_{3,11}=\dfrac{0}{1}\right),
\begin{array}{l}
\ell_{3,2}=0,\\
\mu_{3,11}=1,\\
\ell_{3,11}=6
\end{array}$}
};
\filldraw [black] (v3) circle (1pt) node[left]{
$v_{3}
\Biggl(
\begin{array}{l}
\theta_{3,3}=\dfrac{1}{0},\\
\theta_{3,12}=\dfrac{-1}{-0}
\end{array}
\Biggr),
\begin{array}{l}
\mu_{3,3}=0,\\
\ell_{3,3}=0
\end{array}$
};
\filldraw [black] (v4) circle (1pt) node[above left]{
$v_{4}\left(
\theta_{3,4}=\dfrac{1}{1}\right),
\begin{array}{l}
\mu_{3,4}=0,\\
\ell_{3,4}=0
\end{array}
$};
\filldraw [black] (v5) circle (1pt) node[above]{
\raisebox{.5cm}{\hspace{1cm}{$v_{5}\left(
\theta_{3,5}=\dfrac{q[2]}{[3]}\right),
\begin{array}{l}
\mu_{3,5}=0,\\
\ell_{3,5}=0
\end{array}
$}}};
\filldraw [black] (v6) circle (1pt) node[right]{
\raisebox{1cm}{$v_{6}
\left(
\theta_{3,6}=\dfrac{q\langle7\rangle}{\langle11\rangle}\right),
\begin{array}{l}
\mu_{3,6}=0,\\
\ell_{3,6}=0
\end{array}
$}};
\filldraw [black] (v7) circle (1pt) node[right]{
\raisebox{-1cm}{$v_7
\left(\theta_{3,7}=\dfrac{q\langle5\rangle}{\langle8\rangle}\right),
\begin{array}{l}
\mu_{3,7}=1,\\
\ell_{3,7}=1
\end{array}
$}};
\filldraw [black] (v8) circle (1pt) node[right]{
\raisebox{-1.5cm}{$v_8\left(
\theta_{3,8}=\dfrac{q[3]}{\{5\}}\right),
\begin{array}{l}
\mu_{3,8}=1,\\
\ell_{3,8}=2
\end{array}
$}};
\draw (v0)--(v1) node[midway, above]{\hspace{1cm}$1\hspace{.5cm}\raisebox{.2cm}{\maru{$4$}}$};
\draw (v1)--(v2) node[midway, above]{\raisebox{.4cm}{$1\raisebox{1cm}{\maru{$1$}}$}};
\draw (v1)--(v4) node[midway, right]{$q\raisebox{1cm}{\maru{$2$}}$};
\draw (v1)--(v5)node[midway, right]{$1\hspace{.5cm}\maru{$3$}$};
\draw (v1)--(v8)node[midway, above left]{$1$};
\draw (v2)--(v3)node[midway, left]{$q^{-1}$};
\draw (v2)--(v4)node[midway, left]{$\maru{$0$}\hspace{.2cm}1\hspace{-.1cm}$};
\draw (v3)--(v4)node[midway, above]{$1$};
\draw (v4)--(v5)node[midway, above]{$q^2$};
\draw (v5)--(v6)node[midway, above]{$q^3$};
\draw (v5)--(v7)node[midway, above]{\quad$\raisebox{.1cm}{$q^2$}\;\maru{$5$}$};
\draw (v5)--(v8)node[midway, right]{$q\;\maru{$4$}$};
\draw (v6)--(v7)node[midway, right]{$1$};
\draw (v7)--(v8)node[midway, right]{$1$};
\draw (v8)--(v0)node[midway, right]{$q$};
\end{tikzpicture}
\caption{$q$-Farey labeling $\Farey_q(\T,3)$ and descendant degrees 
for the triangulation $(\T,3)$ with quiddity $\quid(\T)=(1,4,2,1,3,4,1,2,3)$}
\label{fig:q_CCF_Farey_graph1}
\end{center}
\end{figure}
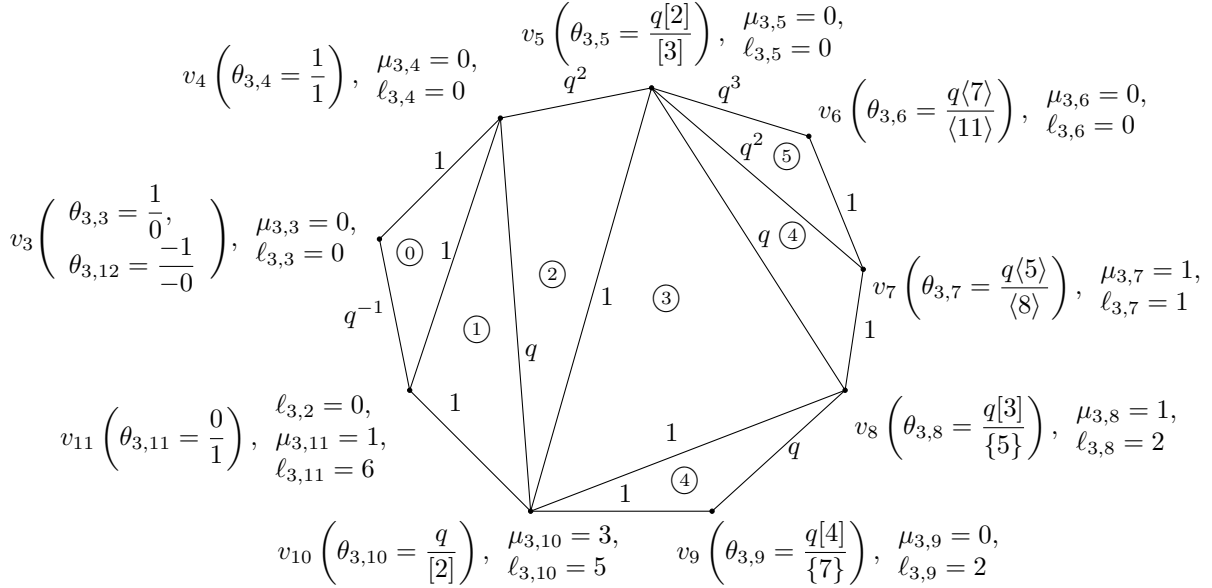
%
%\newpage
\par
From Proposition \ref{prop:q_CCF_q_rational}, the $q$-continued fractions
$[[c_3,c_4,\dots,c_j]]_q$ for $2\leq j\leq11$ are given by the entries in the associated 
$q$-CCF, which are computed in Example \ref{ex:q_CCF}. 
The correspondences between the $q$-continued fractions and the $q$-Farey labeling are as follows.  
\begin{align*}
&[[\quad]]_q=\dfrac{\zeta_{2,3}}{\zeta_{3,3}}=\dfrac{q^{\ell_{3,2}}r_3}{q^{\ell_{3,2}}s_3}
=\dfrac{q^0\cdot1}{q^0\cdot0}=\dfrac{1}{0}\\
&[[c_3]]_q=[[1]]_q=\dfrac{\zeta_{2,4}}{\zeta_{3,4}}
=\dfrac{q^{\ell_{3,3}}r_4}{q^{\ell_{3,3}}s_4}
=\dfrac{q^0\cdot 1}{q^0\cdot1}=\dfrac{1}{1}\\
&[[c_3,c_4]]_q=[[1,3]]_q
=\dfrac{\zeta_{2,5}}{\zeta_{3,5}}
=\dfrac{q^{\ell_{3,4}}r_5}{q^{\ell_{3,4}}s_5}
=\dfrac{q^0\cdot q[2]}{q^0\cdot[3]}=\dfrac{q[2]}{[3]}\\
&[[c_3,c_4,c_5]]_q=[[1,3,4]]_q=\dfrac{\zeta_{2,6}}{\zeta_{3,6}}
=\dfrac{q^{\ell_{3,5}}r_{6}}{q^{\ell_{3,5}}s_{6}}
=\dfrac{q^0\cdot q\langle7\rangle}{q^0\cdot\langle11\rangle}
=\dfrac{q\langle7\rangle}{\langle11\rangle}\\
\intertext{}
&[[c_3,c_4,c_5,c_6]]_q=[[1,3,4,1]]_q=\dfrac{\zeta_{2,7}}{\zeta_{3,7}}
=\dfrac{q^{\ell_{3,6}}r_{7}}{q^{\ell_{3,6}}s_{7}}
=\dfrac{q^0\cdot q\langle5\rangle}{q^0\cdot\langle8\rangle}
=\dfrac{q\langle5\rangle}{\langle8\rangle}\\
&[[c_3,c_4,c_5,c_6,c_7]]_q=[[1,3,4,1,2]]_q=\dfrac{\zeta_{2,8}}{\zeta_{3,8}}
=\dfrac{q^{\ell_{3,7}}r_8}{q^{\ell_{3,7}}s_8}
=\dfrac{q\cdot q[3]}{q\cdot\{5\}}
=\dfrac{q^2[3]}{q\{5\}}\\
&[[c_3,c_4,c_5,c_6,c_7,c_8]]_q=[[1,3,4,1,2,3]]_q=\dfrac{\zeta_{2,9}}{\zeta_{3,9}}
=\dfrac{q^{\ell_{3,8}}r_9}{q^{\ell_{3,8}}s_9}
=\dfrac{q^2\cdot q[4]}{q^2\cdot\{7\}}
=\dfrac{q^3[4]}{q^2\{7\}}\\
&[[c_3,c_4,c_5,c_6,c_7,c_8,c_9]]_q=[[1,3,4,1,2,3,1]]_q=\dfrac{\zeta_{2,10}}{\zeta_{3,10}}
=\dfrac{q^{\ell_{3,9}}r_{10}}{q^{\ell_{3,9}}s_{10}}
=\dfrac{q^2\cdot q}{q^2\cdot[2]}=\dfrac{q^3}{q^2[2]}\\
&[[c_3,c_4,c_5,c_6,c_7,c_8,c_9,c_{10}]]_q=[[1,3,4,1,2,3,1,4]]_q=\dfrac{\zeta_{2,11}}{\zeta_{3,11}}=\dfrac{q^{\ell_{3,10}}r_{11}}{q^{\ell_{3,10}}s_{11}}
=\dfrac{q^5\cdot 0}{q^5\cdot 1}=\dfrac{0}{q^5}\\
&[[c_3,c_4,c_5,c_6,c_7,c_8,c_9,c_{10},c_{11}]]_q=[[1,3,4,1,2,3,1,4,2]]_q
=\dfrac{\zeta_{2,12}}{\zeta_{3,12}}=\dfrac{q^{\ell_{3,11}}r_{12}}{q^{\ell_{3,11}}s_{12}}
=\dfrac{q^6\cdot(-1)}{q^6\cdot(-0)}=\dfrac{-q^6}{0}
\end{align*}
\par
The entries $\zeta_{2,j}$ and $\zeta_{3,j}$ with $3\leq j\leq 12$ of the 
$q$-CCF corresponds with the $q$-Farey labeling as shown in Proposition \ref{prop:q_CCF_q_rational}. 
\par
In Section \ref{sect:q_CCF_weighted_Farey_2}, we show the same correspondence 
for the triangulation $(\T,7)$. 
\end{example}
\vskip.5cm
The following proposition describes the greatest common divisor of 
the numerator and denominator of a $q$-continued fraction. 
This generalizes the result in Morier-Genoud and Ovsienko 
\cite[Proposition 1.6 (iii), (iv)] {morier2020q}.  
\begin{prop}[cf.\ {\cite[Proposition 1.6 (iii), (iv)]{morier2020q}}]
\label{prop:coprime}
Let $\CCF_q(\T)=\{\zeta_{k,\ell}\}_{(k,\ell)\in D_{n+1}}$ 
be the $q$-CCF for a triangulation $(\T,i)$ of an $n$-gon. 
Let $\Farey_q(\T,i)^{\per}=\Bigl(\theta_{i,u}=\dfrac{r_{i,u}}{s_{i,u}}\Bigr)_{u\in\Z}$ 
be its anti-periodic extension.   
\begin{enumerate}[{(}1{)}]
\item For any j with $i-1\leq j\leq i+n-1$, we have 
\begin{align*}
&\gcd(\zeta_{i-1,j+1},\zeta_{i,j+1})=q^{\ell_{i,j}}\,,\quad \gcd(r_{i,j+1},s_{i,j+1})=1\,.
\end{align*}
\item For any $j$ with $i-1\leq j\leq i+n-3$, the entries 
$\zeta_{i-1,j+1}$ have trailing and leading coefficients equal to $1$, 
and all intermediate coefficients are positive.  
\end{enumerate}
\end{prop}
\noindent
\begin{proof} We use the abbreviations 
$\theta_{i,j+1}=\theta_{j+1}$, $r_{i,j+1}=r_{j+1}$, and $s_{i,j+1}=s_{j+1}$. 
\par
\if0
First suppose $j=i-1$. Then $\zeta_{i-1,i}=1$, 
$\zeta_{i,i}=0$, $\theta_{i}=\dfrac{r_i}{s_i}=\dfrac{1}{0}$, 
$\ell_{i,i-1}=0$, so that 
\begin{align*}
&\gcd(\zeta_{i-1,j+1},\zeta_{i,j+1})=\gcd(1,0)=1=q^{\ell_{i,j}}, \\
&\gcd(r_{j+1},s_{j+1})=\gcd(1,0)=1\,,
\end{align*}
the statement (1) follows. Since $\zeta_{i-1,i}=1$, the statement (2) follows.  
\par
\fi
First suppose $i\leq j\leq i+n-3$. 
By the $q$-unimodular rule of the $q$-CCF, 
\[
\zeta_{i-1,j}\zeta_{i,j+1}-\zeta_{i-1,j+1}\zeta_{i,j}=q^{\sum\limits_{u=i}^{j-1}(c_u-1)}\,,
\]
so $\gcd(\zeta_{i-1,j+1},\zeta_{i,j+1})$ is a power of $q$. 
By \eqref{eqn:quiddity_Farey_graph}, 
we have
\begin{align*}
\zeta_{i-1,j+1}=q^{\ell_{i,j}}r_{j+1}\,,\quad
\zeta_{i,j+1}=q^{\ell_{i,j}}s_{j+1}\,.
\end{align*}
Hence
\begin{align*}
\gcd(\zeta_{i-1,j+1},\zeta_{i,j+1})=q^{\ell_{i,j}}\gcd(r_{j+1},s_{j+1})\,.
\end{align*}
Since, by Proposition \ref{prop:weighted_Farey_graph_degree} (1) and (2), 
the constant term of $s_{j+1}$ is equal to $1$, 
we obtain $\gcd(r_{j+1},s_{j+1})=1$. 
Thus the result (1) follows.  
\par
Moreover, by Proposition \ref{prop:weighted_Farey_graph_degree} (1), 
the numerator $r_{j+1}$ has trailing and leading coefficients equal to $1$, and 
all intermediate coefficients are positive.  
Since $\zeta_{i-1,j+1}=q^{\ell_{i,j}}r_{j+1}$, 
the same property holds for $\zeta_{i-1,j+1}$. Thus the result (2) follows. 
\par
For the boundary cases $j=i-1$, $j=i+n-2$, and $j=i+n-1$, 
the statements follow directly from  the explicit values of the corresponding entries 
of the $q$-CCF and the definition of $\ell_{i,j}$. 
\if0 
Then 
$\zeta_{i-1,i+n-1}=0$, $\zeta_{i,i+n-1}=q^{n-2-c_{i-1}}$, and 
$\theta_{i+n-1}=\dfrac{r_{i+n-1}}{s_{i+n-1}}=\dfrac{0}{1}$. 
Furthermore, by \eqref{eqn:ell_i+n-2} in the proof of Theorem \ref{prop:q_CCF_weighted_Farey}, 
we obtain
\begin{align*}
&\gcd(\zeta_{i-1,j+1},\zeta_{i,j+1})=\gcd(0,q^{n-2-c_{i-1}})=q^{n-2-c_{i-1}}=q^{\ell_{i,j}}\,,\\
&\gcd(r_{j+1},s_{j+1})=\gcd(0,1)=1\,.
\end{align*}
Therefore the result (1) follows.  
\par
Suppose $j=i+n-1$. 
Then $\zeta_{i-1,i+n}=-q^{n-3}$, $\zeta_{i,i+n}=0$, and 
$\theta_{i+n}=\dfrac{r_{i+n}}{s_{i+n}}=\dfrac{-r_i}{-s_i}=\dfrac{-1}{-0}$. 
Furthermore, by \eqref{eqn:ell_i+n-1} in the proof of Theorem \ref{prop:q_CCF_weighted_Farey}, 
we obtain
\begin{align*}
&\gcd(\zeta_{i-1,j+1},\zeta_{i,j+1})=\gcd(-q^{n-3},0)=q^{n-3}=q^{\ell_{i,j}}\,, \\
&\gcd(r_{j+1},s_{j+1})=\gcd(-1,0)=1\,.
\end{align*}
Therefore the result (1) follows.  
\fi
\end{proof}
\section{Degrees in \texorpdfstring{$q$}{\it q}-CCFs and 
\texorpdfstring{$q$}{\it q}-Farey labelings}
\label{sect:degree_q_CCF_q_Farey_graph}
\subsection{Degrees in \texorpdfstring{$q$}{\it q}-CCFs}
Let $\dfrac{r}{s}=[[c_1,c_2,\dots,c_k]]$ with $c_u\in\Z$ and $c_u\geq 2$ 
be the negative continued fraction of a rational number $\dfrac{r}{s}>1$. 
Then the sequence $(c_1,c_2,\dots,c_k)$ is a subsequence of the quiddity 
$(c_u)_{u\in\Z}$ of a triangulation with exactly two exterior cells, 
where $c_0=c_{k+1}=1$. 
Morier-Genoud and Ovsienko \cite[Proposition 1.6 (i), Corollary 1.7 (i)]{morier2020q} 
showed that, 
if $[[c_1,c_2,\dots,c_j]]_q=\dfrac{{\mathcal R}_j}{{\mathcal S}_j}$ with $1\leq j\leq k$, 
then
\begin{align}
&\deg_{\min}{\mathcal R}_j=\deg_{\min}{\mathcal S}_j=0\,,\label{eqn:MG_O_q_rat1}\\
&\deg{\mathcal R}_j=c_1+\cdots+c_j-j\,,\label{eqn:MG_O_q_rat2}\\
&\deg{\mathcal S}_j=c_2+\cdots+c_j-j+1\,.\label{eqn:MG_O_q_rat3}
\end{align}
However, for general triangulations, 
these equalities do not hold. 
For example, in Example \ref{ex:q_CCF_Farey_graph1}, 
we have $[[1,3]]_q=\dfrac{q[2]}{[3]}$, so that $\deg_{\min}q[2]=1>0$, and hence, 
\eqref{eqn:MG_O_q_rat1} fails. 
Furthermore, the degrees in 
$[[1,3,4,1,2]]_q=\dfrac{q^2[3]}{q\{5\}}$ are strictly smaller than those 
predicted by \eqref{eqn:MG_O_q_rat2} and 
\eqref{eqn:MG_O_q_rat3}. 
The following proposition describes, for general triangulations, 
the exact degrees of the numerators 
$\zeta_{i-1,j+1}$ and denominators $\zeta_{i,j+1}$ 
of the $q$-continued fractions $[[c_i,c_{i+1},\dots,c_j]]_q$. 
The degrees are determined by the number of occurrences of $1$ in the associated 
$q$-CCF. 
\begin{prop}\label{prop:q_CCF_min_deg}
For any $k,\ell\in\Z$ with $k+1\leq \ell \leq k+n-1$, 
the trailing and leading degrees of $\zeta_{k,\ell}$ are given as follows. 
\begin{align*}
&\deg_{\min}\zeta_{k,\ell}=\nu^{\tri}_{k,\ell-1}\;,\quad 
\deg \zeta_{k,\ell}=\sum_{u=k+1}^{\ell-1}(c_u-1)-\nu^{\tri}_{k+1,\ell}
=\nu^{\tri}_{k,\ell-1}+\nu^{\rec}_{k+1,\ell-1,k+n-1}\;, 
\end{align*}
where, for $\ell=k+1$, we regard $\sum\limits_{u=k+1}^{\ell-1}(c_u-1)$ as $0$.  
\end{prop}
\noindent
\begin{proof}
The second equality for $\deg \zeta_{k,\ell}$ immediately follows from 
Lemma \ref{lem:occ1}. Hence, we only prove the first equality. 
We prove the claim by induction on $\ell-k$. 
\par
Suppose $\ell=k+1$. Then $\zeta_{k,k+1}=1$, 
and hence $\deg_{\min}\zeta_{k,k+1}=\deg \zeta_{k,k+1}=0$. 
Since $U^{\tri}_{k,k}=U^{\tri}_{k+1,k+1}=\phi$, 
we obtain $\nu^{\tri}_{k,k}=\nu^{\tri}_{k+1,k+1}=0$, 
and the statement follows. 
\par
Suppose $\ell=k+2$. Then 
$\zeta_{k,k+2}=[c_{k+1}]=1+q+\cdots+q^{c_{k+1}-1}$, so that  
$\deg_{\min}\zeta_{k,k+2}=0$, $\deg \zeta_{k,k+2}=c_{k+1}-1$. 
Since $U^{\tri}_{k,k+1}=U^{\tri}_{k+1,k+2}=\phi$, we obtain
$\nu^{\tri}_{k,k+1}=\nu^{\tri}_{k+1,k+2}=0$, 
and the statement follows. 
\par
Assume that the statement holds for $\ell\leq k+m+1$ with $m\geq1$. 
We prove the statement for $\ell=k+m+2$. 
By the $q$-unimodular rule of the $q$-CCF, we have 
$\zeta_{k,\ell}
=\dfrac{\zeta_{k,\ell-1}\zeta_{k+1,\ell}
-q^{\sum\limits_{\alpha=k+1}^{\ell-1}(c_\alpha-1)}}{\zeta_{k+1,\ell-1}}$. 
Using the induction hypothesis, we obtain
\begin{align}
&\deg_{\min}\zeta_{k,\ell}
=\deg_{\min}
\Bigl(\zeta_{k,\ell-1}\zeta_{k+1,\ell}
-q^{\sum\limits_{\alpha=k+1}^{\ell-1}(c_\alpha-1)}\Bigr)-
\nu^{\tri}_{k+1,\ell-2}\label{eqn:q_CCF_min_deg1}\,,\\
&\deg \zeta_{k,\ell}
=\deg\Bigl(\zeta_{k,\ell-1}\zeta_{k+1,\ell}
-q^{\sum\limits_{\alpha=k+1}^{\ell-1}(c_\alpha-1)}\Bigr)
-\sum_{\beta=k+2}^{\ell-2}(c_\beta-1)+\nu^{\tri}_{k+2,\ell-1}\,,\label{eqn:q_CCF_deg1}\\
&\deg_{\min}(\zeta_{k,\ell-1}\zeta_{k+1,\ell})
=\nu^{\tri}_{k,\ell-2}+\nu^{\tri}_{k+1,\ell-1}
\label{eqn:q_CCF_min_deg2}\,,\\
&\deg(\zeta_{k,\ell-1}\zeta_{k+1,\ell})
=\hspace{-.2cm}\sum_{\alpha=k+1}^{\ell-2}(c_\alpha-1)
+\hspace{-.2cm}\sum_{\gamma=k+2}^{\ell-1}(c_\gamma-1)
-\nu^{\tri}_{k+1,\ell-1}-\nu^{\tri}_{k+2,\ell}\,.\label{eqn:q_CCF_deg2}
\end{align}
Note that, by Proposition~\ref{prop:weighted_Farey_graph_degree}~(1) and 
\eqref{eqn:quiddity_Farey_graph} in Theorem~\ref{prop:q_CCF_weighted_Farey},  
the product $\zeta_{k,\ell-1}\zeta_{k+1,\ell}$ has 
trailing and leading coefficients equal to $1$. 
To determine the trailing (resp.~leading) degrees of 
$\zeta_{k,\ell}$, we analyze the expression
$\zeta_{k,\ell-1}\zeta_{k+1,\ell}
-q^{\sum\limits_{\alpha=k+1}^{\ell-1}(c_\alpha-1)}$ 
according to the values of $\sigma_{k,\ell-1}$ (resp.~$\sigma_{k+1,\ell}$). 
We treat the trailing and leading degrees separately,
resulting in Cases 1.1 and 1.2 (resp.~Cases 2.1 and 2.2). 
Note that both $\sigma_{k,\ell-1}$ and $\sigma_{k+1,\ell}$ 
cannot be equal to $1$, that is, 
both $\zeta_{k,\ell-1}$ and $\zeta_{k+1,\ell}$ 
cannot be monomials in $q$. 
Indeed, otherwise the $q$-unimodular rule would imply 
$\sigma_{k,\ell}=0$, contradicting 
Proposition~\ref{prop:3desect_cont_frac_CCF} (1), (2).
\par
\noindent
{\bf Case 1.1}. Suppose $\sigma_{k,\ell-1}=1$. 
Then, as shown in Case 1 in the proof of Lemma \ref{lem:occ1}, \\
$\sum\limits_{\alpha=k+1}^{\ell-2}(c_{\alpha}-1)=\nu^{\tri}_{k,\ell-2}+\nu^{\tri}_{k+1,\ell-1}$.  
By \eqref{eqn:q_CCF_min_deg2}, we have  
$\deg_{\min}(\zeta_{i,j-1}\zeta_{i+1,j})=\sum\limits_{\alpha=k+1}^{\ell-2}(c_{\alpha}-1)$. 
The product $\zeta_{k,\ell-1}\zeta_{k+1,\ell}$ has trailing coefficient 
equal to $1$. 
Furthermore, since it is not a monomial, 
Proposition \ref{lem:wFg_deg_posi} (1) implies that 
the coefficient of $q^{\nu^{\tri}_{k,\ell-2}+\nu^{\tri}_{k+1,\ell-1}+1}$ is positive. 
Thus,  
\begin{align*}
\deg_{\min}\Bigl(\zeta_{k,\ell-1}\zeta_{k+1,\ell}-q^{\sum\limits_{\alpha=k+1}^{\ell-2}
(c_{\alpha}-1)}\Bigr)=\nu^{\tri}_{k,\ell-2}+\nu^{\tri}_{k+1,\ell-1}+1\,.
\end{align*}
By \eqref{eqn:q_CCF_min_deg1}, it follows that
\begin{align*}
\deg_{\min}\zeta_{k,\ell}
&=\nu^{\tri}_{k,\ell-2}+\nu^{\tri}_{k+1,\ell-1}+1-\nu^{\tri}_{k+1,\ell-2}\,.
\end{align*}
Since $\sigma_{k,\ell-1}=1$, we have $(k,\ell-1)\in U^{\tri}_{k,\ell-1}$, so that 
$U^{\tri}_{k,\ell-1}=U^{\tri}_{k,\ell-2}\cup U^{\tri}_{k+1,\ell-1}\cup\{(k,\ell-1)\}$. 
Noting that $U^{\tri}_{k,\ell-2}\cap U^{\tri}_{k+1,\ell-1}=U^{\tri}_{k+1,\ell-2}$, 
we obtain
\begin{align*}
\nu^{\tri}_{k,\ell-1}=\nu^{\tri}_{k,\ell-2}+\nu^{\tri}_{k+1,\ell-1}-\nu^{\tri}_{k+1,\ell-2}+1\,.
\end{align*}
Therefore, $\deg_{\min}\zeta_{k,\ell}=\nu^{\tri}_{k,\ell-1}$.  
\par
\noindent
{\bf Case 1.2}. Suppose $\sigma_{k,\ell-1}\geq2$. Then, 
as shown in Case 2 in the proof of Proposition 
\ref{lem:occ1}, 
\\
$\sum\limits_{\alpha=k+1}^{\ell-2}(c_{\alpha}-1)>\nu^{\tri}_{k,\ell-2}+\nu^{\tri}_{k+1,\ell-1}$. 
By \eqref{eqn:q_CCF_min_deg2}, we have 
$\deg_{\min}(\zeta_{k,\ell-1}\zeta_{k+1,\ell})<\sum\limits_{\alpha=k+1}^{\ell-2}(c_{\alpha}-1)$. Thus,  
\begin{align*}
\deg_{\min}\Bigl(\zeta_{k,\ell-1}\zeta_{k+1,\ell}-q^{\sum\limits_{\alpha=k+1}^{\ell-2}(c_{\alpha}-1)}\Bigr)
=\nu^{\tri}_{k,\ell-2}+\nu^{\tri}_{k+1,\ell-1}\,.
\end{align*}
By \eqref{eqn:q_CCF_min_deg1}, it follows that
\begin{align*}
\deg_{\min}\zeta_{k,\ell}
&=\nu^{\tri}_{k,\ell-2}+\nu^{\tri}_{k+1,\ell-1}-\nu^{\tri}_{k+1,\ell-2}\,.
\end{align*}
Since $\sigma_{k,\ell-1}\geq 2$, we have $(k,\ell-1)\not\in U^{\tri}_{k,\ell-1}$, so that 
$U^{\tri}_{k,\ell-1}=U^{\tri}_{k,\ell-2}\cup U^{\tri}_{k+1,\ell-1}$. 
Hence, we obtain
\begin{align*}
\nu^{\tri}_{k,\ell-1}&=\nu^{\tri}_{k,\ell-2}+\nu^{\tri}_{k+1,\ell-1}-\nu^{\tri}_{k+1,\ell-2}\,.
\end{align*}
Therefore, $\deg_{\min}\zeta_{k,\ell}=\nu^{\tri}_{k,\ell-1}$.  
\par
The following arguments in Cases 2.1 and 2.2 are parallel to those
in Cases 1.1 and 1.2, respectively.
\par
\noindent
{\bf Case 2.1}. Suppose $\sigma_{k+1,\ell}=1$. 
Then, as in Case 1.1, 
$\sum\limits_{\gamma=k+2}^{\ell-1}(c_\gamma-1)=\nu^{\tri}_{k+1,\ell-1}+\nu^{\tri}_{k+2,\ell}$. By \eqref{eqn:q_CCF_deg2}, we have
$\deg(\zeta_{k,\ell-1}\zeta_{k+1,\ell})
=\sum\limits_{\alpha=k+1}^{\ell-2}(c_\alpha-1)$. 
The product $\zeta_{i,j-1}\zeta_{i+1,j}$ has leading coefficient equal to $1$, 
and it is not a monomial. Hence, Proposition \ref{prop:weighted_Farey_graph_degree} (1) implies that the coefficient of $q^{\left(\sum\limits_{\alpha=k+1}^{\ell-2}(c_\alpha-1)\right)-1}$ is positive. 
Thus,  
\begin{align*}
\deg\Bigl(\zeta_{k,\ell-1}\zeta_{k+1,\ell}
-q^{\sum\limits_{\alpha=k+1}^{\ell-2}(c_\alpha-1)}\Bigr)
=\left(\sum_{\alpha=k+1}^{\ell-2}(c_\alpha-1)\right)-1\,. 
\end{align*}
By \eqref{eqn:q_CCF_deg1}, it follows that
\begin{align*}
\deg \zeta_{k,\ell}=c_{k+1}-2+\nu^{\tri}_{k+2,\ell-1}\,.
\end{align*}
Since $(k+1,\ell)\in U^{\tri}_{k+1,\ell}$, we have 
$U^{\tri}_{k+1,\ell}=U^{\tri}_{k+1,\ell-1}\cup U^{\tri}_{k+2,\ell} \cup\{(k+1,\ell)\}$. 
Hence,  
%Noting that $U^{\tri}_{k+1,\ell-1}\cap U^{\tri}_{k+2,\ell}=U^{\tri}_{k+2,\ell-1}$, 
we obtain
\begin{align*}
\sum_{u=k+1}^{\ell-1}(c_u-1)-\nu^{\tri}_{k+1,\ell}
&=\sum_{u=k+1}^{\ell-1}(c_u-1)-\Bigl(\nu^{\tri}_{k+1,\ell-1}+\nu^{\tri}_{k+2,\ell}-\nu^{\tri}_{k+2,\ell-1}+1\Bigr)\\
&=\sum_{u=k+1}^{\ell-1}(c_u-1)
-\Bigl(\sum_{\gamma=k+2}^{\ell-1}(c_{\gamma}-1)-\nu^{\tri}_{k+2,\ell-1}+1\Bigr)\\
&=c_{k+1}-2-\nu^{\tri}_{k+2,\ell-1}\,.
\end{align*}
Therefore, 
$\deg \zeta_{k,\ell}=\sum\limits_{u=k+1}^{\ell-1}(c_u-1)-\nu^{\tri}_{k+1,\ell}$. 
\par
\noindent
{\bf Case 2.2}. Suppose $\sigma_{k+1,\ell}\geq2$. 
Then, as in Case 1.2, 
$\sum\limits_{\gamma=k+2}^{\ell-1}(c_{\gamma}-1)
>\nu^{\tri}_{k+1,\ell-1}+\nu^{\tri}_{k+2,\ell}$. 
By \eqref{eqn:q_CCF_deg2}, we have
$\deg(\zeta_{k,\ell-1}\zeta_{k+1,\ell})>\sum\limits_{\alpha=k+1}^{\ell-2}(c_\alpha-1)$. 
Thus, 
\begin{align*}
\deg\Bigl(\zeta_{k,\ell-1}\zeta_{k+1,\ell}-q^{\sum\limits_{\alpha=k+1}^{\ell-2}(c_\alpha-1)}\Bigr)
&=\sum_{\alpha=k+1}^{\ell-2}(c_\alpha-1)
+\sum_{\gamma=k+2}^{\ell-1}(c_\gamma-1)
-\nu^{\tri}_{k+1,\ell-1}-\nu^{\tri}_{k+2,\ell}\,.
\end{align*}
By \eqref{eqn:q_CCF_deg1}, it follows that
\begin{align*}
\deg \zeta_{k,\ell}
&=\sum_{u=k+1}^{\ell-1}(c_u-1)-\nu^{\tri}_{k+1,\ell-1}-\nu^{\tri}_{k+2,\ell}+\nu^{\tri}_{k+2,\ell-1}\,.
\end{align*}
Since $(k+1,\ell)\not\in U^{\tri}_{k+1,\ell}$, we have  
$U^{\tri}_{k+1,\ell}=U^{\tri}_{k+1,\ell-1}\cup U^{\tri}_{k+2,\ell}$. 
Hence, we obtain
\begin{align*}
\sum_{u=k+1}^{\ell-1}(c_u-1)-\nu^{\tri}_{k+1,\ell}
&=\sum_{u=k+1}^{\ell-1}(c_u-1)-\nu^{\tri}_{k+1,\ell-1}-\nu^{\tri}_{k+2,\ell}+\nu^{\tri}_{k+2,\ell-1}\,.
\end{align*}
Therefore, 
$\deg \zeta_{k,\ell}=\sum\limits_{u=k+1}^{\ell-1}(c_u-1)-\nu^{\tri}_{k+1,\ell}$\,. 
\end{proof}
\subsection{Occurrences of 
\texorpdfstring{$1$}{\it 1} in CCFs and descendant degrees 
in associated triangulations}
\begin{prop}\label{prop:count1_desendant} 
The number of occurrences of $1$ in the triangular region of 
a CCF is given by the accumulated descendant degrees of the associated triangulation of an $n$-gon as follows. 
For $(i,j)\in D_{n+1}$, we have 
\begin{align}
\nu^{\tri}_{i,j}=
\begin{cases}
\ell_{i,j}& \text{if $i\leq j\leq i+n-2$}\,, \\
\ell_{i,i+n-1}+1=n-2& \text{if $j=i+n-1$}\,,\\
\ell_{i,i+n}+c_i+1=n-2+c_i& \text{if $j=i+n$}\,,\\
\ell_{i,i+n+1}+c_i+c_{i+1}+1=n-2+c_i+c_{i+1}& \text{if $j=i+n+1$}\,, 
\end{cases}
\end{align}
where $\ell_{i,j}$ extends the quantity defined in \eqref{eqn:descendant_sum} 
of Theorem~\ref{prop:q_CCF_weighted_Farey} by 
$\ell_{i,i+n}=\ell_{i,i+n-1}+\mu_{i,i+n}$ and 
$\ell_{i,i+n+1}=\ell_{i,i+n}+\mu_{i,i+n+1}$. 
Since $\mu_{i,i+n}=\mu_{i,i}=0$ and 
$\mu_{i,i+n+1}=\mu_{i,i+1}=0$, we have 
$\ell_{i,i+n-1}=\ell_{i,i+n}=\ell_{i,i+n+1}$. 
\end{prop}
\begin{proof} Suppose $i\leq j\leq i+n-2$. 
Theorem~\ref{prop:q_CCF_weighted_Farey} gives
$\zeta_{i,j+1}=q^{\ell_{i,j}}s_{i,j+1}$. This identity also remains valid for 
$j=i+n-1$, where $s_{i,i+n}=0$. 
Proposition \ref{prop:weighted_Farey_graph_degree} (2) implies that 
$\deg_{\min}s_{i,j+1}=0$ for $i\leq j\leq i+n-2$. 
Therefore, $\deg_{\min}\zeta_{i,j+1}=\ell_{i,j}$. 
Proposition \ref{prop:q_CCF_min_deg} identifies
$\deg_{\min}\zeta_{i,j+1}$ with $\nu^{\tri}_{i,j}$.
Consequently, $\nu^{\tri}_{i,j}=\ell_{i,j}$. 
\par
Finally, consider the cases $j=i+n-1,i+n,i+n+1$.
By \eqref{eqn:ell_i+n-1} in the proof of 
Theorem~\ref{prop:q_CCF_weighted_Farey}, we have 
$\ell_{i,i+n-1}=\ell_{i,i+n}=\ell_{i,i+n+1}=n-3$. 
Lemma \ref{lemm:count1_trianglation} yields
\begin{align*}
\nu^{\tri}_{i,i+n-1}=\ell_{i,i+n-1}+1,\quad
\nu^{\tri}_{i,i+n}=\ell_{i,i+n}+c_i+1,\quad
\nu^{\tri}_{i,i+n+1}=\ell_{i,i+n+1}+c_i+c_{i+1}+1.
\end{align*}
\end{proof}
\par
Proposition \ref{prop:count1_desendant} 
immediately implies the following relation 
between the occurrences of the entry $1$ and descendant degrees. 
\begin{cor}\label{cor:occ1_descendant_deg} 
For $(i,j)\in D_{n+1}$, we have 
\begin{align*}
\nu^{\tri}_{i,j}-\nu^{\tri}_{i,j-1}=
\begin{cases}
\mu_{i,j}&\text{if $i\leq j\leq i+n-2$},\\
\mu_{i,i+n-1}+1=c_{i-1}&\text{if $j=i+n-1$},\\
\mu_{i,j}+c_j=c_j&\text{if $j=i+n$ or $j=i+n+1$},
\end{cases}
\end{align*}
where we set $\nu^{\tri}_{i,i-1}=0$, equivalently, $U^{\tri}_{i,i-1}=\phi$.
\end{cor}
\begin{proof}
For $i\leq j\leq i+n-2$, the definition \eqref{eqn:descendant_sum} of $\ell_{i,j}$ 
together with Proposition \ref{prop:count1_desendant} gives 
\begin{align*}
\nu^{\tri}_{i,j}-\nu^{\tri}_{i,j-1}=\ell_{i,j}-\ell_{i,j-1}=\mu_{i,j}\,.
\end{align*}  
\par
Next consider the case $j=i+n-1$. Lemma \ref{lemm:count1_trianglation} yields $\nu^{\tri}_{i,j}=n-2$. 
Moreover, Proposition \ref{prop:count1_desendant} and \eqref{eqn:ell_i+n-2} in the proof 
of Theorem~\ref{prop:q_CCF_weighted_Farey} imply that  
$\nu^{\tri}_{i,j-1}=\ell_{i,j-1}=\ell_{i,i+n-2}=n-2-c_{i-1}$. 
Therefore, 
\begin{align*}
\nu^{\tri}_{i,j}-\nu^{\tri}_{i,j-1}=c_{i-1}\,.
\end{align*}
Since $\mu_{i,i+n-1}=\mu_{i,i-1}=c_{i-1}-1$, the result follows. 
\par
Finally, let $j=i+n$ or $j=i+n+1$. 
Proposition \ref{prop:count1_desendant} shows that $\ell_{i,i+n-1}=\ell_{i,i+n}=\ell_{i,i+n+1}$. 
Since $\mu_{i,i+n}=\mu_{i,i}=0$, the conclusion follows immediately. 
\end{proof}
\vskip.2cm
\noindent
Note that, in Appendix \ref{sect:another_proof_cor}, 
we provide alternative proofs 
of Proposition \ref{prop:count1_desendant} 
and Corollary \ref{cor:occ1_descendant_deg}, 
based more directly on the relation between the number of occurrences of the entry $1$ in a CCF and 
the number of diagonals in the associated triangulation. 
\subsection{Degrees of numerator and denominator polynomials in \texorpdfstring{$q$}{\it q}-Farey labelings}
The following proposition determines the minimum and maximum degrees 
of the numerators and denominators in the $q$-Farey labeling.
\begin{prop}
Let $\Farey_q(\T,i)=(\theta_{i,j})_{i-1\leq j\leq i+n-2}$ with 
$\theta_{i,j}=\dfrac{r_{i,j+1}}{s_{i,j+1}}$ be the $q$-Farey labeling of 
the triangulation $(\T,i)$ of an $n$-gon. 
\begin{enumerate}[{(}1{)}] 
\item For the numerators $r_{i,j+1}$ with $i-1\leq j\leq i+n-3$, we have
\begin{align}
&\deg_{\min}r_{i,j+1}
=\nu^{\tri}_{i-1,j}-\nu^{\tri}_{i,j}
=\nu^{\md}_{i-1,j}\,,\nonumber\\
&\deg r_{i,j+1}
=\sum_{u=i}^j(c_u-1)-\nu^{\tri}_{i,j+1}-\nu^{\tri}_{i,j}
=\nu^{\md}_{i-1,j}+\nu^{\rec}_{i,j,i+n-2}\,.
\end{align}
In particular, for $j=i+n-2$, $r_{i,i+n-1}=0$.  
\item For the denominators $s_{i,j+1}$ with $i\leq j\leq i+n-2$, we have
\begin{align}
&\deg_{\min}s_{i,j+1}=0\,,\nonumber\\
&\deg s_{i,j+1}
=\sum_{u=i+1}^j(c_u-1)-\nu^{\tri}_{i+1,j+1}-\nu^{\tri}_{i,j}
=\nu^{\rec}_{i+1,j,i+n-1}\,.
\end{align}
In particular, for $j=i-1$, $s_{j+1}=s_i=0$.   
\end{enumerate}
\end{prop}
\begin{proof} We use the abbreviations $r_{i,j+1}=r_{j+1}$ and $s_{i,j+1}=s_{j+1}$. 
\par
\noindent
(1) Suppose $i-1\leq j\leq i+n-3$. 
By Proposition \ref{prop:q_CCF_min_deg}, we have 
\begin{align*}
&\deg_{\min}\zeta_{i-1,j+1}=\nu^{\tri}_{i-1,j}\,,\quad 
\deg \zeta_{i-1,j+1}=\sum_{u=i}^j(c_u-1)-\nu^{\tri}_{i,j+1}=\nu^{\tri}_{i-1,j}+\nu^{\rec}_{i,j,i+n-2}\,.
\end{align*}
Since $\zeta_{i-1,j+1}=q^{\ell_{i,j}}r_{j+1}$ by Theorem~\ref{prop:q_CCF_weighted_Farey}, 
and $\ell_{i,j}=\nu^{\tri}_{i,j}$ by Proposition \ref{prop:count1_desendant}, it follows that
\begin{align*}
&\deg_{\min}r_{j+1}=\nu^{\tri}_{i-1,j}-\ell_{i,j}=\nu^{\tri}_{i-1,j}-\nu^{\tri}_{i,j}\,,\\
&\deg r_{j+1}=\sum_{u=i}^j(c_u-1)-\nu^{\tri}_{i,j+1}-\ell_{i,j}=\sum_{u=i}^j(c_u-1)-\nu^{\tri}_{i,j+1}-\nu^{\tri}_{i,j}=\nu^{\tri}_{i-1,j}+\nu^{\rec}_{i,j,i+n-2}-\nu^{\tri}_{i,j}\,.
\end{align*}
Since, by Definition \ref{def:occ1}, 
$U^{\tri}_{i-1,j}\backslash U^{\tri}_{i,j}=U^{\md}_{i-1,j}$, that is, 
$\nu^{\tri}_{i-1,j}-\nu^{\tri}_{i,j}=\nu^{\md}_{i-1,j}$, 
the statement (1) holds. 
\par
\noindent
(2) Suppose $i\leq j\leq i+n-2$. 
By Proposition~\ref{prop:weighted_Farey_graph_degree}~(2), 
we have $\deg_{\min}s_{j+1}=0$. 
Moreover, Proposition~\ref{prop:q_CCF_min_deg} gives 
\[
\deg \zeta_{i,j+1}
=\sum_{u=i+1}^j(c_u-1)-\nu^{\tri}_{i+1,j+1}
=\nu^{\tri}_{i,j}+\nu^{\rec}_{i+1,j,i+n-1}.
\]
Since $\zeta_{i,j+1}=q^{\ell_{i,j}}s_{j+1}$ by Theorem~\ref{prop:q_CCF_weighted_Farey}, 
and $\ell_{i,j}=\nu^{\tri}_{i,j}$ by Proposition~\ref{prop:count1_desendant}, it follows that
\begin{align*}
\deg s_{j+1}
&=\sum_{u=i+1}^j(c_u-1)
-\nu^{\tri}_{i+1,j+1}
-\ell_{i,j}=\sum_{u=i+1}^j(c_u-1)
-\nu^{\tri}_{i+1,j+1}
-\nu^{\tri}_{i,j}
=\nu^{\rec}_{i+1,j,i+n-1}.
\end{align*}
This proves~(2). 
\end{proof}

%%% Appendix %%%%%%%%%
\section{Appendix}\label{Appendix}

\subsection{Proof of \eqref{eqn:q-mat_n-1} of Lemma \ref{lemm:n-1_q}}\label{sect:proof_lem}
\begin{proof}
We argue by induction on $n$. 
\par
For $n=3$, 
the quiddity of a triangle is uniquely $(1,1,1)$, and a direct computation gives 
\[
M_q(1,1,1)=-I\,.
\]
\par
For $n=4$, the quiddity is either $(1,2,1,2)$ or $(2,1,2,1)$, and a direct computation shows that 
\[
M_q(1,2,1,2)=M_q(2,1,2,1)=-qI\,.
\]
\par
Assume that the claim holds for $n\leq m$ with $m\geq4$, 
and consider $n=m+1$. 
By Proposition \ref{prop:quiddity} (2), 
there exist at least two non-adjacent indices $\alpha,\beta$ with $i\leq \alpha,\beta\leq i+n-1$ 
such that $c_{\alpha}=c_{\beta}=1$. 
Hence, we may choose $\alpha$ with $i<\alpha<i+n-1$. 
Then the cell ${\mathcal C}=v_{\alpha-1}v_{\alpha}v_{\alpha+1}$ is an exterior cell 
with $c_{\alpha-1},c_{\alpha+1}\geq 2$.  
Let $\T'$ be the triangulation of the $n-1$-gon obtained by removing the exterior cell 
${\mathcal C}$ from the $n$-gon $\T$. Its quiddity is
\begin{align*}
\quid(\T')=
(c_0,c_1,\dots,c_{\alpha-1}-1,\hspace{-.1cm}\raisebox{.3cm}{$\overset{\alpha}{\vee}$}
c_{\alpha+1}-1,c_{\alpha+2},\dots,c_{m})\,.
\end{align*}
By the induction hypothesis, 
\begin{align}\label{eqn:induc_assump}
M_q(c_0,c_1,\dots,c_{\alpha-1}-1,\hspace{-.1cm}\raisebox{.3cm}{$\overset{\alpha}{\vee}$}
c_{\alpha+1}-1,c_{\alpha+2},\dots,c_{m})=
-q^{m-3}I\,.
\end{align}
A straightforward computation shows that for $a,b\geq 2$,  
$M_q(a,1,b)=qM_q(a-1,b-1)$. 
Applying this relation at index $\alpha$, we obtain 
\begin{align*}
M_q(c_0,c_1,\dots,c_{\alpha-1},c_{\alpha},c_{\alpha+1},c_{\alpha+2},\dots,c_{m})
&=M_q(c_0,c_1,\dots,c_{\alpha-1},1,c_{\alpha+1},c_{\alpha+2},\dots,c_{m})\\
&=q
M_q(c_0,c_1,\dots,c_{\alpha-1}-1,\hspace{-.1cm}\raisebox{.3cm}{$\overset{\alpha}{\vee}$}
\hspace{-.1cm}c_{\alpha+1}-1,c_{\alpha+2},\dots,c_{m})\\
&=-q^{m-2}I\,.
\end{align*}
This completes the induction. 
\end{proof}
\subsection{Another example of correspondence between \texorpdfstring{$q$}{\it q}-CCFs and \texorpdfstring{$q$}{\it q}-Farey labelings}
\label{sect:q_CCF_weighted_Farey_2}
\begin{example}\label{ex:q_CCF_weighted_Farey_2}
We consider the triangulation $\T$ of a $9$-gon introduced in Example \ref{ex:triangular}, 
using the edge $v_6v_7$ as the base edge, instead of the choice in Example \ref{ex:q_CCF_Farey_graph1}. 
Figure \ref{fig:q_CCF_Farey_graph2} illustrates the $q$-Farey labeling 
$\Farey_q(\T,7)=(\theta_{7,j})_{7\leq j\leq 15}$.   
%\newpage
\par
\begin{figure}[htbp]
\begin{center}
\begin{tikzpicture}[scale=.8]
\coordinate(v0)at(0,0);
\coordinate(v1)at(-3,0);
\coordinate(v2)at(-5,2);
\coordinate(v3)at(-5.5,4.5);
\coordinate(v4)at(-3.5,6.5);
\coordinate(v5)at(-1,7);
\coordinate(v6)at(1.6,6.2);
\coordinate(v7)at(2.5,4);
\coordinate(v8)at(2.2,2);
\filldraw [black] (v0) circle (1pt) node[below]{\hspace{3.5cm}$v_9\left(
\theta_{7,9}=\dfrac{\langle5\rangle}{[3]}\right), 
\begin{array}{l}
\mu_{7,9}=0\\
\ell_{7,9}=0
\end{array}
$};
\filldraw [black] (v1) circle (1pt) node[below]{
\hspace{-1.3cm}$v_{10}
\left(\theta_{7,10}=\dfrac{[3]}{[2]}\right),
\begin{array}{l}
\mu_{7,10}=1\\
\ell_{7,10}=1
\end{array}
$};
\filldraw [black] (v2) circle (1pt) node[below left]{
$v_{11}
\left(\theta_{7,11}=\dfrac{\langle7\rangle}{\langle5\rangle}\right),
\begin{array}{l}
\mu_{7,11}=0\\
\ell_{7,11}=1
\end{array}\hspace{-.2cm}
$};
\filldraw [black] (v3) circle (1pt) node[left]{
$v_{12}\left(\theta_{7,12}=\dfrac{\langle11\rangle}{\langle8\rangle}\right),
\begin{array}{l}
\mu_{7,12}=0\\
\ell_{7,12}=1
\end{array}\hspace{-.2cm}
$};
\filldraw [black] (v4) circle (1pt) node[above left]{
$v_{13}\left(\theta_{7,13}=\dfrac{[4]}{[3]}\right),
\begin{array}{l}
\mu_{7,13}=2\\
\ell_{7,13}=3
\end{array}
$};
\filldraw [black] (v5) circle (1pt) node[above]{
\raisebox{.5cm}{\hspace{.5cm}$v_{14}\left(\theta_{7,14}=\dfrac{1}{1}\right),
\begin{array}{l}
\mu_{7,14}=3\\
\ell_{7,14}=6
\end{array}
$}};
\filldraw [black] (v6) circle (1pt) node[right]{\hspace{-.1cm}
\raisebox{1cm}{$v_{15}\left(\theta_{7,15}=\dfrac{0}{1}\right),
\begin{array}{l}
\ell_{7,6}=0,\\
\mu_{7,15}=0,\\
\ell_{7,15}=6
\end{array}
$}};
\filldraw [black] (v7) circle (1pt) node[right]{
\raisebox{-1cm}{$v_7
\Biggl(
\begin{array}{l}
\theta_{7,7}=\dfrac{1}{0}\\
\theta_{7,16}=\dfrac{-1}{-0}
\end{array}
\Biggr),
\begin{array}{l}
\mu_{7,7}=0\\
\ell_{7,7}=0
\end{array}$}
};
\filldraw [black] (v8) circle (1pt) node[right]{
\raisebox{-1.5cm}{$v_8\left(\theta_{7,8}=\dfrac{[2]}{1}\right),
\begin{array}{l}
\mu_{7,8}=0\\
\ell_{7,8}=0
\end{array}
$}};
\draw (v0)--(v1) node[midway, above]{\hspace{1cm}$1\hspace{.5cm}\raisebox{.2cm}{\maru{$3$}}$};
\draw (v1)--(v2) node[midway, above]{\raisebox{.4cm}{$q^2\raisebox{1cm}{\maru{$4$}}$}};
\draw (v1)--(v4) node[midway, right]{$q\raisebox{1cm}{\maru{$3$}}$};
\draw (v1)--(v5)node[midway, right]{$1\hspace{.5cm}\maru{$2$}$};
\draw (v1)--(v8)node[midway, above left]{$q$};
\draw (v2)--(v3)node[midway, left]{$q$};
\draw (v2)--(v4)node[midway, left]{$\maru{$5$}\hspace{.2cm}1\hspace{-.1cm}$};
\draw (v3)--(v4)node[midway, above]{$1$};
\draw (v4)--(v5)node[midway, above]{$1$};
\draw (v5)--(v6)node[midway, above]{$1$};
\draw (v5)--(v7)node[midway, above]{$\quad 1\;\maru{$0$}$};
\draw (v5)--(v8)node[midway, right]{$1\;\maru{$1$}$};
\draw (v6)--(v7)node[midway, right]{$q^{-1}$};
\draw (v7)--(v8)node[midway, right]{$q$};
\draw (v8)--(v0)node[midway, right]{$q^2$};
\end{tikzpicture}
\caption{$q$-Farey labeling $\Farey_q(\T,7)$ and descendant degrees 
for the triangulation $(\T,7)$ with quiddity $\quid(\T)=(1,4,2,1,3,4,1,2,3)$}\label{fig:q_CCF_Farey_graph2}
\end{center}
\end{figure}
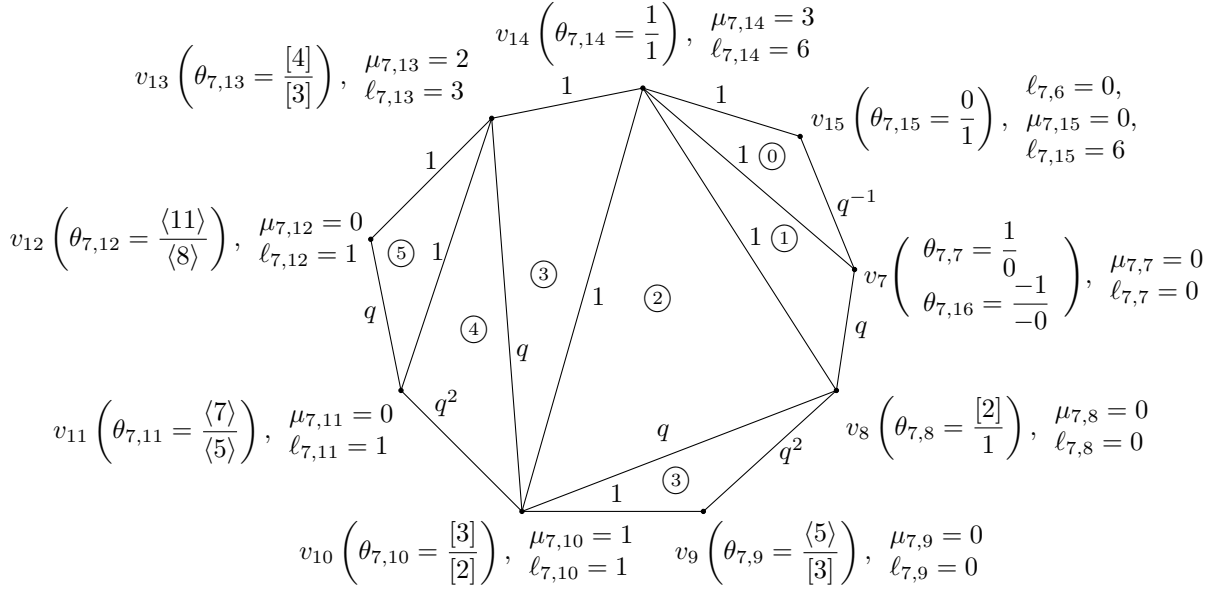
\par
As in Example \ref{ex:q_CCF_Farey_graph1}, 
Proposition \ref{prop:q_CCF_q_rational}, and Theorem \ref{prop:q_CCF_weighted_Farey} yield 
the following correspondence between $q$-continued fractions, 
the entries of the associated $q$-CCF, and the $q$-Farey labeling. 
\begin{align*}
&[[\quad]]_q=\dfrac{\zeta_{7,6}}{\zeta_{8,6}}=\dfrac{q^{\ell_{7,6}}r_7}{q^{\ell_{7,6}}s_7}
=\dfrac{q^0\cdot1}{q^0\cdot0}=\dfrac{1}{0}\\
&[[c_7]]_q=[[2]]_q=\dfrac{\zeta_{7,7}}{\zeta_{8,7}}
=\dfrac{q^{\ell_{7,7}}r_8}{q^{\ell_{7,7}}s_8}
=\dfrac{q^0\cdot[2]}{q^0\cdot1}=\dfrac{[2]}{1}\\
&[[c_7,c_8]]_q=[[2,3]]_q
=\dfrac{\zeta_{7,8}}{\zeta_{8,8}}
=\dfrac{q^{\ell_{7,8}}r_9}{q^{\ell_{7,8}}s_9}
=\dfrac{q^0\cdot\langle5\rangle}{q^0\cdot[3]}=\dfrac{\langle5\rangle}{[3]}\\
&[[c_7,c_8,c_9]]_q=[[2,3,1]]_q=\dfrac{\zeta_{7,9}}{\zeta_{8,9}}
=\dfrac{q^{\ell_{7,9}}r_{10}}{q^{\ell_{7,9}}s_{10}}
=\dfrac{q^0\cdot[3]}{q^0\cdot[2]}=\dfrac{[3]}{[2]}\\
&[[c_7,c_8,c_{9},c_{10}]]_q=[[2,3,1,4]]_q=\dfrac{\zeta_{7,10}}{\zeta_{8,10}}
=\dfrac{q^{\ell_{7,10}}r_{11}}{q^{\ell_{7,10}}s_{11}}
=\dfrac{q\langle7\rangle}{q\langle5\rangle}\\
&[[c_7,c_8,c_9,c_{10},c_{11}]]_q=[[2,3,1,4,2]]_q=\dfrac{\zeta_{7,11}}{\zeta_{8,11}}
=\dfrac{q^{\ell_{7,11}}r_{12}}{q^{\ell_{7,11}}s_{12}}
=\dfrac{q\langle11\rangle}{q\langle8\rangle}\\
&[[c_7,c_8,c_9,c_{10},c_{11},c_{12}]]_q=[[2,3,1,4,2,1]]_q=\dfrac{\zeta_{7,12}}{\zeta_{8,12}}
=\dfrac{q^{\ell_{7,12}}r_{13}}{q^{\ell_{7,12}}s_{13}}
=\dfrac{q[4]}{q[3]}\\
\intertext{}
&[[c_7,c_8,c_9,c_{10},c_{11},c_{12},c_{13}]]_q=[[2,3,1,4,2,1,3]]_q=\dfrac{\zeta_{7,13}}{\zeta_{8,13}}
=\dfrac{q^{\ell_{7,13}}r_{14}}{q^{\ell_{7,13}}s_{14}}
=\dfrac{q^3\cdot1}{q^3\cdot1}=\dfrac{q^3}{q^3}\\
&[[c_7,c_8,c_9,c_{10},c_{11},c_{12},c_{13},c_{14}]]_q=[[2,3,1,4,2,1,3,4]]_q=\dfrac{\zeta_{7,14}}{\zeta_{8,14}}=\dfrac{q^{\ell_{7,14}}r_{15}}{q^{\ell_{7,14}}s_{15}}
=\dfrac{q^6\cdot0}{q^6\cdot1}=\dfrac{0}{q^6}\\
&[[c_7,c_8,c_9,c_{10},c_{11},c_{12},c_{13},c_{14},c_{15}]]_q=[[2,3,1,4,2,1,3,4,1]]_q=\dfrac{\zeta_{7,15}}{\zeta_{8,15}}=\dfrac{q^{\ell_{7,15}}r_{16}}{q^{\ell_{7,15}}s_{16}}
=\dfrac{q^6\cdot(-1)}{q^6\cdot(-0)}=\dfrac{-q^6}{0}
\end{align*}
\par
\end{example}

\subsection{Alternative proofs of Proposition \ref{prop:count1_desendant} and 
Corollary \ref{cor:occ1_descendant_deg}}\label{sect:another_proof_cor}
In the following alternative proofs, we repeatedly use 
Lemma \ref{lemm:count1_trianglation}. 
The statements follow by interpreting $\nu^{\tri}_{i,j}$ as the number of diagonals 
via Lemma \ref{lemm:count1_trianglation}.
\begin{proof}[Alternative proof of Proposition {\rm\ref{prop:count1_desendant}}]
From Definition \ref{def:occ1}, 
the sets $U^{\ad}_{i,u}$ with $i\le u\le j$ are disjoint and 
their union is $U^{\tri}_{i,j}$. Hence 
\[
\nu^{\tri}_{i,j}=\sum_{u=i}^j \nu^{\ad}_{i,u}.
\]
\par
Suppose $i\leq j\leq i+n-2$. 
Using Lemma \ref{lemm:count1_trianglation} and the definition 
of the descendant degree, we obtain
\begin{align*}
\nu^{\tri}_{i,j}
=\sum_{u=i}^j \delta(V[i,u-2],\{v_u\})
=\sum_{u=i}^j \mu_{i,u}
=\ell_{i,j}\,.
\end{align*}
\par
Suppose $j=i+n-1$. Lemma \ref{lemm:count1_trianglation} gives  
$\nu^{\tri}_{i,i+n-1}=\sum\limits_{u=i}^{i+n-2}\delta(V[i,u-2],\{v_u\})+c_{i-1}$. 
The sum $\sum\limits_{u=i}^{i+n-2}\delta(V[i,u-2],\{v_u\})$ counts the diagonals in $T$ that are not incident to $v_{i+n-1}=v_{i-1}$.  
A triangulation of an $n$-gon has $n-3$ diagonals, of which exactly $c_i-1$ are incident to $v_{i-1}$. Hence the sum equals $n-3-(c_i-1)=n-2-c_i$. 
\par
For $j=i+n$ and $j=i+n+1$, the claim follows directly from 
Lemma \ref{lemm:count1_trianglation} by counting the additional contributions as follows: 
\begin{align*}
\nu^{\tri}_{i,i+n}
&=\sum\limits_{u=i}^{i+n-2}\delta(V[i,u-2],\{v_u\})+c_{i-1}+c_i=n-2+c_i\,,\\
\nu^{\tri}_{i,i+n+1}
&=\sum\limits_{u=i}^{i+n-2}\delta(V[i,u-2],\{v_u\})+c_{i-1}+c_i+c_{i+1}=n-2+c_i+c_{i+1}\,.
\end{align*}
\end{proof}

\begin{proof}[Alternative proof of Corollary {\rm\ref{cor:occ1_descendant_deg}}]
By definition, $U^{\tri}_{i,j}\setminus U^{\tri}_{i,j-1}=U^{\ad}_{i,j}$, and hence
\[
\nu^{\tri}_{i,j}-\nu^{\tri}_{i,j-1}=\nu^{\ad}_{i,j}.
\]
Lemma \ref{lemm:count1_trianglation} implies that
\begin{align*}
\nu^{\tri}_{i,j}-\nu^{\tri}_{i,j-1}=\nu^{\ad}_{i,j}
=
\begin{cases}
\delta(V[i,j-2],\{v_j\})&\text{if $i\leq j\leq i+n-2$}\,,\\
c_j&\text{if $i+n-1\leq j\leq i+n+1$}\,. 
\end{cases}
\end{align*}
\par
For $i\leq j\leq i+n-2$, by \eqref{another_def:descendant} again, 
we obtain $\nu^{\tri}_{i,j}-\nu^{\tri}_{i,j-1}=\mu_{i,j}$. 
\par
For $j=i+n-1$, $\mu_{i,i+n-1}=\delta(V[i,i+n-3],\{v_{i+n-1}\})=c_{i-1}-1$. Hence, 
$\nu^{\tri}_{i,i+n-1}-\nu^{\tri}_{i,i+n-2}=c_{i-1}=\mu_{i,i+n-1}+1$. 
\par
For $j=i+n$ and $j=i+n+1$, the claim follows from
$\mu_{i,i+n}=\mu_{i,i}=0$ and $\mu_{i,i+n+1}=\mu_{i,i+1}=0$. 
\end{proof}

%　参考文献
\addcontentsline{toc}{section}{Reference}
\bibliographystyle{plain}
%\bibliography{sample}

\begin{thebibliography}{10}

\bibitem{assem2010friezes}
Ibrahim Assem, Christophe Reutenauer, and David Smith.
\newblock Friezes.
\newblock {\em Advances in Mathematics}, 225(6):3134--3165, 2010.

\bibitem{bapat2023q}
Asilata Bapat, Louis Becker, and Anthony~M Licata.
\newblock $q$-deformed rational numbers and the $2$-{C}alabi--{Y}au category of
  type ${A}_2$.
\newblock {\em Forum of Mathematics, Sigma}, 11:e47, 41 pp., 2023.

\bibitem{ccanakcci2018cluster}
{\.I}lke {\c{C}}anak{\c{c}}{\i} and Ralf Schiffler.
\newblock Cluster algebras and continued fractions.
\newblock {\em Compositio mathematica}, 154(3):565--593, 2018.

\bibitem{conley2023quiddities}
Charles~H Conley and Valentin Ovsienko.
\newblock Quiddities of polygon dissections and the {C}onway--{C}oxeter frieze
  equation.
\newblock {\em ANNALI SCUOLA NORMALE SUPERIORE-CLASSE DI SCIENZE},
  24(4):2125--2170, 2023.

\bibitem{conway1973triangulated}
John~H Conway and Harold~SM Coxeter.
\newblock Triangulated polygons and frieze patterns.
\newblock {\em The Mathematical Gazette}, 57(400):87--94, 1973.

\bibitem{conway1973triangulated_conti}
John~H Conway and Harold~SM Coxeter.
\newblock Triangulated polygons and frieze patterns, continued.
\newblock {\em The Mathematical Gazette}, 57(401):175--183, 1973.

\bibitem{coxeter1971frieze}
Harold~SM Coxeter.
\newblock Frieze patterns.
\newblock {\em Acta Arithmetica}, 18(1):297--310, 1971.

\bibitem{fock1999quantum}
Vladimir~V Fock and Leonid~O Chekhov.
\newblock A quantum {Teichm{\"u}ller} space.
\newblock {\em Theoretical and Mathematical Physics}, 120(3):1245--1259, 1999.

\bibitem{fomin2002cluster}
Sergey Fomin and Andrei Zelevinsky.
\newblock Cluster algebras $\mathrm{I}$: foundations.
\newblock {\em Journal of the American mathematical society}, 15(2):497--529,
  2002.

\bibitem{Fomin2002ClusterAI}
Sergey Fomin and Andrei Zelevinsky.
\newblock Cluster algebras $\mathrm{I}\hspace{-1.2pt} \mathrm{I}$: Finite type
  classification.
\newblock {\em Inventiones mathematicae}, 154:63--121, 2002.

\bibitem{fomin2007cluster}
Sergey Fomin and Andrei Zelevinsky.
\newblock Cluster algebras $\mathrm{I}\hspace{-1.2pt}\mathrm{V}$: coefficients.
\newblock {\em Compositio Mathematica}, 143(1):112--164, 2007.

\bibitem{kogiso2020q}
Takeyoshi Kogiso.
\newblock $q$-deformations and $t$-deformations of {Markov} triples, 2020.
\newblock \texttt{arXiv:2008.12913}.

\bibitem{kogiso2025arithmetic}
Takeyoshi Kogiso, Kengo Miyamoto, Xin Ren, Michihisa Wakui, and Kohji Yanagawa.
\newblock Arithmetic on $q$-deformed rational numbers.
\newblock {\em Arnold Mathematical Journal}, 11(3):42--92, 2025.

\bibitem{kogiso2019bridge}
Takeyoshi Kogiso and Michihisa Wakui.
\newblock A bridge between {Conway--Coxeter} friezes and rational tangles
  through the {Kauffman} bracket polynomials.
\newblock {\em Journal of Knot Theory and Its Ramifications}, 28(14):article
  no.\ 1950083, 40 pp., 2019.

\bibitem{labbe2022q}
S{\'e}bastien Labb{\'e} and M{\'e}lodie Lapointe.
\newblock The $q$-analog of the {Markoff} injectivity conjecture over the
  language of a balanced.
\newblock {\em Combinatorial Theory}, 2(1):article no.\ 9, 25 pp., 2022.

\bibitem{leclere2021q}
Ludivine Leclere and Sophie Morier-Genoud.
\newblock $q$-deformations in the modular group and of the real quadratic
  irrational numbers.
\newblock {\em Advances in Applied Mathematics}, 130:article no.\ 102223, 28
  pp., 2021.

\bibitem{leclere2024radius}
Ludivine Leclere, Sophie Morier-Genoud, Valentin~Yur'evich Ovsienko, and
  Aleksandr~Petrovich Veselov.
\newblock On radius of convergence of $q$-deformed real numbers.
\newblock {\em Moscow Mathematical Journal}, 24(1):1--19, 2024.

\bibitem{lee2019cluster}
Kyungyong Lee and Ralf Schiffler.
\newblock Cluster algebras and {Jones} polynomials.
\newblock {\em Selecta Mathematica}, 25(4):article article no.\ 58, 41 pp.,
  2019.

\bibitem{mcconville2021rank}
Thomas McConville, Bruce~E Sagan, and Clifford Smyth.
\newblock On a rank-unimodality conjecture of {Morier-Genoud and Ovsienko}.
\newblock {\em Discrete Mathematics}, 344(8):article no.\ 112483, 13 pp., 2021.

\bibitem{morier2019farey}
Sophie Morier-Genoud and Valentin Ovsienko.
\newblock {Farey} boat: continued fractions and triangulations, modular group
  and polygon dissections.
\newblock {\em Jahresbericht der Deutschen Mathematiker-Vereinigung},
  121:91--136, 2019.

\bibitem{morier2020q}
Sophie Morier-Genoud and Valentin Ovsienko.
\newblock $q$-deformed rationals and $q$-continued fractions.
\newblock In {\em Forum of Mathematics, Sigma}, volume~8, page e13, 2020.

\bibitem{morier2021quantum}
Sophie Morier-Genoud and Valentin Ovsienko.
\newblock Quantum numbers and $q$-deformed {Conway--Coxeter} friezes.
\newblock {\em The Mathematical Intelligencer}, 43(2):61--70, 2021.

\bibitem{morier2015sl}
Sophie Morier-Genoud, Valentin Ovsienko, and Serge Tabachnikov.
\newblock Sl$_2(\mathbb{Z})$-tilings of the torus, {Coxeter--Conway} friezes
  and {Farey} triangulations.
\newblock {\em L'Enseignement Math{\'e}matique}, 61(1):71--92, 2015.

\bibitem{nagai2020cluster}
Wataru Nagai and Yuji Terashima.
\newblock Cluster variables, ancestral triangles and {Alexander} polynomials.
\newblock {\em Advances in Mathematics}, 363:article no.\ 106965, 37 pp., 2020.

\bibitem{ouguz2025oriented}
Ezgi~Kantarc{\i} O{\u{g}}uz.
\newblock Oriented posets, rank matrices and $q$-deformed {Markov} numbers.
\newblock {\em Discrete Mathematics}, 348(2):article no.\ 114256, 17 pp., 2025.

\bibitem{ouguz2023rank}
Ezgi~Kantarc{\i} O{\u{g}}uz and Mohan Ravichandran.
\newblock Rank polynomials of fence posets are unimodal.
\newblock {\em Discrete Mathematics}, 346(2):article no.\ 113218, 20 pp., 2023.

\bibitem{ovsienko2021towards}
Valentin Ovsienko.
\newblock Towards quantized complex numbers: $q$-deformed gaussian integers and
  the picard group.
\newblock {\em Open Communications in Nonlinear Mathematical Physics},
  1:73--93, 2021.

\bibitem{rabideau2018f}
Michelle Rabideau.
\newblock F-polynomial formula from continued fractions.
\newblock {\em Journal of Algebra}, 509:467--475, 2018.

\bibitem{ren2022radiuses}
Xin Ren.
\newblock On radiuses of convergence of $q$-metallic numbers and related
  $q$-rational numbers.
\newblock {\em Research in Number Theory}, 8(3):article no.\ 37, 14 pp., 2022.

\bibitem{ren2023corrigendum}
Xin Ren.
\newblock Corrigendum to: On radiuses of convergence of $q$-metallic numbers
  and related $q$-rational numbers.
\newblock {\em Research in Number Theory}, 9(2):article no.\ 39, 4 pp., 2023.

\bibitem{series1985modular}
Caroline Series.
\newblock The modular surface and continued fractions.
\newblock {\em Journal of the London Mathematical Society}, 2(1):69--80, 1985.

\end{thebibliography}

\end{document}